\DeclareSymbolFont{AMSb}{U}{msb}{m}{n}
\documentclass[reqno, a4paper, 11pt]{amsart}
\usepackage[T1]{fontenc}
\usepackage[utf8]{inputenc}
\usepackage[foot]{amsaddr}
\usepackage{amsthm,amsfonts,amssymb,amsxtra,appendix,bookmark,euscript,delarray,dsfont,mathrsfs,stackrel,xifthen,mathtools}
\usepackage[normalem]{ulem}
\usepackage[shortlabels]{enumitem}
\usepackage{xcolor}
\usepackage{latexsym,amssymb}
\usepackage{mathrsfs}
\usepackage{longtable}
\usepackage{lmodern}
\usepackage{microtype}
\usepackage{tikz}
\usetikzlibrary{cd}
\usepackage{adjustbox}
\usepackage[bbgreekl]{mathbbol}
\usepackage{verbatim} 
\usepackage{comment}
\usepackage{calc}
\usepackage{centernot}
\usepackage[includeheadfoot,margin=2.54cm]{geometry}
\usepackage[scr=boondox,  
            cal=dutchcal]   
           {mathalpha}
\usepackage[nodayofweek]{datetime}
\usepackage{tikz-cd}
\usepackage{adjustbox}

\hypersetup{hidelinks}

\DeclareSymbolFontAlphabet{\mathbbm}{bbold}
\DeclareSymbolFontAlphabet{\mathbb}{AMSb}%

\DeclareRobustCommand{\SkipTocEntry}[5]{}

\newcommand{\fugcrit}{\varphi_c}
\newcommand{\fug}{\varphi}

\newcommand{\SCE}{\mathsf{SCE}}
\newcommand{\ExgCE}{\mathsf{ExCE}}
\newcommand{\Exg}{\mathsf{Ex}}

\newcommand{\ExgE}{\bN\times \bN_0}

\newcommand{\EhNL}{{\ensuremath{\hat{E}^{N,L}}}}
\newcommand{\pgen}{\mathop{}\!\mathrm{Q}^{N,L}}
\newcommand{\cgen}{\mathop{}\!\mathrm{\mathcal{Q}}^{N,L}}
\newcommand{\equi}{\omega}

\newcommand*{\vhat}[2][0pt]{
	\setbox0=\hbox{$#2$}%
	\widehat{\mathrlap{\phantom{\rule{\wd0}{\ht0+{#1}}}}\smash{#2}}%
}

\DeclareMathOperator{\Div}{div}
\DeclareMathOperator{\gDiv}{\overline{div}}
\newcommand{\DivL}{\mathop{\vhat[-0.5pt]{\Div}{\vphantom{\Div}}^{L}}}

\newcommand{\weakto}{\xrightharpoonup{}}

\makeatletter
\newtheorem*{rep@theorem}{\rep@title}
\newcommand{\newreptheorem}[2]{%
	\newenvironment{rep#1}[1]{%
		\def\rep@title{#2 \ref{##1}}%
		\begin{rep@theorem}}%
		{\end{rep@theorem}}}
\makeatother

\theoremstyle{plain}
\newtheorem{theorem}{Theorem}[section]
\newtheorem{lemma}[theorem]{Lemma}
\newreptheorem{lemma}{Lemma}
\newtheorem*{lemma*}{Lemma}
\newtheorem{corollary}[theorem]{Corollary}
\newtheorem{proposition}[theorem]{Proposition}

\newtheorem*{conjecture*}{Conjecture}
\newtheorem{assumption}[theorem]{Assumption}

\theoremstyle{definition}
\newtheorem{definition}[theorem]{Definition}

\newreptheorem{example}{Example}
\theoremstyle{remark}
\newtheorem{remark}[theorem]{Remark}

\DeclareMathOperator{\AC}{AC}

\DeclareMathOperator{\supp}{supp}

\DeclareMathOperator{\Mom}{M}

\DeclareMathOperator{\tail}{T}
\DeclareMathOperator{\TV}{TV}
\DeclareMathOperator{\ExDiv}{\overline{div}}
\DeclareMathOperator{\ExNabla}{\overline{\nabla}}
\DeclareMathOperator{\VorCel}{VorCell}

\DeclareMathOperator*{\GammaLim}{\Gamma-\lim}

\def\bq{\begin{eqnarray}}
	\def\eq{\end{eqnarray}}
\def\bqq{\begin{eqnarray*}}
	\def\eqq{\end{eqnarray*}}

\def\epsilon{\varepsilon}

\def\net{\mathup{net}}

\newcommand\1{{\ensuremath {\mathds 1} }}
\newcommand\pscal[1]{{\ensuremath{\left\langle #1 \right\rangle}}}

\newcommand{\bsnorm}[2][]{%
	\ifthenelse{\isempty{#1}}%
	{{\ensuremath{|\! |\! |  #2 |\! |\! |_{\beta,s}}}}
	{{\ensuremath{|\! |\! |  #2 |\! |\! |_{\beta,s}^#1}}}
}
\newcommand{\flipL}[1]{{\ensuremath{S_\# {#1^L}^\dagger}}}
\newcommand{\flip}[1]{{\ensuremath{S_\# #1^\dagger}}}

\def\d{{\rm d}}

\def\bC{\mathbb{C}}
\def\bN{\mathbb{N}}
\def\bR{\mathbb{R}}

\def\bJ{\mathbb{J}} 
 
\def\bW{\mathbb{W}}
\def\bV{\mathbb{V}} 

\def\cB{\mathcal{B}}

\def\cD{\mathcal{D}}
\def\cF{\mathcal{F}}

\def\cL {\mathcal{L}}

\def\cE {\mathcal{E}}

\def\N{\mathcal{N}}

\def\cP{\mathcal{P}}
\def\cM{\mathcal{M}}
\def\cR{\mathcal{R}}
\def\cZ{\mathcal{Z}}

\allowdisplaybreaks

\DeclareMathAlphabet{\mathup}{OT1}{\familydefault}{m}{n}
\newcommand{\dx}[1]{\mathop{}\!\mathup{d} #1}

\newcommand{\pderiv}[3][]{\frac{\mathop{}\!\mathup{d}^{#1} #2}{\mathop{}\!\mathup{d} #3^{#1}}}

\newcommand{\eps}{\varepsilon}

\DeclarePairedDelimiter{\abs}{\lvert}{\rvert}
\DeclarePairedDelimiter{\norm}{\lVert}{\rVert}
\DeclarePairedDelimiter{\bra}{(}{)}
\DeclarePairedDelimiter{\pra}{[}{]}
\DeclarePairedDelimiter{\set}{\{}{\}}
\DeclarePairedDelimiter{\skp}{\langle}{\rangle}
\DeclarePairedDelimiter{\ceil}{\lceil}{\rceil}

\makeatletter
\newcommand{\customlabel}[2]{%
   \protected@write \@auxout {}{\string \newlabel {#1}{{#2}{\thepage}{#2}{#1}{}} }%
   \hypertarget{#1}{}
}
\makeatother

\numberwithin{figure}{section}
\numberwithin{equation}{section}

\def\calD{{\mathcal D}} \def\calE{{\mathcal E}}

\def\calP{{\mathcal P}}  \def\calR{{\mathcal R}}

\def\rmA{{\mathrm A}} \def\rmB{{\mathrm B}} 
\def\rmD{{\mathrm D}}

  \def\sfC{{\mathsf C}}
\def\sfD{{\mathsf D}} \def\sfE{{\mathsf E}} 
  
  \def\sfL{{\mathsf L}}
  
  \def\sfR{{\mathsf R}}

\def\scrD{{\mathscr  D}}

\def\bbJ{{\mathbb J}}  
 \def\bbN{{\mathbb N}} 
  \def\bbR{{\mathbb R}}

\title[Variational Convergence of irreversible EDG]{Variational convergence for an irreversible exchange-driven stochastic particle system}

\author{Jasper Hoeksema}
\email{j.hoeksema@tue.nl} 
\address{Department of Mathematics and Computer Science, Eindhoven University of Technology}
\author{Chun Yin Lam}
\email{chun.lam@uni-ulm.de}
\author{André Schlichting}
\email{andre.schlichting@uni-ulm.de}
\address{Institute for Applied Analysis, University of Ulm, Germany}

\thanks{This work is partly funded by the Deutsche Forschungsgemeinschaft (DFG, German Research Foundation) under Germany's Excellence Strategy EXC 2044 --390685587, Mathematics M\"unster: Dynamics--Geometry--Structure.}

\begin{document}
 \begin{abstract}
 We show the variational convergence of an irreversible Markov jump process describing a finite stochastic particle system to the solution of a countable infinite system of deterministic time-inhomogeneous quadratic differential equations known as the exchange-driven growth model, which has two conserved quantities.  As a bounded perturbation of the reversible kernel, the variational formulation is a generalization of the gradient flow formulation of the reversible process and can be interpreted as the large deviation functional of the Markov jump process. As a consequence of the variational convergence result, we show the propagation of chaos of the Markov processes to the limiting equation and the $\Gamma$-convergence of the energy functional. The latter convergence is consistent with related results for reversible coagulation-fragmentation equations and reveals the connection of stochastic processes to the long-time condensation phenomena in the limit equation. 
 \end{abstract}

\subjclass[2020]{Primary 35Q70 Secondary 82C22, 70G75, 60F10.}
 
\maketitle

\vspace{-\baselineskip}
\tableofcontents

\section{Introduction}\label{sec:intro}
We show variational convergence of an irreversible time-inhomogeneous stochastic particle system with exchange dynamics in the thermodynamic limit to the deterministic exchange-driven growth model. Consequently, we show the convergence of the stochastic process to the solutions of its mean field equation and connect the phase transition in the deterministic equation with the stochastic process. 
\subsection{Setting}
The EDG model with (possibly time-dependent) kernel $K$ is the system of ordinary differential equations of cluster size density $c$ describing cluster formation in a closed system: 
\begin{equation*}\label{eq:EDG}\tag{\ensuremath{\mathsf{EDG}}}
	\begin{aligned}
		\dot c_k &=\sum_{l\geq 1} K(l,k-1)c_l c_{k-1} - \sum_{l\geq 1} K(k,l-1) c_{k} c_{l-1}  \\
		&\quad-  \sum_{l\geq 1} K(l,k)c_l c_k + \sum_{l\geq 1} K(k+1,l-1)c_{k+1} c_{l-1}\ ,  
	\end{aligned}\qquad\text{ for } k\geq 0\,.
\end{equation*}
The mass $\sum_{k=0}^\infty c_k $ and the first moment $\sum_{k=0} k c_k$ are preserved along the evolution so we normalize $c \in \cP(\bN_0)$ and consider initial values with finite first moments, that is $c(0)\in \cP^1 :=\{c\in \cP(\bN_0): \sum_{k=0}^\infty k c_k <+\infty \}$. 
    
The exchange-driven growth model is a mathematical model for cluster formation with applications in droplet condensation, polymer formation, population dynamics, and wealth exchange~\cite{BenNaimKrapivsky2003}. 
It is a generalization of the well-studied Becker-Döring model~\cite{BCP86} and can also be viewed as a special version of a coagulation-fragmentation equation similar to the Smoluchowski equations~\cite{Smoluchowski1916}, but with quadratic fragmentation dependence on the density.
Basic mathematical theory of EDG, including well-posedness, was investigated in~\cite{Esenturk2018}. 
This model exhibits a phase transition with the formation of an infinite cluster when the exchange rate given by the kernel~$K$ is growing too fast~\cite{EichenbergSchlichting2021,si2024existence}. In the regime of sublinear growth, where global well-posedness results are obtained, the model exhibits the phenomenon of condensation at infinite time, also called phase separation~\cite{EsenturkValazquez2021, Schlichting2020}. This model has recently been generalized to have clusters exchanging mass of arbitrary size~\cite{barik2024discrete}.
 
The system~\eqref{eq:EDG} can be seen as the reaction rate equation with law of mass action kinetics for the chemical network 
\begin{equation}\label{e:EDG:ChemReact}
	\{k-1\} + \{l\} \xrightleftharpoons[K(k,l-1)]{K(l,k-1)} \{k \} + \{l-1\}\ , \qquad \text{for}\quad k, l \geq 1 \ .
\end{equation}
This motivates a microscopic description given by a Markov jump process, in which each jump can be interpreted as a transfer of a particle between clusters.
We consider the infinitesimal generator on the configuration space $\hat{V}^{N,L}=\set[\big]{ c \in   (L^{-1}\bN_0)^{N+1}: \sum_{k=1}^N c_k =1 ;\sum_{k=1}^N c_k k =\frac{N}{L}}$ given by
\begin{align}\label{eq:def:lifted-generator}
	\cgen_t G(c)=\sum_{(k,l): (c,k,l)\in\hat{E}^{N,L}} \kappa^L_t[c](k,l-1) \hat{\nabla}^L_{k,l-1} G(c) \,,
\end{align}
with the edge set $\hat{E}^{N,L}$ containing all $(c,k,l)$ so that $\{c, c^{k,l-1}\}  \subset \hat{V}^{N,L}$. 
Moreover, the discrete gradient from $c$ along an edge labelled by $(c,k,l-1)$ is given by
\begin{equation}\label{eq:def:hatnabla}
	\hat{\nabla}^L_{k,l-1} G(c) :=L( G(c^{k,l-1})-G(c))
	\ \text{ with }\ 
	c^{k,l-1}:=c+L^{-1}\gamma^{k,l-1}
	\text{ and }
	\gamma^{k,l-1}:= \mathbf{e}_{k-1}+\mathbf{e}_{l} -\mathbf{e}_{k}-\mathbf{e}_{l-1}
\end{equation} 
and the jump rates $\kappa^L:\hat V^{N,L} \to \bR_{\ge 0}^{\ExgE}$ defined via the law of mass action by
\begin{equation}\label{eq:jump_rates}
	{\kappa^{L}_t}^{(\dagger)}[c](k,l-1):= \tfrac{L}{L-1}c_t(k)\bigl(c_t(l-1) -L^{-1} \delta_{k,l-1}\bigr) K^{(\dagger)}_t(k,l-1) \qquad \text{ for } (c,k,l) \in \hat{E}^{N,L} \,.
\end{equation}  From the structure of $\gamma$, it is readily seen that the mass and the first moment of $c$ are conserved under a jump.
Hereby, $K^{\dagger}$ denotes the backward kernel~\eqref{eq:backward-kernel}, $K^\dagger=K$ if and only if the kernel is reversible.  We use the notation $(\dagger)$ to define both quantities simultaneously, the one with and without $\dagger$. 
The evolution of the law of the Markov process $\bC^{N,L}$ is given by the \emph{forward Kolmogorov equation~(FKE)}
\begin{equation*}
\partial_t \bC^{N,L}_t =(\cgen_t)^* \bC^{N,L}_t.
\end{equation*} 
The result of \cite{GrosskinskyJatuviriyapornchai2019} implies the convergence of the stochastic process to the classical EDG. Stochastic methods are used in that work; see also the recent generalization in~\cite{LamSchlichting2025}. In this work, we extend the result to time-dependent kernels and remove the second-moment conditions on the initial data, but with slightly more restrictive growth. To our knowledge, it is the first convergence result on the time-inhomogeneous EDG. 

Our main tool is exploiting the variational structure of the above equations, and we will, under suitable assumptions on~$K$, show that the stochastic particle system of exchange dynamics converges to the solution of~\eqref{eq:EDG} in a variational sense. 

In addition, the variational structure provides insight into the mathematical description of condensation via the formal relation to the large deviation rate functional of both the corresponding Markov jump processes and, when the kernel is reversible, the invariant measures and the gradient flow description.

\subsection{Variational convergence}\label{subsec:VarCon}

 We will now introduce the variational formulations of the processes and their limit via energy-dissipation functionals and show that the corresponding gradient structure is stable under convergence.

The general idea of the variational formulation involves relaxing the evolution equation into a continuity equation with an auxiliary flux and then using Young's inequality to define a functional that is zero for the flux that satisfies the evolution equation. 

This approach to view the evolution equation of jump process as a gradient flow of some functional and to show the convergence of finite particle gradient structure to a mean field equation is well-established, see for instance: \cite{ErbarFathiLaschosSchlichting2016,ErbarFathiSchlichting2020} for mean-field limits toward McKean-Vlasov equations on finite graphs, \cite{Schlichting2019} for the macroscopic limit of the related Becker-Döring model~\cite{BCP86}, \cite{FathiSimon2016} for the hydrodynamic limit of the simple exclusion process and~\cite{MaasMielke2020} for the thermodynamic limit of finite chemical reaction networks. Those works are based on quadratic/Riemannian-like gradient flow structure for Markov chains under detailed balance condition by Maas~\cite{maas2011gradient}, Mielke~\cite{mielke2013geodesic}, and Chow~et.~al.~\cite{Chow2012} following the convergence via variational formulation of gradient flows going back to Sandier and Serfaty~\cite{sandier2004gamma,Serfaty11}. The basic strategy of the proof of the present work follows the blueprint in~\cite{ErbarFathiLaschosSchlichting2016}. 

The equivalent continuity equation formulation for \ref{eq:EDG}  is
\begin{equation*}\label{eq:MF}\tag{\ensuremath{\mathsf{MFE}}}
\partial_t c_t + \gDiv \overline{\jmath}_t \text{ with } \overline{\jmath}_t = \kappa_t[c](k,l-1),\quad  \kappa_t[c](k,l-1):=K_t (k,l-1)c_t(k) c_t(l-1), 
\end{equation*}
with $\gDiv$ being the negative dual  of 
\begin{equation}\label{eq:def:exchangeGradDiv}
\overline{\nabla}_{k,l-1} \Phi = - \Phi_k - \Phi_{l-1} + \Phi_{k-1} + \Phi_{l}.
\end{equation}
The energy dissipation functional with respect to \eqref{eq:MF} is given for $(c,j)\in \mathsf{ExCE}(0,T)$ by
\begin{equation*}
\sfL^{\rho\wedge \rho_c}(c,j):=\left.\sfE^{\rho\wedge \rho_c}(c(t))\right|_{t=0}^T + \int_0^T \bigl( \sfR_t(c , j) +  \sfD_t(c)\bigr) \dx t
\end{equation*}
and set to $\sfL^{\rho\wedge \rho_c}(c,j):=+\infty$ if $(c,j)\notin \ExgCE(0,T)$ (see Section~\ref{sec:varEDG} for the complete definition). Here $\rho\wedge \rho_c:= \min(\rho,\rho_c)$ with $\rho_c$ being the critical first moment of the equilibrium (see Definition~\ref{def:eq-meas} below).

Similarly, the forward Kolmogorov equation for the Markov jump process admits the following continuity equation formulation  
\begin{equation*}\label{eq:FKE}\tag{\ensuremath{\mathsf{FKE}_N}}
\partial_t \bC^{N,L}_t + \DivL \hat \bJ^{N,L}_t  = 0, \quad \text{with the reference net flux } \hat \bJ^{N,L}_t:= \hat \bJ^{N,L}[\bC^{N,L}_t],
\end{equation*} 
\begin{equation}\label{eq:def:FluxN}
	\hat \bJ^{N,L}_t[\bC^{N,L}](c,k,l-1) :=    \nu^L_t[\bC^{N,L}](c,k,l-1) 
\end{equation}
and
\begin{equation}\label{eq:expected-flux-finite}{\nu^L_t}^{(\dagger)}[\bC^{N,L}](c,k,l-1):={\kappa^L_t}^{(\dagger)}[c](k,l-1) \bC^{N,L}_t(c). \end{equation}The graph divergence $\DivL$ is
\begin{equation}\label{eq:def-divL}
\DivL \bJ^{N,L} (c) :=  \!\! \sum_{(k,l): (c,k,l) \in \hat{E}^{N,L}}\!\! L\Bigl(\bJ^{N,L}(c,k,l-1)-\bJ^{N,L}(c^{k,l-1},l,k-1)\Bigr).
\end{equation} 
The corresponding energy dissipation functional is given for any $(\bC^{N,L},\bJ^{N,L}) \in \ref{eq:cgce}(0,T)$
\begin{equation}
	\cL^{N,L}(\bC^{N,L}, \bJ^{N,L}):=
\left.\cE^{N,L}(\bC^{N,L}_t)\right|_{t=0}^T
\!+\int_0^T  \Bigl[  \calR^{N,L}_t(\bC^{N,L},\bJ^{N,L})+\calD^{N,L}_t(\bC^{N,L})  \Bigr]\d t \,,
\end{equation}  
see Section~\ref{sec:var4FKE} and Definition~\ref{def:EDF_FKE} for its complete specification.

As the equations are defined on different spaces, we connect \eqref{eq:FKE} and \eqref{eq:MF} by the Liouville equation and the superposition principle. The limit measure $(\bC,\overline\bJ)$ of $(\bC^{N,L}, \hat{\bJ}^{N,L})$ satisfies the Liouville equation
\begin{equation*}\label{eq:Li}\tag{\ensuremath{\mathsf{LiE}}}
\partial_t \bC_t+\ExDiv^\infty \overline\bJ_t=0, \qquad\text{with}\qquad \overline\bJ_t(\dx c,k,l-1)= \nu_t[\bC] (\d c,k,l-1),
\end{equation*}
where
\begin{equation}\label{eq:expected-flux-in-limit}
    {\nu_t}^{(\dagger)}[\bC] (\d c,k,l)= \bC_t(\d c) {\kappa_t}^{(\dagger)}[c](k,l)
\end{equation} 
and the lifted exchange divergence operator $\ExDiv^\infty$ is the dual operator of the lifted gradient operators $	\ExNabla \! \nabla^\infty$ defined via smooth cylindrical test functions $\Phi: \cP^1 \to \bR$ by 
\begin{equation}\label{eq:def:NablaInf}
	\ExNabla_{k,l-1} \! \nabla^\infty \Phi(c) :=(\partial_{c_{k-1}} -  \partial_{c_{k}} +  \partial_{c_{l}} -  \partial_{c_{l-1}}) \Phi(c),
\end{equation}
\begin{equation}\label{eq:def:GradInf}\text{with the infinite gradient }\nabla^\infty \Phi(c) := \bigl(\partial_{c_k} \Phi(c)\bigr)_{k\in \bN_0}  \end{equation} and $\ExNabla_{k,l-1}$ the exchange gradient defined in~\eqref{eq:def:exchangeGradDiv}.  The superposition principle (Theorem~\ref{thm:super_in_dExg}) projects the energy-dissipation functional $\cL$ (Definition ~\ref{def:EDF4Lie}) of the Liouville equation \eqref{eq:Li} to the energy-dissipation functional of \eqref{eq:MF}: 
\begin{equation}\label{eq:SuperPosEDF}
\cL^{\rho\wedge \rho_c}(\bC,\bJ)= \int_{{C([0,T]; \cP^1)}} \sfL^{\rho\wedge \rho_c}\biggl(\gamma, \frac{\d \bJ}{\d \bC}(\gamma)\biggr) \,\lambda (\d \gamma)  \,,
\end{equation}
where $\lambda$ is a path measure concentrated on curves $\gamma$ such that $\bigl(\gamma,\frac{\d \bJ}{\d \bC}(\gamma)\bigr)\in\ExgCE(0,T)$.

By characterizing the zero point of the energy-dissipation functionals as the solution of the corresponding evolution equations, the variational convergence of the functionals implies the convergence of the solutions of the evolution equations via the superposition principle, see Figure~\ref{fig:scheme}.
\begin{figure}[ht]\centering
\adjustbox{scale=0.95,center}{%
\begin{tikzcd}[row sep=large]
\bC^{N,L}\text{ solves }\eqref{eq:FKE} \iff \cL^{N,L}(\bC^{N,L},\bJ^{N,L})=0  \; \arrow[r, "{\substack{N,L \to \infty \\ N/L \to \rho}}", ""']  &  \;  &[-68pt] \cL^{\rho\wedge \rho_c}(\bC,\bJ) =0 \iff \bC \text{ solves } \eqref{eq:Li} \arrow[d, "\substack{\text{via}\\
	\text{superposition principle}\ }"] \\
   (c_i) \in (\frac{1}{L}\bN_0)^{N+1}  \arrow[u, "\substack{\text{The law of the process}\\ \text{with generator $\cgen$~\eqref{eq:def:lifted-generator}}}"] &  \; &\qquad \sfL^{\rho\wedge \rho_c}(c,j(c))=0 \iff c\text{ solves } \eqref{eq:EDG}.  \qquad
  \end{tikzcd}
  }
\caption{Overview of the variational characterizations and proof strategy.}\label{fig:scheme}
\end{figure} 
	The variational characterizations are contained in Proposition~\ref{prop:zero-set-mean-field-EDP-functional}, Proposition~\ref{prop:EDP-solution-and-finite-FKE-solution} and Proposition~\ref{prop:EDP-solution-Li-solution}, respectively. 
The solution here includes a suitable energy dissipation balance and for this reason is called \emph{EDP solution} for the respective equation. The results also entail that all three energy dissipation functionals are non-negative. The non-negativity of $\sfL^{\rho\wedge \rho_c}$ and $\cL^{N,L}$ are consequences of suitable chain rules proved in Proposition~\ref{prop:chain_rule} and Proposition~\ref{prop:finitechainrule}, whereas the non-negativity of $\cL^{\rho\wedge \rho_c}$ follows by the superposition principle from~\eqref{eq:SuperPosEDF} and the non-negativity of $\sfL^{\rho\wedge \rho_c}$.

\subsection{Assumptions}\label{subsec:ass}
We consider the time-dependent perturbed forward kernels $K$  from a given family of time-independent kernel $\overline{K}$ which satisfies the detailed balance condition. 

\begin{assumption}[Local assumptions on the reversible kernel]\label{as:elective}
The reversible kernel $\overline{K}$ is assumed to have at most linear growth, that is 
\begin{equation*}\label{e:ass:K1}
	0 \leq \overline{K}(k,l-1) \leq C_{\overline{K}} \, k\, l \qquad\text{ for } k,l \geq 1.\tag{$\overline{K}_1$}
\end{equation*} 
A stronger assumption is a strict sublinear growth assumption. 
There exists $m$ sublinear: $\lim_{x\to\infty}\frac{x}{m(x)} = +\infty$ such that
\begin{equation*}\label{e:ass:K2}
	0 \le \overline{K}(k,l-1)\le m(k)m(l)   \qquad\text{ for } k,l \geq 1 \,.
\tag{$\overline{K}_2$}
\end{equation*} 
Furthermore, the uniqueness of solutions to~\eqref{eq:EDG} with $\overline{K}$ is ensured under the assumption (see~\cite[Theorem 1.3]{Schlichting2020})
\begin{equation*}
	\label{e:ass:Ku}
	|\overline{K}(l,k)-\overline{K}(l,k-1)|\le C_{\overline{K}} l 
	\quad\text{ and }\quad
	 |\overline{K}(l+1,k-1)-\overline{K}(l,k-1)|\le C_{\overline{K}} k \tag{$\overline{K}_u$} \qquad\text{ for } k,l \geq 1.
\end{equation*} 
\end{assumption}

\begin{assumption}[Global assumptions on the reversible kernel]\label{as:global}
The kernel $\overline{K}:\bN\times \bN_0\to [0,\infty)$ is assumed to be positive 
\begin{equation*}\label{positive_weight}
\overline{K}(l,0)>0 \quad\text{ and }\quad \overline{K}(1,l-1)>0 \quad \forall l \in \bN \tag{$\overline{K}_{>0}$}.
\end{equation*} 
Moreover, the following limit exists and is positive and finite
\begin{equation*}\label{e:ass:Kc}
	\lim_{k \to \infty} \frac{\overline{K}(k,0)}{\overline{K}(1,k-1)} = \fugcrit \in (0,\infty). \ \tag{$\overline{K}_c$}
\end{equation*}
The kernel satisfies the Becker-Doering assumption
\begin{equation*}\label{ass:BDA}\tag{\ensuremath{\mathsf{BDA}}}
	\frac{\overline{K}(k,l-1)}{\overline{K}(l,k-1)}=\frac{\overline{K}(k,0)\overline{K}(1,l-1)}{\overline{K}(l,0)\overline{K}(1,k-1)} \qquad\text{ for } k,l \geq 1.
\end{equation*}
\end{assumption}
Hereby the condition \eqref{e:ass:Ku} ensures uniqueness~\cite{Schlichting2020}, it is the sublinearity requirement \eqref{e:ass:K1} that ensures stability of \eqref{eq:EDG} and convergence of the particle system under narrow convergence. Moreover, \eqref{e:ass:Kc} implies additional regularity and exponential tails on $w$, which is crucial in proving the chain rule. 
\begin{assumption}\label{e:ass:b}
A time-inhomogeneous kernel $K:[0,\infty)\times \bN \times \bN_0$ is called \emph{admissible} provided that there exists 
a reverse kernel~$\overline{K}$ satisfying Assumptions \ref{as:elective} and \ref{as:global} and a bounded function $b\in L^\infty([0,\infty)\times \bN \times \bN_0)$ satisfying for some $C_b\in (0,\infty)$ the bound
\[  \sup_{t\geq 0} | b_t(l,k) - b_t(l,k-1)|\le C_b k^{-1} \text{ and}\quad \sup_{t\geq 0} | b_t(l+1,k-1) -  b_t(l,k-1) | \le C_b l^{-1} \,, \] 
such that 
\[ K_t(k,l)= \overline{K}(k,l) \exp\bigl(b_t(k,l)\bigr) \,. \] 
\end{assumption}
The assumption of Assumption \ref{e:ass:b} is chosen such that the condition \eqref{e:ass:Ku} holds for $K$. It can be readily checked that an adaptation of the arguments in the classic proof of existence via truncation and uniqueness in~\cite[Theorem 2.12]{Schlichting2019} generalizes to the time-inhomogeneous setting. 
Nevertheless, in Proposition~\ref{prop:EDG-time:Existence} in Appendix~\ref{append:Exist}, we provide a direct proof via the Banach fixed-point theorem for the existence and uniqueness of the time-dependent kernel~$K$. This implies the uniqueness of the solution of~\eqref{eq:EDG} with perturbed kernel~$K$.

The time-inhomogeneous Markov jump process can be rigorously defined via the piecewise deterministic jump process~\cite{davis1993markov}. 
Under the boundedness and measurability assumption of $b$, the existence and uniqueness of the Markov process for the finite particle system can be obtained from the time-homogeneous theory by the time-space process because the corresponding generator is a bounded operator for fixed $(N,L)$.
We note that in \cite{faggionato2022}, this type of perturbation from the reversible kernel has been considered for linear response problems.
\begin{definition}[Backward kernel]
 The backward kernel $K^\dagger$
is defined via 
\begin{equation}\label{eq:backward-kernel}
        K_t(k,l-1)\equi_k \equi_{l-1}= K^{\dagger}_t (l,k-1)\equi_{l}\equi_{k-1} \,  \text{ for all } t\ge0.
\end{equation}
\end{definition}
\begin{remark} 
	As discussed in \cite[after Example 1.6]{Schlichting2020}, the \eqref{ass:BDA} is equivalent to the detailed balance condition 
	\begin{equation}\label{eq:unperturbed-DBC}
		\overline{K}(k,l-1)\equi_k \equi_{l-1}=\overline{K}(l,k-1) \equi_l \equi_{k-1}
	\end{equation}
	for some equilibrium measures $\omega \in \cP^1$ which can be parametrized by the first moment $\rho \in [0,\rho_c]$ (with $\rho<+\infty$ if $\rho_c=+\infty)$) up to the maximum first moment $\rho_c\in [0,\infty]$. 
	However, the equation \eqref{eq:backward-kernel} is invariant under the change of first moment.
	In other words, the perturbed kernel $K$ does not depend on the first moment of the equilibrium measure but only on the kernel $\overline{K}$.
\end{remark} 
More explicitly, the detailed balance condition \eqref{eq:unperturbed-DBC} is satisfied by a one-parameter family of probability measures:
\begin{definition}[Equilibrium cluster distributions]\label{def:eq-meas}
The \emph{equilibrium cluster distributions} of the reversible kernel $\overline{K}$ 
\begin{align}\label{eq:canmeasure}
	\equi^\fug(n)= \frac{1}{Z(\fug)} \fug^n w(n) \quad\text{with}\quad Z(\fug) := \sum_{n= 0}^\infty \fug^n w(n) \quad\text{for any } \fug \in [0,\fugcrit) \,,
\end{align}
where $\fugcrit:=\sup \{ \fug \colon Z(\fug)<+\infty\}\in [0,\infty]$ and
\begin{equation}\label{eq:def:w} 
w(0):=1 , \ 
w(1):=1 , \ 
w(n):= \prod_{k=2}^{n} \frac{\overline{K}(1,k-1)}{\overline{K}(k,0)} \quad \text{ if } 1 \le n\in \bN.  
\end{equation}
The \emph{maximal cluster volume} is given by $\rho_c:= \limsup_{\fug\nearrow \fugcrit} \Mom_1(\equi^\fug)\in [0,\infty]$ with the convention $\rho_c=0$ for $\fugcrit=0$.
Since the first moment map $\fug \mapsto \rho (\fug): = \Mom_1(\omega^\fug)$ is also monotone, 
which implies that for $\rho \in [0,\rho_c]$ there is a unique $\Phi(\rho)\in [0,\fugcrit]$ such that
\begin{align}\label{eq:canmeasure:rho}
	\sum_{k\geq 1} k \equi^{\Phi(\rho)}(k) = \rho \qquad\text{and we write}\qquad \equi^\rho := \equi^{\Phi(\rho)} . 
\end{align}
In other words, the measure $\equi^\rho $ is the equilibrium distribution for the cluster sizes in the exchange-driven growth model~\eqref{eq:EDG} with the initial first moment $\rho $. 
\end{definition}
Clearly, \eqref{e:ass:Kc} implies that
\[\lim_{k\to \infty} \frac{w(k-1)}{w(k)}=\fugcrit,\]
motivating why the parameter $\fug$ ranges from $0$ to $\fugcrit$. This parameter is called \emph{fugacity}, which determines the expected cluster size. 
To exclude the discussion about the \emph{vacuum} state for $\rho_c=0$, the assumption~\eqref{e:ass:Kc} also enforces $\fugcrit>0$, which implies $\rho_c>0$.
\begin{remark}
	The definition of $K^\dagger$ \eqref{eq:backward-kernel} implies 
	\[
		\quad K^{\dagger}_t(k,l)=\overline{K}_t(k,l)\exp(b_t(l+1,k-1)) \quad \forall t\ge0.
	\]
	The definition of $K^\dagger$ is related to the jump rate of the time-reversed process. We refer to \cite[Remark 10.2]{Hoeksema-thesis} for details.
\end{remark}

\subsection{Main results}
Our main result is the Energy Dissipation Principle (EDP) convergence of a Markov jump process with a time-dependent, irreversible quadratic kernel, which is a bounded perturbation from the detail-balanced kernel of the exchange-driven growth model in the thermodynamic limit.  Under this convergence, the structure of the EDP functionals in energy is stable.  To show the EDP convergence, we establish a general $\Gamma$-convergence of relative entropy with product reference measure, which, in particular, covers the equilibrium measure of the reversible kernel $\overline{K}$.

\begin{theorem}[$\Gamma$-convergence of energy]\label{thm:gameconv-energy} 
 	The sequence of entropies $\bigl(
    \bC^{N,L}\mapsto \mathcal{Ent}(\bC^{N,L}\,|\,\mathbbm{\Pi}^{N,L}_\omega) \bigr)_{L}$ is equicoercive in the narrow topology of $\cP(\cP_{\le\bar{\rho}})$ (Definition \ref{def:conv_of_pair_measure}), and
    \begin{equation}\label{eq:thmgammalim}
        \GammaLim_{L\to \infty, N/L\to \rho} \frac{1}{L} \mathcal{Ent}\bigl(\cdot\,|\,\mathbbm{\Pi}^{N,L}_\omega\bigr)\bigl(\bC\bigr)
        = \int  \pra[\Big]{ \mathsf{Ent}\bigl(c\,|\,\equi^{\rho \wedge \rho_c}\bigr) + \bigl( \log \Phi(\rho_c)- \log \Phi(\rho)\bigr)_+ \bigl(\rho-\Mom(c) \bigr)  } \bC(\d c),
    \end{equation} where $\Phi(\rho)$ and $\rho_c$ are defined in  Definition~\ref{def:eq-meas} and we use the convention $+\infty \cdot 0 = 0$. The functions $\mathcal{Ent}, \mathsf{Ent}$ are the relative entropies on the respective spaces.

    Equivalently, the family $(\mathbbm{\Pi}^{N,L}_\omega)_{N,L}$ is exponentially tight 
    and satisfies the large deviation principle with rate $L$ and rate function given in terms of the integrand in~\eqref{eq:thmgammalim}.
\end{theorem}

\begin{remark}
\begin{enumerate}[ (i) ]
  \item In the following applications of the above theorem, we always choose $\omega$ as the equilibrium measure of the reversible part $\overline{K}$. In this case, by Assumption \ref{e:ass:Kc}, $\Phi(\rho_c) \in (0,\infty)$.
  \item For the reversible kernel $\overline{K}$,  the condensation phenomena at equilibrium in the  EDG system were proved in \cite[Theorem 1.7]{Schlichting2020} for $\rho_c<\infty$. The $\Gamma$-convergence result, in particular, connects the condensation phenomena in EDG to the stochastic process. When the first moment $N/L$ along the finite stochastic particle system converges to $\rho>\rho_c$, then the invariant state of the finite particle process converges weakly to $\omega^{\rho_c}$. This loss of the first moment hints at the formation of an infinite cluster at long time. The condensation phenomenon has also been studied in other related stochastic particle systems such as zero-range processes \cite{Grosskinsky-Schutz-Spohn2003,ChlebounGarbielGrosskinsky2022,ChlebounGrosskinsky2014}.
  \item The same functional has been proved as the large deviation rate functional of the empirical measure for the reversible Smoluchowski coagulation and fragmentation equation in~\cite{sun2024condensation}, for the Boltzmann equation and Kac model in~\cite{Basile2024}. Similarly, a generalization to multiple constraints, in the context of extended Sanov's theorem, can be found in~\cite{Nam2020}. In all cases, the proofs work via a modification of G\"artner-Ellis theorem, resulting in a dual representation of the rate function. In contrast, we employ a mixed approach, using both dual representations and explicit recovery sequences in the $\Gamma$-convergence. Moreover, our techniques allow us to relax the assumptions on the equilibrium measure. 
  \item The extra term $\bigl( \log \Phi(\rho_c)- \log \Phi(\rho)\bigr)_+ \bigl(\rho-\Mom(c) \bigr)$ of the limiting rate functional (see right-hand side of~\eqref{eq:thmgammalim}) also occurs as the defect in the relative entropy in the characterization of the longtime limit of~\eqref{eq:EDG} if started from an initial datum with $\rho>\rho_c$, see~\cite[Theorem~1.7]{Schlichting2020}.
\end{enumerate}
\end{remark}

\begin{definition}[Convergence of pair measures]\label{def:conv_of_pair_measure}
	Let $N,L\to \infty$ with $\frac{N}{L} \to \rho \in (0,+\infty)$ and set $\overline{\rho}:=\sup \frac{N}{L}$.
	A pair $(\bC^{N,L}, \bJ^{N,L})\in \ref{eq:cgce}(0,T)$ converges to $(\bC,\bJ)$, provided that
	\begin{alignat*}3
		\bC^{N,L}_t &\weakto \bC_t && \text{ narrowly in } \cP(\cP_{\le \overline{\rho}}) \text{ for each } t\in [0,T],\\
		\bJ^{N,L}_t(\d c,k,l-1)\dx r&\to \bJ_t(\d c,k,l-1)\dx r \quad  && \text{ narrowly in } \cM^+ \bigl(\cP_{\le \overline{\rho}}\times \bN \times \bN_0\times [s,t]  \bigr).
	\end{alignat*} 
 where $\cP_{\le\overline{\rho}}$ is the subset of $\cP^1$ with first moment less than or equal to $\overline{\rho}$ and is equipped with the topology of weak convergence, for each  $0\le s\le t \le T$,
\end{definition}

\begin{theorem}[EDP convergence]\label{thm:EDP_con}
	Let $0<\rho  <+\infty$. The finite particle gradient structure $\bigl(\hat V^{N,L},\hat E^{N,L},\calE^{N,L},\calR^{N,L}\bigr)$ \emph{EDP converges} to the limiting gradient structure $\bigl(\cP_{\leq\overline\rho},E(\cP_{\leq\overline\rho}),\allowbreak\calE^{\rho\wedge \rho_c},\calR\bigr)$, that is 
	\begin{enumerate}[(a)] 
		\item $\Gamma$-convergence of the free energy on $\cP(\cP_{\le \overline{\rho}})$ in the narrow topology
\begin{equation}\label{eq:Gamma-free} 
		\GammaLim_{N/L \to \rho} \cE^{N,L}(\bC^{N,L}) =\cE^{\rho\wedge \rho_c}(\bC) \,.
		\end{equation}
		\item  $\Gamma$-$\liminf$ of dissipation potentials: Any sequence $(\bC^{N,L},\bJ^{N,L})\in \ref{eq:cgce}(0,T)$ converging to $(\bC,\bJ)\in \ref{eq:def:CEinfty}(0,T)$ 
		with
		\begin{equation}\label{eq:uniform potential bound}\begin{split}\sup_{N/L \to \rho}   \int_{0}^{T} \cR^{N,L}_t(\bC^{N,L},\bJ^{N,L}) \dx t<+\infty \,,
		\end{split}\end{equation} 
		satisfies
		\begin{equation}\label{eq:lscPotentials}
			\liminf_{N/L\to \rho} \int_0^T  \Bigl( \cR^{N,L}_t(\bC^{N,L},\bJ^{N,L})+ \cD^{N,L}_t (\bC^{N,L})\Bigr) \dx t 
                \ge \int_0^T \Bigl(\cR_t(\bC,\bJ)+ \cD_t (\bC) \Bigr) \dx t \,
		\end{equation}  where $\cR,\cD$ are the dissipation potentials of the energy-dissipation functional of the limit system.
		In particular, the limit flux satisfies $\bJ_t\ll\mathbbm{\Theta}_t [\bC_t]$ for each $k,l-1$ and almost all $t\in[0,T]$.
		\item Under assumption \eqref{e:ass:K2}, let $((\bC^{N,L},\bJ^{N,L})\in \ref{eq:cgce}(0,T))_{\frac{N}{L}\to \rho}$ and \eqref{eq:uniform potential bound} hold. Then 
		\begin{enumerate}[(i)]
			\item $(\bC^{N,L},\bJ^{N,L})\in \ref{eq:cgce}(0,T)$ converges to $(\bC,\bJ)\in \ref{eq:def:CEinfty}(0,T)$ along a subsequence.
			\item For each $t\in[0,T]$, $\bC_t$ is concentrated on $\cP_{\le \!\rho}$.
		\end{enumerate}
		
	\end{enumerate}
\end{theorem}

\begin{theorem}[$\Gamma$-lower-semicontinuity of energy-dissipation functional]\label{thm:EDP_con_EDP}
Let $(\bC^{N,L},\bJ^{N,L})\in \ref{eq:cgce}(0,T)$ be a sequence converging to $(\bC,\bJ)\in \ref{eq:def:CEinfty}(0,T)$ and satisfy
\begin{enumerate}[(a)]
\item well-preparedness of initial data:
\begin{equation}\label{eq:well-pose-initial-data}
\lim_{N/L\to \rho} \cE^{N,L}(\bC^{N,L}_0) =\cE^{\rho\wedge \rho_c}(\bC_0) <+\infty \,,
\end{equation} 
\item uniformly bounded energy dissipation functional:
\[
	\sup_{N/L \to \rho} \cL^{N,L}(\bC^{N,L},\bJ^{N,L})<+\infty \,;
\] 
\item EDP solution property:
\[
	\liminf_{N/L\to \rho} \cL^{N,L}(\bC^{N,L},\bJ^{N,L})\le 0 \,;
\]
\end{enumerate}
Then it holds
\[
	\cL^{\rho\wedge\rho_c}(\bC,\bJ)\le 0 \,.
\]
\end{theorem}
Recall that by an EDP solution, we mean a pair of a curve and a flux such that the associated energy dissipation functional is non-positive.
By the variational characterization of evolution equations, they are the weak solutions of the corresponding evolution equation. 
As a consequence of the theorem above, we obtain the convergence of solutions in law.
\begin{corollary}[Convergence of EDP solutions]\label{cor:sol2sol}
	Under assumption  \eqref{e:ass:K2}, consider a sequence  $(\bC^{N,L},\bJ^{N,L})\in\ref{eq:cgce}(0,T)$ in the limit $\frac{N}{L}\to \rho\in (0,\infty)$.
	Suppose
	\begin{enumerate}[(a)]
		\item $(\bC^{N,L},\bJ^{N,L})$ is an EDP solution of~\eqref{eq:FKE} (see Proposition \ref{prop:EDP-solution-and-finite-FKE-solution});
		\item $\bC^{N,L}_0 \to \bC_0$ narrowly and it is well-prepared satisfying~\eqref{eq:well-pose-initial-data}.
	\end{enumerate}
	Then  
	\begin{enumerate}[(i)]
		\item there exists a convergent subsequence  with the  limit $(\bC,\bJ)\in \ref{eq:def:CEinfty}(0,T)$  and it is an EDP solution of~\eqref{eq:Li} (see Proposition~\ref{prop:EDP-solution-Li-solution});
		\item there exists a path measure $\lambda$ which is concentrated on EDP solutions to~\eqref{eq:MF} (see Proposition \ref{prop:zero-set-mean-field-EDP-functional}) and $(e_t)_\#\lambda = \bC_t$ for each $t$;
		\item the energy converges 
		 \[
		 	\lim_{N/L\to \rho} \cE^{N,L}(\bC^{N,L}_t) = \cE^{\rho\wedge \rho_c}( \bC_t ) \qquad \text{for each } t\in [0,T] \, .
		 \]
	\end{enumerate}
\end{corollary}
As a consequence of EDP convergence, we have the convergence of the stochastic process to the deterministic solution pointwise in time, given the convergence of the initial conditions. 
\begin{theorem}[Propagation of chaos]\label{thm:PropaChaos}
Under assumptions \eqref{e:ass:K2} and \eqref{e:ass:Ku}, suppose 
\begin{enumerate}[(a)]
\item $c_0\in \cP_{\!\le\rho}$ with $\sfE^{\rho\wedge \rho_c}(c_0)<+\infty$,
\item $(\bC,\bJ)\in \ref{eq:def:CEinfty}(0,T)$ is the unique EDP solution of~\eqref{eq:Li} with deterministic initial data $\bC_0=\delta_{c_0}$,
\item $c_t$ is the unique solution to~\eqref{eq:MF} with initial data $c_0$,
\item $(\bC^{N,L},\bJ^{N,L})$ satisfies the assumptions of Corollary \ref{cor:sol2sol}.
\end{enumerate}
Then, the limit is deterministic for all $t\in [0,T]$, that is 
\[
	\bC^{N,L}_t \to \bC_t = \delta_{c_t} \text{ narrowly }  \quad\text{ and }\quad \lim_{N/L\to\rho }\cE^{N,L}(\bC^{N,L}_t) = \cE^{\rho\wedge \rho_c}(\bC_t).
\] 
\end{theorem}
We observe the loss of the first moment along the converging processes for a certain class of initial conditions.
\begin{corollary}[Loss of first moment at initial time]\label{cor:loss-1-mom} 
	In the setting of Theorem \ref{thm:PropaChaos} assuming $\rho_c<+\infty$ and let $N/L\to\rho>\rho_c$, then there exists $\bC^{N,L}_0$ converging to $\bC_0=\delta_{c_0}$ for some $c_0\in\cP_{\rho_c}$ and $\Mom(c_t)=\Mom(c_0)=\rho_c$ for all $t\in [0,T]$.
\end{corollary}

\subsection{Relations to previous works and implications for future works}

\subsubsection*{Variational convergence and Large Deviation Principle}
One novelty of this work is the variational convergence for an irreversible system. Yet, as a perturbation of the reversible kernel, our approach of the EDP convergence and the variational formulations still stems from the convergence of the gradient flow structures for Markov processes under the detailed balance conditions.

A natural starting point is to use the $\cosh$ gradient flow structure in continuity equation format for the Markov jump process under a detailed balance condition as in ~\cite{MielkePeletierRenger2014GFLDP,LieroMielkePeletierRenger2017,KaiserJackZimmer2019,PRST22}.  This is a natural structure in the sense that it has a microscopic origin--- the variational formulation coincides with the large deviation rate functional of the Markov jump process. The structure of \emph{energy dissipation principle (EDP)} convergence was formulated in ~\cite{Mielke2016,LieroMielkePeletierRenger2017,MielkeMontefuscoPeletier2021,PeletierSchlichting2022}.
The rigorous definitions of the energy-dissipation (ED) functionals for the variational formulation of gradient structures for linear jump processes are extensively discussed in~\cite{PRST22}.  

Beyond reversible kernels, however, the $\cosh$ structure for the antisymmetric net fluxes is no longer possible; instead, the exponential structure for one-way fluxes is a natural generalization. Their variational formulations are related by a contraction principle in large deviation theory. In particular, the convergence using one-way fluxes formulation implies that of the net fluxes \cite{Renger2018}. The irreversible evolution is no longer a gradient flow, but the large deviation interpretation of the variational formulation is still valid. For example, in \cite{KaiserJackZimmer2018, PattersonRengerSharma2024, Renger2018}, they proposed an orthogonal decomposition of the LDP rate functionals via forces with net fluxes without detailed balance conditions. In this decomposition, the symmetric part of the force is the driving functional for the gradient flow, which is connected to macroscopic fluctuations theory \cite{BertiniDeSoleGabrielleJonaLasinioLandim2015} and can be viewed as a generalization of gradient flow and GENERIC structures.

Here we use the decomposition of the rate functions with one-way fluxes first described in the work by one of the authors \cite{Hoeksema-thesis, HoeksemaTse2023}, which allows us to treat irreversible jump processes with  a bounded, possibly time dependent perturbation of the reversible kernel. With this variational structure, the EDP convergence for a version of the  Bolker–Pacala and Dieckmann-Law (BPDL) model was proven.  Under the assumption of weak detailed balance and the reference measure being invariant in the driving functional, it is shown in \cite[Chapter 10.5.3]{Hoeksema-thesis} that this formulation can be viewed as a decomposition of symmetric and antisymmetric forces. 

This interpretation does not hold in our case because we do not have the weak detailed balance assumption, and more critically, our reference measure satisfies the detailed balance for the unperturbed kernel, which is not an invariant measure in the perturbed kernel. Nevertheless, it can be observed that the additional cost due to the non-invariant reference measure arises from the $\cD^{N,L}/\sfD$ terms in the rate functionals, with the precise physical interpretation not directly clear. Since this term has no sign, it is a non-dissipative term, and the driving functional is not a Lyapunov functional of evolution.  

There is also an intimate connection to flux large deviation principles, also called level $2.5$ as established in~\cite{Renger2018,BertiniFaggionatoGabrielli2014,PattersonRenger2019,AgazziAndreisPattersonRenger2022}. 
 We expect that it is possible to extend the EDP convergence statement to a large deviation principle for the system along the lines of~\cite{AdamsDirrPeletierZimmer2011,AdamsDirrPeletierZimmer2013,Fathi2016,Mariani2018,KaiserJackZimmer2018}.

\subsubsection*{Kernels with linear growth}

By the analysis of the variational functional, under the sublinear kernel assumption, we observe a possible loss of mass at the initial time of the approximating stochastic processes but conservation of the first moment in time. The kernel's sublinear growth is necessary to conserve the first moment for the process and pass to the limit.  It is known that the Marcus-Lushnikov process converges to the Smoluchowski equation with sublinear kernel and to the Flory equation with kernel of linear growth \cite{fournier2009marcus}. One might as well expect similar interactions in EDG. This leads to our conjecture for the kernel with linear growth:   the processes converge to a relaxation of the EDG, in which there is interaction between the finite cluster and the infinite cluster with mass $\Mom(c_0)- \Mom(c)$.
Given the validity of this conjecture, our result would be optimal.
\subsection{Structure of this paper}
In Chapter 2, we introduce the metric well-adapted to the evolution of EDG. This leads to a notion of a continuity equation for exchange dynamics.

In Chapters 3-5, we prove the variational characterization of the evolution equations via the continuity equation in the mean-field equation, Markov process, and the Liouville equation levels.

In Sections 6.1-6.3, we show the compactness of the sequence of processes and the convergence to the continuity equation. 
In Sections 6.4-6.5, we show the $\Gamma$-convergence of the energy functional and the lower semicontinuity of the dissipation potentials.
In Section 6.6, we provide proof of the main theorems given the functional convergences.

\section{The exchange metric}\label{sec:metric}
We introduce the topology induced by the following exchange metric for the infinite-dimensional state space~$\cP^1$. This will give a notion of continuity in time for the exchange dynamics.
\begin{definition}[Exchange metric]
The \emph{tail distribution} $\tail:\cP^1\to \ell^1(\bN)$ for some $\mu\in \cP^1$ is defined by
\begin{equation*}
	\tail(\mu)=(\tail_k(\mu))_{k\in\bN}\qquad\text{with}\qquad \tail_k(\mu) := \sum_{n= k}^\infty \mu_n \,.
\end{equation*}
The \emph{exchange metric} on $\cP^1$ is defined for $\mu,\nu \in \cP^1$ via
\begin{equation}\label{eqdef:dEx}
\d_{\Exg}(\mu,\nu)= \lVert \tail(\mu) - \tail(\nu) \rVert_{\ell^1(\bbN)} .
\end{equation}
\end{definition}
Thanks to the Kantorovich-Rubinstein duality~\cite{KantorovichRubinstein1958}, we can identify the exchange metric as the Wasserstein distance.
\begin{proposition}\label{prop:dex-is-W1}
The metric $\d_{\Exg}$ on $\cP^1$ is the 1-Wasserstein distance of probability measures on $\bN_0$ with the Euclidean  metric on $\bR$ restricted to $\bN_0$. In particular, 
\begin{equation}
    \d_{\Exg}(\mu,\nu)= \sup \set[\bigg]{\abs[\bigg]{\sum_{k\geq 0} f_k (\mu_k -\nu_k)}: f \, \text{ is } 1\mbox{-Lipschitz} }.
\end{equation}

\end{proposition}
\begin{proof}
Let $\gamma \in \Gamma(\mu,\nu)$ be any coupling between $\mu$, $\nu$. We estimate
\begin{align*}
\d_{\Exg}(\mu,\nu)&=\|\tail(\mu) -\tail(\nu)\|_{\ell^1(\bN)}=\sum_{i=1}^\infty \biggl|\sum_{k=i}^\infty \mu_k - \sum_{k=i}^\infty \nu_k \biggr| \\
&=\sum_{i=1}^\infty \biggl|\sum_{k=1}^\infty \1_{k\ge i}\mu_k - \1_{k\ge i}\nu_k \biggr|
=\sum_{i=1}^\infty \biggl| \sum_{k,l=1}^\infty (\1_{k\ge i} - \1_{l\ge i}) \gamma(k,l) \biggr| \\
&\le  \sum_{k,l=1}^\infty\sum_{i=1}^\infty \left|\1_{k\ge i} - \1_{l\ge i}\right| \gamma(k,l)\\
&\le    \sum_{k,l=1}^\infty |k-l| \gamma(k,l).
\end{align*}
Now we take the infimum over all couplings to get 
\begin{equation}\label{eq:dex upbd}\d_{\Exg}(\mu,\nu) \le  \inf_{\gamma\in \Gamma(\mu,\nu)} \sum_{k,l=1}^\infty |k-l| \gamma(k,l)=W_1(\mu,\nu),\end{equation} for the metric space $(\bN_0, \|\cdot \|_{\bR})$.
On the other hand, let $f\in \bR^{\bN_0}$ such that $|f(i)-f(i-1)|\le 1$, $f(-1)$ be arbitrary and  $\mu,\nu \in \cP(\bN_0)$, we have 
\begin{align*}
\d_{\Exg}(\mu,\nu)&=\sum_{k=0}^\infty \lvert\tail_k(\mu-\nu)\rvert\ge \sum_{k=0}^\infty (f(k)-f(k-1)) \tail_k(\mu-\nu)\\
&=\sum_{k=0}^\infty f(k) (\tail_{k}(\mu-\nu) - \tail_{k+1}(\mu-\nu)) - f(-1) \tail_0(\mu-\nu)\\
&=\sum_{k=0}^\infty f(k) (\mu_{k} - \nu_{k}).
\end{align*}
Taking supremum over all 1-Lipschitz $f$, we have 
\begin{equation}\label{eq:dex lowbd}\sup_{f\in 1\text{-Lip}}\sum_{k=0}^\infty f(k) (\mu_k -\nu_k) \le  \d_{\Exg}(\mu,\nu).\end{equation}
By the Kantorovich-Rubinstein theorem, we have the equality between 
 \[\sup_{f \in 1\text{-Lip}}\sum_{k=0}^\infty f(k) (\mu_k -\nu_k) = \inf_{\gamma\in \Gamma(\mu,\nu)} \sum_{k,l=1}^\infty |k-l| \gamma(k,l).\]
Hence together with \eqref{eq:dex lowbd} and \eqref{eq:dex upbd}, we conclude $\d_{\Exg}(\mu,\nu)= W_1(\mu,\nu)$. \end{proof}
From standard properties of the Monge-Kantorovich distance~\cite{Rachev1991}, we conclude the separability and completeness of the metric space $(\cP^1,\d_\Exg)$.
\begin{corollary}\label{prop:d_separable}
The metric space $(\cP^1,\d_\Exg)$ is separable and complete.
\end{corollary}
 
In the following, we will see the connection of the exchange metric $\d_\Exg$ to the particle exchange via the characterisation with functions with bounded exchange gradient, which are crucial for stating continuity properties for the exchange continuity equation~Definition \ref{def:ExCE}.
\begin{definition}[Dual metric]\label{def:Xmetric}
The set of functions with the exchange gradient bounded by one is
\[
	\cF:=\set[\bigg]{ f:\bN_0 \to \bR \colon \sup_{k,l\geq 1}|\ExNabla_{k,l-1} f | \le 1 }, 
\] 
and measures summable with respect to the family $\cF$ are
\[
	\cP_\cF:=\set[\bigg]{\mu \in \cP(\bN_0)\colon \sum_{k\geq 0} |f_k| \mu_k < \infty\quad \forall f \in \cF }. 
\]
The dual metric on $\cP_\cF$ is defined as  
\begin{equation*}
 \d_\cF(\mu,\nu):=\| \mu -\nu\|_{\cF}^* := \sup \set[\bigg]{\abs[\bigg]{\sum_{k\geq 0} f_k (\mu_k -\nu_k)}: f \in \cF }.\end{equation*} 
\end{definition}
\begin{remark}
\begin{enumerate}
\item By definition, $\mu^n \to \mu$ in $\| \cdot \|_{\cF}^*$ if and only if $\pscal{f,\mu^n} \to \pscal{f,\mu}$ for $f\in\cF$ uniformly.
\item Any element $f\in\cF$ gives rise to a linear functional on $\cP_{\cF} \to \bR$ by setting $f(\mu)= \pscal{f,\mu} = \sum_{k=0}^\infty f_k \mu_k$, 
      which is $1$-$\d_\cF$~Lipschitz by definition. 
\end{enumerate}
\end{remark}
\begin{proposition}[Connection between $\d_\cF$ and $\d_{\Exg}$]\label{prop:closedformdcF}
The class of function $\cF$ is characterized by the sum of a linear part and a $1/2$-Lipschitz function. In particular, $f \in \cF$ has at most linear growth. Furthermore,  the dual metric $\d_\cF$ is an extended metric on $\cP^1$ and a metric precisely on any $\cP_{\!\rho}$ for $\rho\in [0,\infty)$, where it agrees with $\d_{\Exg}$ up to a factor of $1/2$, that is if $\mu,\nu \in \cP^1$, then
\[
	\d_\cF(\mu,\nu)=
		\begin{cases}
			\frac12  \d_{\Exg}(\mu,\nu)  &\text{ if } \Mom_1(\mu)=\Mom_1(\nu); \\
			+ \infty &\text{ otherwise.}
		\end{cases}
\]
\end{proposition}
\begin{proof} 
 Notice that $f\in \cF$ is equivalent to the variation of $(\partial^-_l f)_{l\in \bN_0}$ to be bounded by one where 
\begin{equation}\label{def:discrete-derivative}
(\partial^{-}f)_l:= f_l -f_{l-1} \quad \text{ for } l \in \bN.
\end{equation}
Let  $m=m(f):= \frac12 (\inf_l \partial^-_l f + \sup_l \partial^-_l f)$ then it holds \[m-\frac12\le \inf_l \partial^-_l f \le m \le \sup_{l} \partial^-_l f \le m+\frac12.\] 
It follows that $g_l=f_l - m l$ is $1/2$-Lipschitz. Conversely, for any $g$ that is $1/2$-Lipschitz and   $m \in \bR$, the sequence $(f_l := g_l + m l)_{l\in\bN_0}$ has its exchange gradient bounded by $1$. Hence, for every $f = \cF$, $f_l$ can be decomposed into a sum of a linear function and a $\frac12$-Lipschitz function. 
In the case when $\mu,\nu \in \cP_{\!\rho}$ for some $\rho\in [0,\infty)$, we have for $f\in \cF$, 
\[\begin{split}\sum_{k\in\bN_0} f_k(\mu_k -\nu_k) &=\sum_{k\in\bN_0} g_k(\mu_k -\nu_k) + m \sum_{k\in\bN_0} k(\mu_k -\nu_k) 
=\sum_{k\in\bN_0} g_k(\mu_k -\nu_k).  
\end{split}\]
Above we used $\sum_{n=0}^\infty n (\mu_n - \nu_n)=0$. 
  Therefore, we obtain the identity by Proposition \ref{prop:dex-is-W1}
\begin{equation*}\begin{split}
  \sup_{ f \in \cF } \set[\bigg]{\abs[\bigg]{\sum_{k\in\bN_0} f_k (\mu_k -\nu_k)}} = \sup_{g: 1/2\text{-Lipschitz} } \set[\bigg]{\abs[\bigg]{\sum_{k\in\bN_0} g_k (\mu_k -\nu_k)}}= 
  \frac12 \d_{\Exg}(\mu,\nu)
\end{split}\end{equation*}
which is the claim when $\Mom_1(\mu)=\rho=\Mom_1(\nu)$.
 
In the case $\mu$ and $\nu$ have different first moments, we need to show 
$\d_\cF(\mu,\nu)=+\infty$.
Indeed, we take the test function in the definition of $\d_\cF$ to be $f_k=\lambda k$ for $\lambda\in \bR$. Then
\[
	\d_{\cF}(\mu,\nu) \ge \biggl|\sum_{k\in\bN} \lambda k (\mu_k - \nu_k)\biggr|=\lvert\lambda\rvert \bigl\lvert\Mom_1(\mu) - \Mom_1(\nu)\bigr\rvert \,,
\]
which diverges to $+\infty$ as $\lambda \to \pm \infty$.  
\end{proof}
\begin{remark}
It follows from Proposition \ref{prop:closedformdcF} that the metric $\d_\cF$ is stronger than the metric induced by $\ell^1$ on $\cP^1$. Indeed, we can take  $f_k = \frac14 \operatorname{sign} (\mu_k - \nu_k)$ so that $(f_k)_{k\in\bN_0}\in \cF$  and thus  
 \[\frac14 \sum_{k\ge 0} |\mu_k - \nu_k |\le \d_{\cF}(\mu, \nu)\] for $\mu, \nu \in \cP^1$. 
Therefore, on probability measures with a fixed first moment, the topology induced by $\d_\Exg$ is between the metric induced by the pure $\ell^1$ norm and the first-moment-weighted $\ell^1$ norm, denoted by $\ell^{1,1}$, \begin{equation}\label{def:ell1,1} 
	\| \mu\|_{\ell^{1,1}}:= \sum_{k=0}^\infty (1+k) \mu_k. 
\end{equation}
\end{remark}
 The connection to the microscopic dynamic induced through the jumps~\eqref{eq:def:exchangeGradDiv} becomes apparent through the representation in the next Lemma. 
\begin{lemma}[{Variational form of $\d_{\Exg}$  in terms of jumps}]\label{lem:varform_d_Exg}
Suppose $\mu,\nu \in\cP_{\!\rho}$ for some $\rho\in [0,\infty)$ then 
\begin{equation*}
\d_{\Exg}(\mu,\nu)= 2 \sup_{f \in  \cF}\set[\bigg]{\sum_{k,l\in\bN}\alpha_{k,l-1}\ExNabla_{k,l-1} f  \,\bigg|\, \alpha \in \ell^{1}(\bN\times \bN_0) : \mu= \nu + \sum_{k,l\in\bN} \alpha_{k,l-1} \gamma^{k,l-1}  }. 
\end{equation*}
As a consequence, if $\mu$ and $\nu$ is connected by a single exchange jump, 
that is $\mu=\nu + m \gamma^{k,l-1}$ for some $k,l\in \bN$ and $m\in \bR$, then $\d_{\Exg}(\mu,\nu)=2|m|$.
\end{lemma}
\begin{proof}
By Proposition \ref{prop:closedformdcF}, we have 
\[
	\d_\Exg(\mu,\nu)=2 \d_{\cF}(\mu,\nu)=2 \sup_{f \in \cF}  \abs[\Bigg]{\sum_{k\geq 0} f_k (\mu_k -\nu_k)}\,.
\]
Let $\alpha\in \ell^1(\bN\times \bN_0)$ be such that $\mu=\nu +\sum_{k,l\in\bN} \alpha_{k,l-1}\gamma^{k,l-1}$. The definition of the exchange gradient~\eqref{eq:def:exchangeGradDiv} implies
\[
	\d_{\Exg}(\mu,\nu)=2 \sup_{f \in \cF} \abs[\Bigg]{\sum_{n\geq 0}f_n  \biggl[\sum_{k,l\in\bN} \alpha_{k,l-1}  \gamma^{k,l-1}_n\biggr]}  =2 \sup_{f \in \cF}\abs[\Bigg]{ \sum_{k,l\in\bN} \alpha_{k,l-1} \ExNabla_{k,l-1}f } \, .
\]
For $f\in \cF$, we can apply the Fubini theorem to exchange the summation order if $\sum_{k,l\in\bN} |\alpha_{k,l-1}|<+\infty$. We conclude the proof by noting that $\cF$ is invariant under $f \mapsto -f$ so we can remove the absolute value. 

Now suppose $\mu-\nu = m \gamma^{k,l-1}$ for some $k,l\in\bN$, since the unit vectors $\mathbf{e}_l$ and $\mathbf{e}_{k}$ are  in $\cF$. We can choose the former if $m\ge 0$ and the latter otherwise. Then the supremum is attained and $\d_\Exg = 2 |m|$.
\end{proof}

\section{Variational formulation of the systems}\label{sec:var}

\subsection{The exchange-driven growth system}\label{sec:varEDG}

\begin{definition}\label{def:ExCE} A pair $(c,j) \in \ExgCE(0,T)$ solves the exchange continuity equation provided that
\begin{enumerate}
\item the curve $c$ is absolutely $\d_\Exg$-continuous, that is $c \in \AC((0,T); (\cP^1, \d_{\Exg}))$;
\item the Borel measurable family of fluxes $(j_t)_{t\in[0,T]} \subset \cM^+ (\bN\times \bN_0))$ satisfy the bound
\begin{equation*}
\int_0^T  \sum_{k,l\geq 1}  j_t({k,l-1})\dx t<\infty ;
\end{equation*}
\item the couple $(c,j)$ satisfies the exchange continuity equation, that is, for any sequence $\Phi:\bN_0 \to \bR$ with finite support, and for all $0 \le s\le t\le T$,
\begin{equation*}\label{eq:intro-ExCE}
  \sum_{k\ge 0} \Phi_{k} (c_t(k)- c_s(k)) = \int_s^t \sum_{k,l \ge 1} \ExNabla_{k,l-1}\! \Phi \ j_{r}(k,l-1)\dx r ,\tag{\ensuremath{\mathsf{ExCE}}}
\end{equation*}  
where $\ExNabla$ is the exchange gradient from ~\eqref{eq:def:exchangeGradDiv}.
\end{enumerate} 
\end{definition}

Clearly, condition $(3)$ is equivalent to requiring that \ref{eq:intro-ExCE} is valid for $\Phi=\mathbf{e}_{k}$ for any $k\in \bN$, in this case 
\[c_t(k)-c_s(k)=-\int_s^t (\gDiv j_r)(k) \dx r \]
which is similar to the pointwise form of \eqref{eq:EDG}. 
Moreover, we can extend \eqref{eq:intro-ExCE} to all functions $\Phi$ with bounded exchange gradient. 
\begin{lemma}\label{lm:bgrad}
For any $(c,j)\in \ExgCE(0,T)$ and any $\Phi:\bN_0 \to \bR$ with $\sup_{k,l \in \bN}\lvert\ExNabla_{k,l-1} \Phi\rvert <\infty$, and for all $0 \le s\le t\le T$, the equality \eqref{eq:intro-ExCE} holds. 

In particular, with the choice $\Phi_k=k$, we have conservation of the first moment, i.e. $\Mom_1(c_t)$ is constant in time. 
\end{lemma}
\begin{remark}
For more general exchange equations, where for example $(k,l)\to (k-m,l+m)$ according to some rate depending on $k,l,m$, the first moment might not be preserved along the solution, in a similar way to the Lu-Wennberg solutions to the Boltzmann equations \cite{LUW2022}. However, in our current setting, truncation of the function $\Phi_k=k$ in the continuity equation leads to uniform estimates and hence conservation via a dominated convergence.
\end{remark}

\begin{proof}
Fix any $\Phi$ such that $\sup_{k,l \in \bN}\lvert\ExNabla_{k,l-1} \Phi\rvert < \infty$, and note that by Proposition \ref{prop:closedformdcF} that there exists constants $A,B$ such that $|\Phi_k|\leq A k+B$ and $\Phi$ is $A$-Lipschitz, and in particular the left-hand side of \eqref{eq:intro-ExCE} is well-defined. Moreover, consider the sequence of functions $\Phi^n$ 
\[ \Phi^n_k := \left\{ 
\begin{aligned}
    &\Phi_k, \qquad &&0\leq k \leq n \\
    &n^{-1}(-k+n)\Phi_n +\Phi_n , \qquad &&n\leq k \leq 2n
\end{aligned}
\right.\]
Note that $|\Phi^n_k|\leq A k+B$ for all $k,n$ and $\Phi^n$ is $A$-Lipschitz, uniformly in $n$.  Therefore, a dominated convergence argument allows us to take the limit $n\to \infty$ in \eqref{eq:intro-ExCE}.
\end{proof}

\begin{definition}[EDF for MFE]\label{def:EDF4MFE}
The Energy-Dissipation Functional for $\rho\in [0,\rho_c]$ with $\rho<\infty$ is defined as
\begin{equation}
\mathsf{L}^{\rho}(c,j) := \begin{cases}
\mathsf{E}^{\rho}(c(T))-\mathsf{E}^{\rho}(c(0)) + \int_0^T \Bigl(\mathsf{R}_t(c,j) + \mathsf{D}_t(c) \Bigr)\dx t \ &\text{if } (c,j) \in \mathsf{ExCE}(0,T), \\
+\infty\ & \text{otherwise.}
\end{cases}
\end{equation}    
In the above formula, the energy functional $\mathsf{E}: \cP^1 \to \bR_{\ge 0}$ is
\[\mathsf{E}^{\rho}(c):= \frac12 \mathsf{Ent}(c | \equi^{\rho}) =\frac12 \sum_{k\in\bN_0}  c_k \log \frac{c_k}{\equi_k^{\rho}}.\] See also Subsections~\ref{subsec:VarCon} and~\ref{subsec:ass} for the definition of the notations used here. The Fisher information $\mathsf{D}:\cP^1 \to \bR$ is
\[\mathsf{D}(c) := \sum_{k,l\in\bN}\mathrm{D}(\kappa[c](k,l-1), \kappa^{\dagger}[c](k,l-1)),\]
where $\rmD(u,v)=\frac12 \bigl((\sqrt{u}-\sqrt{v})^2 + (u -v)\bigr)= u- \sqrt{uv}$ and can be written with the Hellinger distance
\begin{align*} 
\mathsf{D}(c)
& = \mathsf{H}^2(\kappa[c] , \kappa^{\dagger}[c])+\frac12 \sum_{k,l\in \bN}( \kappa[c](k,l-1) - \kappa^{\dagger}[c](k,l-1) ).
\end{align*} 
with 
\[
    \mathsf{H}^2(\kappa[c] , \kappa^{\dagger}[c]) := \sum_{k,l\in\bN} \frac12  \biggl( \sqrt{ \kappa[c](k,l-1)} - \sqrt{\kappa^{\dagger}[c](k,l-1)}\biggr)^{\!2}. 
\]
The geometric average  $\theta$ of $\kappa,\kappa^\dagger$ is
\begin{equation}\label{def:theta}\theta[c](k,l-1) := \sqrt{ \kappa[c](k,l-1) \kappa^\dagger[c](k,l-1)}=K(k,l-1)\equi_{k}^\rho\equi_{l-1}^{\rho}\sqrt{u_k u_{l-1}u_{l}u_{k-1} },\end{equation} where $u \in \bR^{\bN_0}$ is given by
\[u := \frac{c}{\equi^\rho}.\]
The primal dissipation potential is $\mathsf{R}:\cP^1 \times \cM^{+}(E) \to \bR_{\ge 0}$ 
\[
    \mathsf{R}(c,j):= \mathsf{Ent}(j \,|\, \theta[c])=\sum_{k,l\in \bN} \phi\bigl(j(k,l-1) \,\big|\, \theta[c](k,l-1) \bigr) , \]
where $\phi(x)=x\log x -x +1$, and its perspective function is given by  
\[
    \phi(a\mid b)=
    \begin{cases} 
        b\phi(\frac{a}{b}) , & a\ne 0 , b>0 ;\\
        0 , & a=0 ; \\
        +\infty , & a \ne , b=0 .
    \end{cases}
\]
\end{definition}
\begin{remark}\label{rem:RD}
If the density $c\in \cP^1$ is positive, the identity $\mathsf{D}(c)= \mathsf{R}^{*}(c, -\frac{1}{2} \overline{\nabla} \log u)$ holds with the Legendre dual $\sfR^*$ of $\sfR$ given by
\[\mathsf{R}^{*}(c,\zeta)= \sum_{k,l\in\bN}(e^{\zeta(k,l-1)}-1 ) \theta[c](k,l-1) = \sum_{k,l\in\bN} (e^{\zeta(k,l-1)}-1 ) \sqrt{u_k u_{l-1} u_{l} u_{k-1}}\equi_k^\rho \equi_{l-1}^\rho  K(k,l-1).\]
Here and in the following, we use the convex duality pair $\phi, \psi:[0,\infty) \to [0,\infty)$ defined by
$\phi(r)= r \log r -r +1$ and
$\psi(r)=\phi^*(r) = e^r -1$. 
\end{remark}
The Remark~\ref{rem:RD} shows that we need to take care of the case when the argument of $\log$ is zero.
In particular, this is necessary in the chain rule for the entropy, where an appropriate regularization is used, for which we introduce 
\begin{definition}\label{def:A-and-B}
The extended function $\rmA:\bR_{\ge 0}\times \bR_{\ge 0} \to [-\infty,+\infty]$ is defined by 
\begin{align*}
\rmA(u,v):= \begin{cases} \log(v) -\log (u) &\text{ if } u,v \in \bR_{\ge 0}\times \bR_{\ge 0} \setminus \{(0,0)\},\\
0 &\text{ if } u=v=0
\end{cases}
\end{align*} 
and 
$\rmB(u,v,w):=\rmA(u,v) w$ with the convention for infinity cases as in \cite[Equation 4.41]{PRST22} given for $a\in [-\infty,\infty]$ by
\begin{equation}\label{eq:infinities}
    \abs*{\pm \infty} := +\infty, \quad a \cdot (+\infty) := 
    \begin{cases}
        + \infty , & \text{ if } a>0 ; \\
        0 , & \text{ if } a= 0 ; \\
        - \infty , & \text{ if } a< 0 , 
    \end{cases}
    , \quad a \cdot (-\infty) := -a \cdot (+\infty) \,.
\end{equation}
\end{definition}
\begin{proposition}[Chain rule]\label{prop:chain_rule}
Under assumption \eqref{e:ass:K1}, any curve $(c,j) \in \ExgCE(0,T)$ starting from $c(0)\in \cP^1$, $\rho\in[0,\rho_c]$ with $\rho<\infty$ and
\begin{equation}\label{ass:EDG-chain-rule:finiteRE}
  \int_0^T \sfR_r(c , j) \dx r < +\infty \qquad\text{and}\qquad  \sfE^{\rho}(c(0))<+\infty 
\end{equation}
satisfies for all $0\leq s \leq t \leq T$ the chain rule identity
\begin{equation}\label{eq:EDG-chain-rule}
\sfE^{\rho}(c(t) ) - \sfE^{\rho}(c(s))= \int_s^t \sum_{k,l\geq 1}\frac12 \rmB\bigl(\kappa_r[c](k,l-1),\kappa^{\dagger}_r[c](l,k-1), j_{r}(k,l-1)\bigr) \dx  r.
\end{equation} 
In particular, the energy dissipation functional is non-negative
\begin{equation}\label{eq:ineq_L_non-negative}
\sfL^{\rho}(c,j)\ge 0.
\end{equation} 
\end{proposition}
\begin{proof}
If $\rho=0$, the finiteness of $\sfE^{\rho}(c(0))$ and conservation of first moment imply all terms in \eqref{eq:EDG-chain-rule} are zeros, so it holds. Otherwise, we define the $m$-truncated logarithm for $m>0$ by
\begin{equation}\label{def:m-log}
	\log_m(x) := \max\set*{ \min\set*{ \log x , m},-m} .
\end{equation}
Further, we define the $m$-Lipschitz entropy density function by
\begin{equation*}
	\phi_m(x) := \int_1^x \log_m(y) \dx{y} . 
\end{equation*}
With this, we define the regularised energy 
\begin{equation}\label{eq:def:regularized:energy}
  \sfE^\rho_m(c) := \frac12 \sum_{k\geq 0} \equi^\rho_k \phi_m\biggl( \frac{c_k}{\equi^\rho_k}\biggr).
\end{equation}
\noindent
\emph{Step 1:} We first establish the chain rule for the $m$-regularised energy by a time-convolution and passing to the limit as $m\to +\infty$.
To do so, it is enough to observe that $|\log_m(x)| \leq m$ for any $x\geq 0$ by definition \eqref{def:m-log}.
Indeed, by Definition~\ref{def:ExCE} of $(c,j)\in \ExgCE(0,T)$, we can choose the test function to be $\Phi=\mathbf{e}_k$ for $k\in\bN_0$ and obtain
by linearity of the continuity equation, the strong form of the continuity equation for the time regularised pair $(c^{\delta}=c\ast \eta^\delta, j^{\delta}=j\ast \eta^\delta)$ by the rescaled standard mollifier $\eta$, that is
\begin{equation}\label{eq:cdelta:CE:strong}
	\partial_t c^\delta_k (t)= - (\ExDiv j^\delta)_k (t) \qquad\text{ for all } k \ge 0. 
\end{equation}
Next, we define an admissible test function in the continuity equation by additionally truncating the regularised energy in~\eqref{eq:def:regularized:energy} for some $B\in\bN$ by
\[ 
	\sfE^\rho_{m,B}\colon \bR^\bN \to \bR \qquad\text{with}\qquad  
	\sfE^\rho_{m,B}(c):=\frac12 \sum_{k\ge 0}^B \equi^\rho_k \phi_m \Bigl(\frac{c_k}{\equi^\rho_k}\Bigr) .
\]
Then, we obtain from~\eqref{eq:cdelta:CE:strong} and by the usual chain-rule 
\begin{equation}\label{eq:B-decay}\begin{split}
 2 \pderiv{}{t} \sfE^\rho_{m,B}(c^\delta(t)) &= \sum_{k= 0}^B \equi^\rho_k \phi_m'\biggl(\frac{c^\delta_k(t)}{\equi^\rho_k}\biggr) \frac{\partial_t c^\delta_k(t)}{ \equi^\rho_k} 
	= \sum_{k= 0}^B \log_m\biggl(\frac{c^\delta_k(t)}{\equi^\rho_k}\biggr) (-\ExDiv j^\delta_k(t)) .
\end{split}\end{equation}
Using the boundedness $|\log_m(x) |\le m $ and  \[\sum_{k\geq 0} | (\ExDiv j^\delta)_k (t)|  \le 4 \sum_{k,l\geq 1} | j_{k,l- 1}^\delta(t) |\le 4 \sum_{k,l\geq 1} | j_{k,l-1}(t) |,\] together TV bound of $j$ from $(c,j) \in \ExgCE(0,T)$,  we have the estimate uniform in  $B$ and $\delta$
\begin{equation}\label{eq:fluxbound} 
	\int_0^T \sum_{k = 0}^B \biggl\lvert \log_m\biggl(\frac{c^\delta_k(t)}{\equi^\rho_k} \biggr) (-\ExDiv j^\delta_k(t)) \biggr\rvert \dx t \le  4 m \int_0^T \sum_{k,l\geq 1} |j_{k,l-1}(t) |\dx t  <\infty.  \end{equation}
So we can take limit $B \to \infty$  for the integral of  \eqref{eq:B-decay} in time to obtain 
 \[\sfE^\rho_m(c^\delta(t))-\sfE^\rho_m(c^\delta(0))=\frac12 \int_0^t \skp[\bigg]{\log_m\biggl(\frac{c^\delta(r)}{\equi^\rho}\biggr), -\ExDiv j^\delta (r) }_{\!\bN} \dx r \qquad \forall t\ge 0. \]
 Since the estimate \eqref{eq:fluxbound} is uniform in $\delta$, we have 
 \[\lim_{\delta \to 0} \int_0^t \skp[\bigg]{\log_m\biggl(\frac{c^\delta(r)}{\equi^\rho}\biggr), -\ExDiv j^\delta (r) }_{\!\bN} \d r =  \int_0^t \skp[\bigg]{\log_m \biggl(\frac{c(r)}{\equi^\rho}\biggr), -\ExDiv j (r) }_{\!\bN} \d r \qquad \forall t\ge 0\]  
by the pointwise convergence of $c^\delta_k(t) \to c_k(t)$ and $\ExDiv_k j^\delta(t) \to j(t)$ on $[0,T]$ for each $k\ge 0$ and dominated convergence.
On the other hand, by the definition of $\phi_m$, we have 
\[
	\sum_{k\ge0} \biggl\lvert\equi^\rho_k \phi_m\biggl(\frac{c^\delta_k(t)}{\equi^\rho_k}\biggr)\biggr\rvert \le \sum_{k\ge0}  m c^\delta_k(t) \le m \sum_{k\ge 0} c_k(t) = m  
\]
which implies $\lim_{\delta \to 0}\sfE^m(c^\delta(t))=\sfE^m(c(t))$ for all $t\ge 0$, again by dominated convergence. Therefore, we get the chain rule for the $m$-regularised energy
\[
	\sfE^\rho_m(c(t))-\sfE^\rho_m(c(0))=\frac12 \int_0^t \skp[\bigg]{\log_m \biggl(\frac{c(r)}{\equi^\rho}\biggr), -\ExDiv j (r) }_{\!\bN} \d r =\frac12 \int_0^t \skp[\bigg]{\ExNabla \log_m\biggl(\frac{c(r)}{\equi^\rho}\biggr), j (r) }_{\!\ExgE} \d r,
\]
for which we used that $\log_m$ is bounded and thus can be used as the test function in the continuity equation \eqref{eq:intro-ExCE}.

\smallskip
\noindent
\emph{Step 2: Existence of limit.}
We claim that for almost all $r \in [0,T]$
\begin{align}\label{eq:rigorousLogDualforMF}
\rmB\biggr(\frac{c_k(r)c_{l-1}(r)}{\equi_k^\rho \equi_{l-1}^\rho}, \frac{c_{k-1}(r)c_l(r)}{\equi^\rho_{k-1}\equi^\rho_{l}}, j_{k,l-1}(r) \biggl)\!
&=\!\!\lim_{m \to \infty}\biggl[- \log_m \!\!\! \frac{c_k(r)c_{l-1}(r)}{\equi^\rho_k \equi^\rho_{l-1}}+ \log_m \!\!\! \frac{c_{k-1}(r) c_l(r)}{\equi^\rho_{k-1} \equi^\rho_{l}}  \biggr]  j_{k,l-1}(r).
\end{align} 
Note, by construction of the truncated logarithm~\eqref{def:m-log}, we get
\begin{equation*}
\lim_{m\to \infty} a \log_m b =
 \begin{cases} + \infty &\text{ if }a <0, b=0\\
 										- \infty &\text{ if } a>0, b=0\\
 										a \log b &\text{ if } a\neq 0, b> 0\\
 										 0 & \text{ if } a=0. \end{cases}\end{equation*} 
To see \eqref{eq:rigorousLogDualforMF}, we note that the integrability condition \eqref{ass:EDG-chain-rule:finiteRE} implies $\sfR(c,j)$ is finite for almost all $r\in[0,T]$. So we have the following two cases on the limit in $m$:							
\begin{enumerate}[(1)]
\item If all of $c_k,c_l,c_{l-1},c_{k-1}$ are non-zero, the limit converges to $\ExNabla_{k,l-1}\log \bigl(\frac{c}{{\equi^\rho}}\bigr) j_{k,l-1}$
\item If $\sfR(c,j)= \sum_{k,l-1}\mathsf{Ent}(j_{k,l-1}| \theta[c](k,l-1) ) <+\infty$ , then when any of $c_k,c_l,c_{l-1},c_{k-1}$ is zero, it also forces $j_{k,l-1}=0$ by recalling the definition of $ \theta$ in~\eqref{def:theta}.
In this case
\[\biggl(- \log_m  \frac{c_kc_{l-1}}{\equi^\rho_k \equi^\rho_{l-1}}+ \log_m \frac{c_{k-1} c_l}{\equi^\rho_{k-1} \equi^\rho_{l}} \biggr) j_{k,l-1} =0\] so the limit is zero.
\end{enumerate}
Comparing with Definition \ref{def:A-and-B}, we see the equality  \eqref{eq:rigorousLogDualforMF} holds for almost all  $r \in [0,T]$. 

\smallskip
\noindent
\emph{Step 3:}
To pass to the limit in $m\to \infty$, we derive a uniform bound via the convex duality
\begin{equation}\label{eq:abs of flux-dissipation}\Big| \frac12 \overline{\nabla}_{k,l-1} \log_m \Big(\frac{c}{\omega^\rho}\Big) \Big| j_{k,l-1}  \le \phi(j_{k,l-1}\,|\,\theta[c](k,l-1)) + \theta[c](k,l-1) \phi^{*}\Big(\Big|\frac12 \overline{\nabla}_{k,l-1} \log_m \Big(\frac{c}{\omega^\rho}\Big)\Big|\Big). \end{equation}
Using the monotonicity of $\phi^*(s)=e^s -1$,
\begin{equation*}\label{eq:nablam-log bound}
	\abs[\Big]{ \ExNabla_{k,l-1} \log_m \frac{c}{\equi^\rho}} \leq \abs[\Big]{ \partial^-_k \log_m \frac{c}{\equi^\rho}} + \abs[\Big]{\partial^-_l \log_m \frac{c}{\equi^\rho}} \leq  \abs[\Big]{ \partial^-_k \log \frac{c}{\equi^\rho}} + \abs[\Big]{\partial^-_l \log \frac{c}{\equi^\rho}},
\end{equation*}
we have
\begin{align*}
 \phi^*\bra*{ \frac12 |\ExNabla_{k,l-1} \log_m u|} &\leq  \exp\bra*{\frac12 \abs*{ \partial^-_k \log u} + \frac12 \abs*{\partial^-_l \log u}} \\
 &=  \max\set*{\sqrt{\frac{u_k}{u_{k-1}}},\sqrt{\frac{u_{k-1}}{u_{k}}}} \max\set*{\sqrt{\frac{u_l}{u_{l-1}}},\sqrt{\frac{u_{l-1}}{u_{l}}}}. \end{align*}
Then
\begin{align}\label{eq:total-daulity-bound}
  \theta[c](k,l-1)  \phi^*\bra[\bigg]{ \frac12 |\ExNabla_{k,l-1} \log_m u |}  &\le   \equi^\rho_k \equi^\rho_{l-1} K(k,l-1) \max\set*{ u_k,u_{k-1}} \max\set*{ u_l ,u_{l-1}}.\end{align} Hence using the assumptions  \eqref{e:ass:K1} and we have a uniform bound in $m$
\begin{equation*}
	\sum_{k,l}  \theta[c](k,l-1)  \phi^*\bra[\bigg]{ \frac12 |\ExNabla_{k,l-1} \log_m u |} \leq  C_K C_{\equi^\rho} \sum_{k,l\geq 1} k l \bra*{  c_k + c_{k-1}} \bra*{c_l + c_{l-1}} 
	\leq 4 C_K C_{\equi^\rho}  (\tilde{\rho}+1)^2 \,,
\end{equation*}
where $\tilde{\rho}\in [0,+\infty)$ is the first moment of $c(0)$,\begin{equation*}
	C_{\equi^\rho} =  \sup_{k\geq 1} \max\set*{\frac{\equi^\rho_k}{\equi^\rho_{k-1}} , \frac{\equi^\rho_{k-1}}{\equi^\rho_{k}}} < \infty
\end{equation*} due to  \eqref{e:ass:Kc}. Indeed, the finiteness of $C_{\equi^\rho}$ is thanks to assumption~\eqref{e:ass:Kc} due to 
\[
  \lim_{k\to \infty}\frac{\equi^\rho_k}{\equi^\rho_{k-1}}=\Phi(\rho) \cdot \lim_{k\to\infty}\frac{K(1,k-1)}{K(k,0)}<\infty \,.
\] 
By dominated convergence, we can pass to the limit in the chain rule for $m$-truncated logarithm and use the definition of $K^\dagger $ to get
\[\begin{split}\sfE^\rho(c(t))-\sfE^\rho(c(0)) &= \int_0^t \sum_{k,l\ge 1}    \frac12 \rmB\biggr(\frac{c_k(r)c_{l-1}(r)}{\equi_k^\rho \equi_{l-1}^\rho}, \frac{c_{k-1}(r)c_l(r)}{\equi^\rho_{k-1}\equi^\rho_{l}}, j_{k,l-1}(r) \biggl) \d r\\
&=  \int_0^t \sum_{k,l\ge 1} \frac12   \rmB\biggr(\kappa_r[c](k,l-1), \kappa^\dagger_r[c](l,k-1) , j_{r}(k,l-1) \biggl) \d r.
\end{split}\]
\noindent
\emph{Step 4:}
Define
\[\mathsf{D}^{-}(c) := \sum_{k,l\in\bN}\mathrm{D}^{-}(\kappa[c](k,l-1), \kappa^{\dagger}[c](k,l-1)),\] with
\[\mathrm{D}^-(\kappa[c](k,l-1), \kappa^{\dagger}[c](k,l-1)):=\begin{cases} 0 \quad&\text{ if } \theta[c] (k,l-1)=0\\
\mathrm{D}(\kappa[c](k,l-1), \kappa^{\dagger}[c](k,l-1)) &\text{ otherwise. }
\end{cases}\] 
We claim that the inequality 
\[-\frac12 \mathrm{B}(u_k u_{l-1}, u_l u_{k-1}, j_{k,l-1}) \le \phi(j_{k,l-1}|\theta[c](k,l-1))+ \mathrm{D}^-(c,k,l-1)
\] holds.
Indeed, if $\theta[c](k,l-1)>0$, it follows from the following form of Fenchel's inequality for $u,v>0$, $j>0$,
\[\frac12 (\log u - \log v) j \le \phi(j|\theta ) + \theta \phi^*\left(\frac12 (\log u -\log v)\right)\]
and if $\theta=K \sqrt{u v}$ for some $K>0$,  we have 
\[\theta \phi^*\left(\frac12 (\log u - \log v)\right)= \frac12 K ( (\sqrt{u}-\sqrt{v})^2 + (u-v)).\] 
Choosing $u=u_{k}u_{l-1}, v = u_{l}u_{k-1}, j = j_{k,l-1}$ and $\theta= \theta[c](k,l-1)$, we have the inequality for the case $\theta[c](k,l-1)>0$.
Now consider the case for $\theta\propto \sqrt{uv} =0$, we have
\[-\mathrm{A}(u,v)j=\mathrm{A}(v,u) j=\begin{cases} -\infty &\text{ if } j>0, u=0, v>0,\\
+\infty &\text{ if } j>0, u>0,v=0,\\
0 &\text{ if } u=v=0 \text{ or } j=0.
\end{cases}
\]
On the other hand,
\[\phi(j| \theta) =\begin{cases} 0 &\text{ if } j = 0 =\theta,\\
                                  +\infty &\text{ if } j >0 , \theta = 0.
\end{cases}\]
Hence if $\theta=0$, $-\frac12 \mathrm{B}(u,v,j) \le \phi(j | \theta)$ with values on $[-\infty,\infty]$. Taking the integral and summation on both sides of the inequality, we get
\[\begin{split} \sfE^{\rho}(c(s) ) - \sfE^{\rho}(c(t))&\le \int_s^t \sum_{k,l} \phi(j_{r}(k,l-1)|\theta_r[c](k,l-1))+ \mathrm{D}^-_r (c,k,l-1)  \d r
\\ &= \int_s^t  \mathsf{R}_r(c,j) + \mathsf{D}^-_r (c) \d r.
\end{split}\]
In fact, by the estimate in Step 3 \eqref{eq:abs of flux-dissipation}, \eqref{eq:total-daulity-bound}, the lower bound of inequality always takes a finite value for our choice of parameters. So we get the inequality \begin{equation*}
\sfL^{\rho,-}(c,j)\ge 0.
\end{equation*}  where $\rmD$ is replaced by $\rmD^-$ in the definition of $\sfL$.
Due to $\rmD(u,v)= u- \sqrt{uv} \ge 0$ if $uv=0$ and $u \ge 0$, we have $\rmD \ge \rmD^-$ and $\sfD\ge \sfD^-$ so that $\sfL^\rho(c,j) \ge 0$.
\end{proof}

\begin{lemma}[Equality conditions]\label{lem:equ_in_D} Let $u,v \ge 0$.
If $j\ge 0$ is such that
\begin{equation}\label{eq:equality_in_duality}
-\frac12 \rmB(u, v, j)= \phi\bigl(j| \theta \bigr) +  K \rmD(u,v)\in \bR,
\end{equation} 
$\rmD(u,v)=\frac12 ((\sqrt{u}-\sqrt{v})^2 + (u -v))$, $\theta=K\sqrt{uv}$, $K>0$
then we have either $\{u=0\}$  or  $\{u> 0, v>0\}$ and
\begin{equation}\label{eq:equal_implies_flux}
j=K u.
\end{equation}
In particular, in this case $\rmD(u,v) =\rmD^-(u,v)$.

Conversely, if $j=Ku$, then
\begin{equation}\label{eq:equality_in_duality_inf}
-\frac{1}{2}\rmB(u, v, j)= \phi(j|\theta)+ K \rmD(u,v) \in [0,+\infty].
\end{equation} 
\end{lemma}
\begin{proof}
We consider different cases for $j$.
First suppose $j=0$, then we have $-\frac12 \rmB(u, v, j)=0$ and $\phi\bigl(j| \theta \bigr) +  K \rmD(u,v) = \theta +  K D(u,v)=K u $ so the Equation \eqref{eq:equality_in_duality}  implies  $u=0$. 

Alternatively for $j > 0$ and $\theta=0$, we have $\phi(j|\theta)= +\infty$, 
so this cannot happen if the equality \eqref{eq:equality_in_duality} holds in $\bR$. Thus, we must have $j>0$ and $\theta>0$, which implies $u>0$ and $v>0$. By the definition of Legendre transform, $\phi^*(w) = \sup_{a > 0} (w a - \phi(a))$, we find that  the supremum is attained at 
\[
	j = \theta \exp\bra*{ \tfrac12 (\log u - \log v)} = K u \,.
\]

Conversely, we start with $j = Ku$. We consider cases where $u,v$ are zero and positive.  If $u=j=0$, the functions $\mathrm{B}(u,v,j)$, $\phi(j|\theta)$ and $D(u,v)$ are zero. If $u>0$ and  $v=0$, both sides equal to $+\infty$. In the case $u>0, v>0$, by the definition of $\rmB$
\begin{equation*}
-\frac12 \mathrm{B}(u,v, Ku) - \phi(Ku|\theta) = \frac12 K u \log(\frac{u}{v})- K \sqrt{u v} \phi (\sqrt{\frac{u}{v}})=Ku(1 - \sqrt{\frac{v}{u}}) =   K D(u,v) \,. \qedhere
\end{equation*}
\end{proof}

\begin{proposition}[EDP solution for~\eqref{eq:MF}]\label{prop:zero-set-mean-field-EDP-functional}
Under assumption \eqref{e:ass:K1}, let $\rho\in [0,\rho_c]$ with $\rho<\infty$. Then for any $c(0)\in \cP^1$ with $\sfE^{\rho}(c(0))<\infty$, the following are equivalent
\[\sfL^{\rho}(c,j) = 0 \iff 
\begin{cases}  (c,j) \in \ExgCE(0,T)  \\ 
j(t) = \overline\jmath_t[c] \text{ for almost all } t \in [0,T]\\
\sfE^{\rho}(c(T))+ \int_0^T\mathsf{F}_t(c)    \d t  \le \sfE^{\rho}(c(0)). \end{cases}\]
Hereby, 
\[
	\mathsf{F}_t(c):=-\frac12 \sum_{k,l\in\bbN} \rmB\Bigl(\kappa_t[c](k,l-1) ,\kappa^{\dagger}_t[c](k,l-1) , \overline\jmath_t [c](k,l-1) \Bigr) \,,
\]
with $\overline\jmath_t[c]= \kappa_t[c]$ the expected flux defined in~\eqref{eq:MF}.
\end{proposition}
\begin{proof} 
For $\sfL^{\rho}(c,j)=0$, we can apply the chain rule  in Proposition \ref{prop:chain_rule}  to see the equality 
\[\int_0^T -\frac12 \sum_{k,l\in\bbN} \rmB\Bigl(\kappa_t[c](k,l-1) ,\kappa^{\dagger}_t [c](k,l-1) , \overline\jmath_t[c](k,l-1) \Bigr) \d t =  \int_0^T \sfR_t (c,j)+  \sfD_t (c)\d t \] holds. Then Lemma \ref{lem:equ_in_D} implies  $j_{t}(k,l-1)= K_t [\omega](k,l-1)u_{k}(t) u_{l-1}(t)= K_t [c](k,l-1)$ for each $k,l$ and for almost all $t\in[0,T]$. This implies
\[\sfE^{\rho}(c(T))+ \int_0^T\mathsf{F}_t(c)    \d t  = \sfE^{\rho}(c(0)). 
\]
Conversely, we can apply the converse of the equality condition in Lemma \ref{lem:equ_in_D} to get 
\[\int_0^T \mathsf{F}_t(c) \d t = \int_0^T \sfR_t(c,j)+  \sfD_t(c)\d t.\] Hence $\sfL^{\rho}(c,j)\le 0$ and by the non-negativity from~\eqref{eq:ineq_L_non-negative}, we have $\sfL^{\rho}(c,j)= 0$ . 

\end{proof}

\subsection{The forward Kolmogorov equation}
\label{sec:var4FKE}
\begin{definition}[Finite particle continuity equation]
	Let $N, L\in \bN$, a pair $(\bC^{N,L}_t,\bJ^{N,L}_t)_{t\in[0,T]}\subset  \cP(\hat{V}^{N,L})\times \cM^+(\hat{E}^{N,L})$ is said to satisfy the continuity equation, denoted by $(\bC^{N,L},\bJ^{N,L})\in \mathsf{CE}^N(0,T)$ provided that:
\begin{enumerate}
\item the fluxes $(\bJ^{N,L}_t)_{t\in[0,T]} \subset \cM^+(\hat{E}^{N,L})$ is a Borel measurable family which maps to finite measure on $\hat{E}^{N,L}$ satisfying 
\[
	\int_0^T \sum_{(c,k,l)\in \hat E^{N,L}}    \bJ^{N,L}_t(c,k,l-1) \dx t <\infty \,.
\]
\item $(\bC^{N,L},\bJ^{N,L})$ satisfies the continuity equation $\partial_t \bC_t^{N,L} = - \DivL  \bJ^{N,L}_t$ in the sense that
 for all $\mathbbm{\Phi} :\hat{V}^{N,L}\to \bR$, and all $0\le s\le t\le T$
\begin{align*}\label{eq:cgce}\tag{\ensuremath{\mathsf{CE}^N}}
  \sum_{c\in \hat{V}^{N,L}} \bigl(\bC^{N,L}_t(c) -\bC^{N,L}_s(c)\bigr) \mathbbm\Phi(c) =\int_s^t \sum_{(c,k,l) \in \hat{E}^{N,L}} \hat{\nabla}^L_{k,l-1}\mathbbm\Phi(c) \, \bJ^{N,L}_r(c,k,l-1)  \dx r. 
\end{align*}
\end{enumerate}
\end{definition}
Now we introduce the microscopic model on $V^{N,L}$ for which the Markov process on $\hat{V}^{N,L}$ can be considered as the lifted version. The microscopic detailed balance condition will be lifted to a condition of the stochastic particle systems.
\begin{definition}[Microscopic model]\label{def:micromodel}
Consider $N$ indistinguishable particles on a complete graph with $L$ nodes, which are also called \emph{clusters} in the present context. The system is described by counting the number of particles on 
\begin{equation}\label{eq:def:micro-state-space}
	\eta \in  V^{N,L}:=\biggl\{\eta\in \bN_0^{L} :  \sum_{x =1 }^L \eta_x = N\biggr\} .
\end{equation}
The infinitesimal generator of the jump process with state space $V^{N,L}$, for any $f: V^{N,L}\to \mathbb{R}$, is given by 
\begin{equation}\label{eq:def:generator}
	\pgen f(\eta) := \frac{1}{L-1}\sum_{x,y=1}^L K(\eta_x,\eta_y) ( f(\eta^{x,y}) - f(\eta))  ,
\end{equation}
with $\eta^{x,y}:= \eta- \mathbf{e}_x +\mathbf{e}_y$ is the system after a particle jump from the $x$-th cluster to the $y$-th cluster and $\mathbf{e}_x\in \bN_0^L$ is the $x$-th canonical unit vector. We denote the edge set  by 
\begin{equation}\label{eq:def edge-set-micro}
E^{N,L}:=\{(\eta,x,y)\in V^{N,L}\times \{1,...,L\} \times \{1,...,L\}\colon 	\eta^{x,y}\in V^{N,L}\}.
\end{equation}
The kernel $K:\bN\times \bN_0\to \bR_{\geq0}$ in the generator \eqref{eq:def:generator} determines the jump rate $K(\eta_x,\eta_y)$ for a particle jumping from a cluster of $\eta_x$ particles to a cluster of $\eta_y$ particles. The overall prefactor $\frac{1}{L-1}$ is the uniform rate to jump to any of the $L-1$ neighbours from the given cluster and shows the mean-field scaling of the equation.

The two generators are related by the consistency relation $\pgen(G\circ C^L) (\eta) =: \cgen G(c)$ for any $\eta\in V^{N,L}$ such that $C^L[\eta]=c$, 
\begin{equation}\label{eq:def of lift}
C^L[\eta] := \frac{1}{L}\sum_{x =1}^L \delta_{\eta_x}  \text{ for } \eta\in V^{N,L}
\end{equation} and for any function $G: \hat V^{N,L}\to \bR$.
\end{definition}
\begin{definition}[Reference measure]\label{def:lifted-equilibrium} Define the lifted measure $\mathbbm{\Pi}^{N,L}$ on $\cP(\hat{V}^{N,L})$ of the probability measure $\pi \in  \cP(V^{N,L})$  as the pushforward measure of $C^L$ 
\[\mathbbm{\Pi}^{N,L} = C^L_{\#}\pi^{N,L},\] where 
$C^L(\eta)=\frac1L \sum_{x=1}^N \delta_{\eta_x} \in \hat{V}^{N,L}$, $\pi^{N,L}(\eta)=\frac{1}{\cZ^{N,L}}\prod_{x=1}^L w(\eta_x) \in \cP(V^{N,L})$ with normalization $\cZ^{N,L}$ for $\eta \in V^{N,L}$ and $w$ as defined in \eqref{eq:def:w}.
\end{definition}
 
%


\begin{proposition}\label{prop:property-backwardkernel}
The backward kernel $K^{\dagger}$ satisfies the 
\begin{enumerate}[(i)]
\item \label{eq: backward-on-position-rep}
$\pi^{N,L}(\eta)K(\eta_x,\eta_y) = \pi^{N,L}(\eta^{x,y})K^{\dagger}(\eta_y^{x,y},\eta_x^{x,y})$,

\item \label{eq: lifted-backward}$
\mathbbm{\Pi}^{N,L}(c) \kappa^L[c](k,l-1) = \mathbbm{\Pi}^{N,L}(c^{k,l-1}){\kappa^{L}}^{\dagger}[c^{k,l-1}](l,k-1)$.
\end{enumerate}
\end{proposition}
\begin{proof}
\noindent
\emph{Ad (i):}
\item Since $\pi^{N,L}(\eta^{x,y})= \pi^{N,L}(\eta)  \frac{\omega(\eta_x -1)}{\omega(\eta_x)} \frac{\omega(\eta_y +1 )}{\omega(\eta_y)}$, where  $\omega=\omega^\fug$ for any $\fug\in[0,\fugcrit)$ defined in \eqref{eq:canmeasure}, we have
\[\begin{split}\pi^{N,L}(\eta^{x,y}) K^\dagger(\eta^{x,y}_y, \eta^{x,y}_x) &= \pi^{N,L}(\eta)\frac{1}{\omega(\eta_x)\omega(\eta_y)} \omega(\eta_x-1)\omega(\eta_y+1) K^{\dagger}(\eta_y +1 , \eta_x -1) \\
&= \pi^{N,L}(\eta) K(\eta_x ,\eta_y).
\end{split}\]
\noindent
\emph{Ad (ii):}
We decompose the microscopic edge set $E^{N,L}$ from \eqref{eq:def edge-set-micro} into
\[
	F(c,k,l-1):=\bigl\{(\eta,x,y)\in E^{N,L}: C^L(\eta)=c, C^L(\eta^{x,y})=c^{k,l-1} \bigr\}.
\]
We observe that by definition 
\begin{equation*}\begin{split}
F(c,k,l-1)&= \bigl\{ (\eta,x,y)\in E^{N,L} : C^L(\eta)=c, \eta_x =k ,\eta_{y}=l-1  \bigr\} \\
& = \bigl\{ (\eta,x,y)\in E^{N,L} : C^L(\eta^{x,y})=c^{k,l-1}, \eta^{x,y}_x= k-1, \eta^{x,y}_y = l  \bigr\} \,,
\end{split}\end{equation*} 
so that we can describe a transition from $c$ to $c^{k,l-1}$ either by the initial distribution with the number of particles in the two boxes involved or by the final distribution with the number of particles in the two boxes involved.

From \ref{eq: backward-on-position-rep}, we have
\begin{equation}\label{eq:DBC-position-level} 
	\pi^{N,L}(\eta) K(k,l-1)=\pi^{N,L}(\eta^{x,y}) K^\dagger (l, k-1) \quad \text{ for all } (\eta,x,y) \in  F(c,k,l-1)\,.
\end{equation}
By summing over the above equation and after multiplying by $L^{-2}$ on both sides, we arrive at
 \begin{equation*}\begin{split}
L^{-2} \!\! \sum_{\substack{(\eta,x,y\in \\F(c,k,l-1)}}\!\! \pi^{N,L}(\eta) K(\eta_x,\eta_y)
&=\frac{K(k,l-1)}{L^2} \!\!\sum_{(\eta,x,y) \in F(c,k,l-1)}\!\!   \pi^{N,L}(\eta)\\
&= \frac{K(k,l-1)}{L^2}\!\!\sum_{\eta:C^L(\eta)=c}\!\! \pi^{N,L}(\eta) \; \#\bigl\{(x,y) :  C^L(\eta^{x,y})=c^{k,l-1} \bigr\}\, .
\end{split}\end{equation*}
The idea behind the last equality is that we can sum over $F(c,k,l-1)$ in two steps. First, we fix an~$\eta$ with $C^L(\eta)=c$ and count how many possible pathways to change from $C^L(\eta)$ to $C^L(\eta^{x,y})$ and second sum over $\{\eta: C^L(\eta)=c\}$. For a fixed $\eta $ with $C^L(\eta)=c$, we have
\begin{equation}\label{eq:counting-forward}
\#\bigl\{(x,y) : C^L(\eta^{x,y})=c^{k,l-1}\bigr\}=  L^2 c_{k}\bigl(c_{l-1}-L^{-1} \delta_{k,l-1}\bigr) \,,
\end{equation} 
because $L c_k$ is the number of boxes with $k$ particles in $\eta$ and we only count the jump between different boxes. So we need to reduce the number by $1$ if $k=l-1$. It follows that
\begin{equation*}\begin{split}
\MoveEqLeft\frac{K(k,l-1)}{L^2}\sum_{\eta:C^L(\eta)=c} \pi^{N,L}(\eta) \; \#\{(x,y) : C^L(\eta^{x,y})=c^{k,l-1} \}\\
&=K(k,l-1) c_{k} (c_{l-1}-L^{-1}\delta_{k,l-1})\sum_{\eta:C^L(\eta)=c} \pi^{N,L}(\eta) \\
&=K(k,l-1) c_{k} (c_{l-1}-L^{-1}\delta_{k,l-1}) C^L_{\#}\pi^{N,L}(c)\\
&= \frac{L-1}{L}\mathbbm{\Pi}^{N,L}(c) \kappa^L[c](k,l-1).
\end{split}\end{equation*}
On the other hand, we can analogously modify the summation on the other side of the equation \eqref{eq:DBC-position-level} to get
\begin{equation*}\begin{split}
\MoveEqLeft L^{-2}\sum_{\substack{(\eta,x,y) \in \\F(c,k,l-1)}} \pi^{N,L}(\eta^{x,y}) K^\dagger (\eta_y^{x,y}, \eta_x^{x,y})\\
&=\frac{K^\dagger (l,k-1)}{L^2}\sum_{(\eta,x,y) \in F(c,k,l-1)} \pi^{N,L}(\eta^{x,y}) \\
&=\frac{K^\dagger (l,k-1)}{L^2}\sum_{V^{N,L}[c^{k,l-1}] } \pi^{N,L}(\eta^{z,z'})\; \#\{(x,y) :  C^L(\eta^{x,y})=c^{k,l-1}, C^L(\eta)=c\} \,
\end{split}\end{equation*} 
where set $V^{N,L}[c^{k,l-1}] := \{\eta: C^L(\eta^{z,z'})=c^{k,l-1} \text{for some } z, z' \in \set*{1,\dots,L} \}$ and for the last equality sum in two steps: we fix $\eta$ such that it has a transition pathway to $C^L(\eta^{z,z'})= c^{k,l-1}$ and then count the number of transition pathways that from $c^{k,l-1}$ back to $c$. We note that by definition of $\pi^{N,L}$, it does not depend on the label of the box but the cluster size distribution so we can write $\pi^{N,L}(\eta^{z,z'})$ independent of $(x,y)$ in the product form above.

On the set $V^{N,L}[c^{k,l-1}]$, we get similar to \eqref{eq:counting-forward} the identity
\[
	\#\{(x,y) :  C^L(\eta^{x,y})=c^{k,l-1}, C^L(\eta)=c\}=  L^2 c^{k,l-1}_{l} (c^{k,l-1}_{k-1}-L^{-1}\delta_{l,k-1}) .
\]
Consequently, we have
\begin{equation*}\begin{split}
\MoveEqLeft\frac{K^\dagger (l,k-1)}{L^2}\sum_{\eta \in V^{N,L}[c^{k,l-1}] } \pi^{N,L}(\eta^{z,z'})\; \#\{(x,y) :  C^L(\eta^{x,y})=c^{k,l-1}, C^L(\eta)=c\}\\
&= K^\dagger (l,k-1) c^{k,l-1}_{l} (c^{k,l-1}_{k-1}-L^{-1}\delta_{l,k-1})\sum_{\eta \in V^{N,L}[c^{k,l-1}] } \pi^{N,L}(\eta^{z,z'})\\
&= K^\dagger (l,k-1) c^{k,l-1}_{l} (c^{k,l-1}_{k-1}-L^{-1}\delta_{l,k-1}) C^L_{\#}\pi^{N,L}(c^{k,l-1})\\
&=\frac{L-1}{L}\mathbbm{\Pi}^{N,L}(c^{k,l-1}) {\kappa^L}^{\dagger}[c^{k,l-1}](l,k-1) 
\end{split}\end{equation*}
Finally, by the detailed balance condition \eqref{eq:DBC-position-level}, we have 
\[
	\mathbbm{\Pi}^{N,L}(c)\kappa^L[c](k,l-1) =\mathbbm{\Pi}^{N,L}(c^{k,l-1}) {\kappa^L}^{\dagger}[c^{k,l-1}](l,k-1) \,. \qedhere
\]
\end{proof}

\begin{definition}[EDF for the finite particle system~\eqref{eq:FKE}]\label{def:EDF_FKE}
The energy dissipation function for the finite particle system is defined for any $(\bC^{N,L},\bJ^{N,L}) \in \ref{eq:cgce}(0,T)$ to be given by
\begin{equation}\label{eq:def:EDG:EDF}
	\cL^{N,L}(\bC^{N,L}, \bJ^{N,L}):=
\left. \cE^{N,L}(\bC^{N,L}_t)\right|_{t=0}^T
\!+\int_0^T  \Bigl[  \calR^{N,L}_t (\bC^{N,L},\bJ^{N,L})+\calD^{N,L}_t (\bC^{N,L})  \Bigr]\d t ,  
\end{equation} where the energy functional $\mathcal{E}^{N,L}$ is
\[ \mathcal{E}^{N,L}(\bC^{N,L})=  \frac{1}{2L} \mathcal{Ent}(\bC^{N,L}| \mathbbm{\Pi}^{N,L}),\]\\
the primal dissipation potential is defined as 
\[ \mathcal{R}^{N,L}(\bC,\bJ)= \mathcal{Ent}(\bJ | \mathbbm{\Theta}^L[{\bC}]), \quad \mathbbm{\Theta}^L[\bC]= \sqrt{ \nu^L[\bC] S_{\#}{{\nu^L}^{\dagger}[\bC]}},\] 
with the expected flux $\nu^L[\bC^{N,L}](c,k,l)= \bC^{N,L}(c) \kappa^L[c](k,l)$ and the pushforward measure
\[ \flipL{\nu}[\bC^{N,L}](c,k,l-1)=\bC^{N,L}(c^{k,l-1}) {\kappa^L}^{\dagger}[c^{k,l-1}](l,k-1)\] 
induced by the flip map  
$S(c,k,l-1)=(c+\frac1L \gamma^{k,l-1},l,k-1)$. The inverse of the flip map is $ S^{-1}(c,k,l-1)=(c -\frac1L \gamma^{k,l-1},l,k-1)$ defined for $c-\frac1L \gamma^{k,l-1} \ge 0$.
Finally, we define the Fisher information 
\begin{align*}
\mathcal{D}^{N,L}(\bC^{N,L})
&= \int \sum_{k,l} \mathrm{D}(\nu^L[\bC^{N,L}](c,k,l-1),\flipL{\nu}[\bC^{N,L}](c,k,l-1) )\\
&= \mathcal{H}^2(\nu^L[\bC^{N,L}], \flipL{\nu}[\bC^{N,L}])+\frac12 \int \sum_{k,l} (\nu^L[\bC^{N,L}]-\flipL{\nu}[\bC^{N,L}])(\d c, k ,l-1),
\end{align*}
with
\[\mathcal{H}^2(\mu,\nu) =\frac12 \int \abs[\Bigg]{\sqrt{\frac{\dx \mu}{\d \lambda}} - \sqrt{\frac{\d \nu}{\d \lambda}}} ^2 \dx \lambda,
\] for any measure $\lambda$ dominating both $\mu,\nu\ll \lambda$.

\end{definition}
\begin{remark}\label{rk:change-of-gradient strucutre}
The dissipation potential $\calD^{N,L}$ is related (as in Remark~\ref{rem:RD}) to the primal $\calR^{N,L}$ for the force given in terms of the energy, that is
\[ 
	\mathbb{W}^{N,L}[\bC^{N,L}](c,k,l-1) = -\frac{1}{2} \hat{\nabla}^{N,L}_{k,l-1} D \mathcal{E}^{N,L}(\bC^{N,L}(c))=-\frac{1}{2} ( \log \frac{\bC^{N,L}}{\mathbbm{\Pi}^{N,L}}(c) - \log \frac{\bC^{N,L}}{\mathbbm{\Pi}^{N,L}}(c^{k,l-1}))
\]
with $\frac{\bC^{N,L}}{\mathbbm{\Pi}^{N,L}}(c),\frac{\bC^{N,L}}{\mathbbm{\Pi}^{N,L}}(c^{k,l-1})>0$. 
In this case, whenever $\mathbbm{\Theta}^L[\bC^{N,L}](c,k,l-1)>0$ for all $(c,k,l)\in\hat{E}^{N,L}$, we have the identity
\[ {\mathcal{R}^{N,L}}^{*}(\bC^{N,L},\bW^{N,L}[\bC^{N,L}])= \int \sum_{k,l \ge 1 } \bigl(\exp({\bW^{N,L}[\bC^{N,L}]})-1 \bigr) \mathbbm{\Theta}^L[\bC^{N,L}](\d c,k ,l-1)= \mathcal{D}^{N,L}(\bC^{N,L}).\]
\end{remark}
 
\begin{lemma}[Uniform boundedness of kernel] \label{lem:ubkernel} 
	The jump kernel is uniformly bounded by
\[ \max \Big(\sum_{(c,k,l) \in \hat{E}^{N,L}} {\nu}^L[{\bC}^{N,L}] (c,k,l-1), \sum_{(c,k,l) \in \hat{E}^{N,L}} \flipL{\nu}[\bC^{N,L}]\Big) \le 2 C_{\overline{K}} e^{\|b\|_{\infty}}  \sup_{N/L\to \rho}\frac{N}{L} \biggl(\frac{N}{L}+1\biggr). \]
\end{lemma}
\begin{proof}
We consider the first term
\[\begin{split}\sum_{(c,k,l) \in \hat{E}^{N,L}}\nu^L[\bC^{N,L}](c,k,l-1)&= 
\sum_{(c,k,l) \in \hat{E}^{N,L}}\bC^{N,L}(c) \kappa^L[c](k,l-1)\\
&\le\frac{L}{L-1}\sum_{(c,k,l) \in \hat{E}^{N,L}}\bC^{N,L}(c) K(k,l-1) c_k c_{l-1}\\
&\le C_{\overline{K}} e^{\|b\|_{\infty}}\frac{L}{L-1} \sum_{(c,k,l) \in \hat{E}^{N,L}} \bC^{N,L}(c) k l c_k c_{l-1} \\
&\le C_{\overline{K}} e^{\|b\|_{\infty}}\frac{L}{L-1} \frac{N}{L} \biggl(\frac{N}{L}+1\biggr).\end{split}\]   
For the second term
\[\begin{split}\sum_{(c,k,l) \in \hat{E}^{N,L}} \flipL{\nu}[\bC^{N,L}](c,k,l-1) &= \sum_{(c,k,l) \in \hat{E}^{N,L}} {\nu^L}^{\dagger}[\bC^{N,L}](c,k,l-1) \1_{S^{-1}(\hat{V}^{N,L})}(c,k,l-1) \\
&=\sum_{(c,k,l) \in \hat{E}^{N,L}}{\nu^L}^{\dagger} [\bC^{N,L}](c,k,l-1)\\
&\le C_{\overline{K}} e^{\|b\|_{\infty}}\frac{L}{L-1} \frac{N}{L} \biggl(\frac{N}{L}+1\biggr).
\end{split}\] This concludes the proof.
\end{proof}
 
\begin{proposition}[Chain rule]\label{prop:finitechainrule} Under assumption~\eqref{e:ass:K1}, let $(\bC^{N,L},  \bJ^{N,L}) \in \ref{eq:cgce}(0,T)$ with 
\begin{equation}
\label{ass:finite_chain_rule}  \int_0^T \calR^{N,L}_t (\bC^{N,L}, \bJ^{N,L}) \dx t <+\infty \qquad\text{ and }\qquad \cE^{N,L}(\bC_0^{N,L})<+\infty.
\end{equation}
Then for $0\le s\le t\le T$, the chain rule holds
\begin{equation}\label{eq:EDF_FKE_chain_rule}
\cE^{N,L}(\bC^{N,L}_t)-\cE^{N,L}(\bC^{N,L}_s)= \int^t_s  \sum_{(c,k,l)\in \hat{E}^{N,L}}   \frac1{2L} \hat{\nabla}^L_{k,l-1}\phi'\biggl(\frac{\bC^{N,L}_r}{\mathbbm{\Pi}^{N,L}}(c)\biggr)\bJ^{N,L}_r(c,k,l-1) \dx r 
\end{equation}
and the energy dissipation function defined in~\eqref{eq:def:EDG:EDF} is non-negative
\begin{equation}\label{eq:EDF_FKE_non-negative}
{\cL^{N,L}}(\bC^{N,L},\bJ^{N,L})\ge 0 \,.
\end{equation}
\end{proposition}
\begin{remark}The proof follows closely the proof of the chain rule in \cite[Chapter 12, Section 12.3]{Hoeksema-thesis}. Note that although \cite{Hoeksema-thesis} does not cover time-dependent kernels, upon closer examination of the proof reveals that instead of \cite[Equation 12.1]{Hoeksema-thesis}, we can replace it by uniform boundedness in time of the kernel for fixed $N,L$, the proofs can be generalized to our case. Therefore, under our assumption of bounded perturbations, the chain rule holds in this case.
\end{remark}
As a result, we shall connect our definition of fluxes to theirs.  
\begin{lemma}[Equivalence of flux representation]\label{lem:equ-flux-rep}
Let $\mathcal{J}^{N,L} \in \cM(\hat{V}^{N,L}\times \hat{V}^{N,L})$. 
Suppose $\mathcal{J}^{N,L}(\d c, \d c') 
\ll \sum_{(k,l)} \delta_{c^{k,l-1}}(\d  c') \nu^L_{\bC^{N,L}}(\d c,k,l-1)$. Then
there exists a unique $\bJ^{N,L}$ on $\hat{E}^{N,L}$ such that
\[\mathcal{J}^{N,L}(\d c,\d c') = 
\sum_{k,l} \delta_{c^{k,l-1}}(\d c') \bJ^{N,L}(\d c,k,l-1).
\]
In particular, we can identify measures on $\hat{V}^{N,L}\times \hat{V}^{N,L}$ as measures on $\hat{E}^{N,L}$.
\end{lemma}
\begin{proof}
The assumption on absolute continuity of $\mathcal{J}^{N,L}$ implies its support is contained on a subset of $\hat{V}^{N,L}\times \hat{V}^{N,L}$ in which the two $c$ and $c'$ are connected by a jump vector. Therefore, we can characterise $\mathcal{J}^{N,L}$ by $\bJ^{N,L}(c,k,l-1):= \mathcal{J}^{N,L}(c,c^{k,l-1})$. Then, it is clear that the representation formula holds.
\end{proof}
\begin{remark}
With the identification in Lemma \ref{lem:equ-flux-rep} and under $\mathcal{Ent}(  \bJ^{N,L}| \mathbbm{\Theta}^L_{\bC^{N,L}})<+\infty$, we also have the finiteness of entropy of the associated measures on $\hat{V}^{N,L}\times \hat{V}^{N,L}$. In the proof below, we also cite the results in \cite{PRST22}, though not directly applicable in an irreversible setting, the strategy is very similar.
\end{remark}
\begin{proof}[Proof of Proposition \ref{prop:finitechainrule}]
\emph{Step 1: Chain rule for truncated entropy.} 

Due to the assumption $\cE^{N,L}(\bC_0^{N,L})< +\infty$,  $\bC^{N,L}_0 \ll \mathbbm{\Pi}^{N,L}$ holds. In addition, with \[\int_0^T \mathcal{R}^{N,L}_t (\bC^{N,L},\bJ^{N,L})\d t <+\infty,\] we have  for almost all $t\in [0,T]$,   $\bJ^{N,L}_t \ll \nu^L[\bC^{N,L}_t]$ and $ \bJ^{N,L}_t \ll {\nu^L}^{\dagger}[\bC^{N,L}_t]$. 
Therefore the statement of~\cite[Lemma 12.5, Lemma 12.7]{Hoeksema-thesis} (see also~\cite[Theorem 4.13]{PRST22}) implies $\bC^{N,L}_t \ll \mathbbm{\Pi}^{N,L}$ for almost all $t\in [0,T]$.
With the  truncated entropy function from the proof of Proposition \ref{prop:chain_rule} defined by $\phi_m(x):=\int_1^x \log_m(y) \d y$ and $\phi'_m(x)=\log_m (x) := \max \{ \min\{\log x,m \}, -m\}$, we have for $u_t:= \frac{\bC_t^{N,L}}{\mathbbm{\Pi}^{N,L}}$ is almost everywhere differentiable and the continuity equation holds pointwise   by \cite[Lemma 11.3, Lemma 12.5]{Hoeksema-thesis} (\cite[Corollary 4.14]{PRST22}), 
\[\begin{split}\int (\phi_m( u_t ) - \phi_m(u_s)) \d \mathbbm{\Pi}^{N,L} 
&=\int_s^t \int \phi_m'(u_r) \partial_r u_r  \d \mathbbm{\Pi}^{N,L} \d r \\
&=\int_s^t \int \phi_m'(u_r(c)) \DivL\bJ^{N,L}_r (\d c) \d r\\
&= \int_s^t \int \hat{\nabla}^L_{k,l-1} \phi_m'(u_r(c)) \bJ^{N,L}_r(\d c,k,l-1) \d r.
\end{split}\]
\noindent
\emph{Step 2: Convergence to chain rule.}
To take the limit $m \to \infty$, we need to show the integral on the right is uniformly bounded in $m$ by an integrable function.
By duality, we have the bound
\begin{equation}\label{eq:abs-gradient-wrt-flux}    
(2L)^{-1})\int |\hat{\nabla}^L  \phi_m'(u(c))|   \d \bJ^{N,L}(c,k,l-1) \le \mathcal{R}(\bC^{N,L}, \bJ^{N,L}) + \mathcal{R}^*(\bC^{N,L}, (2L)^{-1} | \hat{\nabla}^L \phi'_m (u)| ) \end{equation}
and 
\[\begin{split}
	\MoveEqLeft \mathcal{R}^*(\bC^{N,L}, (2L)^{-1} |\hat{\nabla}^L \phi'_m(u)|) = \int (e^{\frac1{2L} |\hat{\nabla}^L_{k,l-1} \phi_m' (u)(c) |}-1) \mathbbm{\Theta}^L [\bC^{N,L}](\d c,k,l-1) \\
&= \int   (e^{\frac1{2L} |\hat{\nabla}^L_{k,l-1} \phi_m' (u)(c) |}-1)  \sqrt{u(c) u(c^{k,l-1})} \mathbbm{\Theta}^L[\mathbbm{\Pi}^{N,L}](c,k,l-1)\\
& \le \int (e^{\frac1{2L} \hat{\nabla}^L_{k,l-1}\phi'(u)(c) } +e^{-\frac1{2L} \hat{\nabla}^L_{k,l-1} \phi'(u)(c)}) \sqrt{u(c) u(c^{k,l-1})} \mathbbm{\Theta}^L [\mathbbm{\Pi}^{N,L}](c,k,l-1)\\
&= \int u(c)) \d \nu^L[\mathbbm{\Pi}^{N,L}](c,k,l-1) +\int u(c^{k,l-1}) \d \flipL{\nu}[\mathbbm{\Pi}^{N,L}](c,k,l-1)\\
&= \int \nu^L [\bC^{N,L}]\d (c,k,l-1) + \flipL{\nu}[\bC^{N,L}] \d (c,k,l-1)\\
&\le 2 \frac{L}{L-1} C_K e^{\|b\|_\infty} \frac{N}{L}(\frac{N}{L}+1).
\end{split}
\]
Now the convergence on the left side of \eqref{eq:abs-gradient-wrt-flux} is due to the monotone convergence of  $\phi_m \to \phi$, while on the right side, we have $\nabla^L_{k,l-1}\phi'_m(u) \to \nabla^L_{k,l-1} \phi'(u)$ pointwise and by dominated convergence due to the upper bound above.

\smallskip
\noindent
\emph{Step 3: Non-negativity of EDF.} We begin with the Legendre-Fenchel inequality:
\[\begin{split}-\frac{1}{2L} \hat{\nabla}^{L}_{k,l-1}\phi'\biggl(\frac{\bC^{N,L}_r}{\mathbbm{\Pi}^{N,L}}(c)\!\biggr) \bJ^{N,L}_r(c,k,l-1)
\le \phi\Bigl(\bJ^{N,L}_r \Big| \mathbbm{\Theta}^L_r [\bC]\Bigr) + \mathbbm{\Theta}^L_r [\bC] \phi^{*} \biggl(\!\! -\frac1{2L} \hat{\nabla}^L_{k,l-1}\phi'\Bigl(\frac{\bC^{N,L}_r}{\mathbbm{\Pi}^{N,L}}(c)\Bigr)\!\biggr).
\end{split}\] 
Due to the fact that $\bJ^{N,L}_t \ll \mathbbm{\Theta}^L_t[\bC]$, $\bC^{N,L}_t \ll \mathbbm{\Pi}^{N,L}$ for a.e.~$t\in [0,T]$ and the uniform bound in $m$ for $\eqref{eq:abs-gradient-wrt-flux}$, the left-hand side is finite and coincides with  $-\frac1{2} \mathrm{B}\Bigl(\frac{\bC^{N,L}_r}{\mathbbm{\Pi}^{N,L}}(c), \frac{\bC^{N,L}_r}{\mathbbm{\Pi}^{N,L}} (c^{k,l-1}), \bJ^{N,L}_r(c,k,l-1)\Bigr)$. By the chain rule  \eqref{eq:EDF_FKE_chain_rule} from Step 2  and the definition of dissipation potentials, we get
\[ (\cE^{N,L}(\bC^{N,L}_s)-\cE^{N,L}(\bC^{N,L}_t))\le \int_s^t  \mathcal{R}^{N,L}_r (\bC^{N,L},\bJ^{N,L}) + \mathcal{D}^{N,L}_r (\bC^{N,L}) \d r\]
which implies ${\cL^{N,L}}(\bC^{N,L},\bJ^{N,L})\ge 0 $.
\end{proof}

\begin{proposition}[EDP solution for~\eqref{eq:FKE}]\label{prop:EDP-solution-and-finite-FKE-solution}
Under the assumption~\eqref{e:ass:K1}, let $\bC_0^{N,L}\in\cP(\hat V^{N,L})$ be given with $\cE^{N,L}(\bC_0^{N,L})< \infty$. Then
\begin{equation*}\begin{split}
\cL^{N,L} (\bC^{N,L}, \bJ^{N,L}) =0 \iff \begin{cases}
(\bC^{N,L}_t, \bJ^{N,L}_t) \in \ref{eq:cgce}(0,T)  \,, \\
  \bJ^{N,L}_t = \hat \bJ^{N,L}_t [\bC^{N,L}] \text{ for almost all } t \in [0,T] \,,  \\
\cE^{N,L}(\bC^{N,L}_T) +  \int^T_0 \mathcal{F}^{N,L}_t (\bC^{N,L}) \d t \le \cE^{N,L}(\bC_0^{N,L}) \,,
\end{cases}
\end{split}\end{equation*}
where the reference net flux $\hat \bJ^{N,L}_t[\bC^{N,L}](c,k,l)=\nu^L_t [\bC^{N,L}](c,k,l)$
and 
\[
	\mathcal{F}^{N,L}_t(\bC^{N,L}):=- \frac1{2L}\sum_{(c,k,l) \in \hat{E}^{N,L}} \hat{\nabla}^L_{k,l-1}\log\biggl(\frac{\bC^{N,L}_t}{\mathbbm{\Pi}^{N,L}}\biggr)(c) \; \hat \bJ^{N,L}_t[\bC^{N,L}](c,k,l-1) \,.
\]
\end{proposition}

\begin{proof}
This statement also follows from \cite[Theorem 12.3]{Hoeksema-thesis}. 
 
From $\cL^{N,L}(\bC^{N,L},\bJ^{N,L})=0$, the chain rule \eqref{eq:EDF_FKE_chain_rule} implies the inequality directly.   By Proposition \ref{prop:property-backwardkernel} \eqref{eq: lifted-backward} and $1$-homogeneouity of $\sfD$, we have
\begin{equation}\label{eq:equality-of-B}\begin{split}-&\int_0^T  \cF^{N,L}_t (\bC^{N,L}) \d t \\
=& \int_0^T \mathcal{Ent}(\bJ^{N,L}_t | \mathbbm{\Theta}^L_t [\bC^{N,L}]) + \sum_{(c,k,l)\in \hat{E}^{N,L}} \kappa^L_t [c](k,l-1)\mathbbm{\Pi}^{N,L}(c)\sfD_t \bigg(\frac{\bC^{N,L}}{\mathbbm{\Pi}^{N,L}}(c), \frac{\bC^{N,L}}{\mathbbm{\Pi}^{N,L}}(c^{k,l-1})\bigg)  \d t.  \end{split}\end{equation} Therefore the characterization of equality in Lemma \ref{lem:equ_in_D} implies 
\[\bJ^{N,L}_t = \hat \bJ^{N,L}_t [\bC^{N,L}] \text{ for almost all } t \in [0,T].\]
Conversely, we assume the three conditions on the right. The converse direction of Lemma \ref{lem:equ_in_D} gives the equality \eqref{eq:equality-of-B}. Now with the third assumption, we have $\cL^{N,L}(\bC^{N,L},\bJ^{N,L})\le 0$ and with \eqref{eq:EDF_FKE_non-negative} from the chain rule. The claim follows. 
\end{proof}

\subsection{The Liouville equation}

\begin{definition}[Smooth cylindrical function]
	A function  $f:\bR^\infty\to\bR$ is a \emph{smooth cylindrical test function} provided that $f = \psi\circ p_n$ where $p_n : \bR^\infty \to \bR^n$ is a suitable selection of coordinates and $\psi\in C^1_b(\bR^n)$, the set of functions $\bR^n \to \bR$, which are bounded, continuously differentiable with bounded derivative.
\end{definition}
\begin{definition}\label{def:liftCE}
The pair $(\bC, \bJ)$ solves the \emph{infinite particle continuity equation}, denoted by $(\bC, \bJ) \in \ref{eq:def:CEinfty}(0,T)$, provided that
\begin{enumerate}
\item $(\bC_t)_{t\in[0,T]}\subset \cP(\cP^1)$   is weakly continuous in duality with smooth cylindrical functions;
\item $(\bJ_t)_{t\in [0,T]} \subset  {\cM^+( \cP^1\times \bN\times\bN_0)}$ is a Borel measurable family such that
\[
	\int_0^T  \sum_{k,l \ge 1} \bJ_t(k,l-1)(\cP^1)  \dx t <\infty \, ;
\]
\item \label{cond:cty-weak} for all $0\le s \le t \le T$ and every smooth cylindrical function $\Phi$ the weak form holds
\begin{equation*}\label{eq:def:CEinfty} \tag{\ensuremath{\mathsf{CE}^\infty}}
\begin{split}
\int_{\cP^1} \Phi(c) (\bC_t-\bC_s) (\d c) = \int_s^t \int_{\cP^1} \sum_{k,l\geq 1} \ExNabla_{k,l-1}\!\nabla^\infty \Phi(c)\ \bJ_r(\dx c,k,l-1) \dx r \,.
\end{split}
\end{equation*}
\end{enumerate}
\end{definition}

\begin{remark}
In Lemma \ref{lem:uni_cts_C}, we will see that the limit pair $(\bC, \bJ)$ from the stochastic processes has better time regularity, $\bC$ is  continuous in duality with bounded and continuous functions with respect to $(\cP^1,\d_{\Exg})$.
\end{remark}
We state now a version of the superposition principle, which generalizes the result from~\cite{Ambrosio_Trevisan_14} towards our setting.
\begin{lemma}[Superposition principle]\label{lem:superpos}
Let $(\bC,\bJ)$ be such that
\begin{enumerate}[(a)]
\item $[0,T] \ni t \mapsto \bC(t) \in \cP(\bR^{\bN_0})$ is weakly continuous in duality with smooth cylindrical functions;
\item for each $t\in[0,T]$, $\bC(t)$ is concentrated on $\cP^1\subset \bR^{\bN_0}$;
\item $\bC(0)$ is concentrated on $\cP_{\!\le\rho}$ for some $\rho\in[0,\infty)$;
\item for each $t\in[0,T]$, $k,l\in \bN$, the measure $\bJ_{k,l-1}(t) \in \cM^+(\bR^{\bN_0})$ is absolutely continuous with respect to $\bC(t)$ and
\begin{equation*}\begin{split}
\ExDiv \frac{\d \bJ}{\d \bC}: [0,T]\times \bR^{\bN_0}  \to \bR^{\bN_0}
\quad\text{ with }\quad 
 (t,c)\mapsto \ExDiv \frac{\d \bJ(t)}{\d \bC(t)}(c) = - \sum_{k,l \in\bN} \gamma^{k,l}   \frac{\d \bJ_{k,l-1}(t)}{\d \bC(t)}(c) 
\end{split}\end{equation*} is a Borel vector-valued function;
\item for every  $F: \bR^{\bN_0} \to \bR$ smooth cylindrical and all $0\le s \le t \le T$ the weak form holds
\[ \int_{\cP^1} F(c)( \bC_t(\d c) - \bC_s (\d c)) = \int_s^t  \int_{\cP^1} \skp*{\nabla^\infty F(c),- \ExDiv\frac{\d \bJ_r}{\d \bC_r}(c)}_{\!\bN_0}\bC_r (\d c) \dx r \,;
\] 
\item the flux $\bJ$ satisfies the total variation bound
\begin{equation}\label{eq:TVbd}
\int_0^T \int  \sum_{k,l \in\bN}\frac{\d \bJ_{k,l-1}(r)}{\d \bC(r)}(c)  \, \bC_r (\d c) \dx r <\infty.
\end{equation}
\end{enumerate}
Then  there exists a Borel probability measure $\lambda$ on $C([0,T],\bR^{\bN_0})$ such that 
\begin{enumerate}[(i)]
\item $(e_t)_{\#}\lambda =\bC_t$ for all $t \in [0,T]$, with $e_t: C([0,T],\bR^{\bN_0}) \to \bR^{\bN_0}$ the evaluation map given by $e_t \gamma = \gamma_t$ for $\gamma \in C([0,T],\bR^{\bN_0})$;
\item $\lambda$ is concentrated on curves $\gamma \in \bigcup_{\rho_0 \le \rho}\AC([0,T], \cP_{\!\rho_0}, \d_{\Exg})$ with $(\gamma, \frac{\d \bJ}{\d \bC}(\gamma))\in \ExgCE(0,T)$ for $\lambda$-a.e. $\gamma$.
\end{enumerate}
\end{lemma}
 \begin{proof}
Under the assumptions (a), (d)--(f), we can apply the result of \cite[Theorem 7.1]{Ambrosio_Trevisan_14} to obtain a Borel probability measure $\lambda$ in $C([0,T];\bR^{\bN_0})$ satisfying $(e_t)_{\#} \lambda = \bC(t)$ for all $t \in [0,T]$. The measure $\lambda$ is concentrated on $\gamma \in \AC_{w}([0,T];\bR^{\bN_0})$ \cite[Section 7]{Ambrosio_Trevisan_14}, which satisfies for all sequences $f:\bN_0 \to \bR$  with finitely many non-zero terms, for all $0\le s\le t\le T$, the continuity equation
\begin{equation}\label{eq:lowerCE}
\pscal{ f, (\gamma(t) - \gamma (s))}_{\bN_0} = \int_s^t \pscal{f,-\ExDiv\biggl(\frac{\d \bJ(r)}{\d \bC(r)}(\gamma(r))\biggr) }_{\!\bN_0} \d r= \int_s^t \pscal{\ExNabla\! f,\frac{\d \bJ(r)}{\d \bC(r)}(\gamma(r))}_{\!\bN \times \bN_0} \d r.
\end{equation} 
We denote the set of curves $\gamma$  satisfying \eqref{eq:lowerCE} by $\SCE(\lambda,\bJ)$ so that $\lambda$ is concentrated on this set.
By superposition, we can express the total variation bound~\eqref{eq:TVbd} as
\[
	\int_{C([0,T],\bR^{\bN_0})} \int_0^T \sum_{k,l \in \bN}\frac{\d \bJ_{k,l-1}(r)}{\d \bC(r)}(\gamma(r)) \dx r \, \lambda (\d \gamma) <\infty \,.
\]
Hence, we have 
\begin{equation}
\label{eq:J_TVbd}\int_0^T \sum_{k,l \in \bN} \frac{\d \bJ_{k,l-1}(r)}{\d \bC(r)}(\gamma(r))  \dx r<\infty \qquad\text{ for $\lambda$-a.e. } \gamma \in C([0,T],\bR^{\bN_0}) \, . 
\end{equation} 
Note that by Lemma \ref{lm:bgrad} we can extend for each such $\gamma$, the test functions in~\eqref{eq:lowerCE} to the set $\{f:\bN_0 \to \bR: \sup_{k,l-1}| \ExNabla_{k,l-1} f| \le 1\}=\cF$ as given in  Definition~\ref{def:Xmetric}. Recall that for each $t\in [0,T]$, $\bC(t)$ is concentrated on $\cP^1$ and so also $(\gamma(r))_{r\in[0,T]}\subset \cP^1$. 

Moreover $\bC(0)$ is concentrated on $\cP_{\!\le \rho}$ so that $\gamma(0) \in \cP_{\!\rho_0}$ for some $\rho_0\le \rho$. 
By choosing the test function $(f_k)_{k\in\bN_0}=(k)_{k\in\bN_0} \in \cF$, we see that $\gamma(r)\in \cP_{\!\rho_0}$ for all $r\in[0,T]$. This allows us to apply Proposition~\ref{prop:closedformdcF} about the dual representation of $\d_\Exg$ in terms of the set $\cF$.  

The continuity is observed for $\lambda$-almost all  $\gamma \in \SCE(\lambda,\bJ)$ by taking the absolute value and supremum on the left of \eqref{eq:lowerCE} and estimating
\begin{equation}\label{eq:ac_d_f}\begin{split}
\MoveEqLeft\d_\Exg(\gamma(t),\gamma(s)) =\| \gamma(t) -\gamma(s)\|_{\cF}^* = \sup_{f\in \cF}\bigl|\pscal{ f, \gamma(t) - \gamma (s)}_{\bN_0}\bigr| \\
&=\sup_{f\in \cF}\bigg| \int_s^t \pscal{f,-\ExDiv\biggl(\frac{\d \bJ(r)}{\d \bC(r)}\bigl(\gamma(r)\bigr)\biggr) }_{\!\bN_0} \d r \bigg| \le \int_s^t \sum_{k,l \in \bN}  \frac{\d \bJ_{k,l-1}(r)}{\d \bC(r)}\bigl(\gamma(r)\bigr)  \dx r<+\infty.
\end{split}\end{equation}
By Proposition~\ref{prop:closedformdcF} , \eqref{eq:ac_d_f} and \eqref{eq:J_TVbd}, we see that if for $\lambda$ almost all $\gamma \in\SCE(\lambda,\bJ)$, then $\gamma \in \AC([0,T], \cP_{\!\rho_0}, \d_{\Exg})$ for some $\rho_0\le \rho$. Therefore  $(\gamma, \frac{\d \bJ}{\d \bC}(\gamma))\in \ExgCE(0,T)$.
\end{proof}

 \begin{theorem}\label{thm:super_in_dExg}
	If $(\bC,\bJ)\in  \ref{eq:def:CEinfty}(0,T)$, $\bC_0$ is concentrated on $\cP_{\!\le \rho}$ for some $\rho\in[0,\infty)$ and $\bJ_t (\cdot ,k,l-1) \dx t\ll \bC_t \dx t$ for all $k,l\geq 1$, then  there exists a Borel probability measure $\lambda$ on $C([0,T],\bR^{\bN_0})$ satisfying $(e_t)_{\#}\lambda =\bC_t$ for all $t \in [0,T]$ and  is  concentrated on curves $\gamma\in \bigcup_{\rho_0 \le \rho}\AC([0,T], \cP_{\!\rho_0}, \d_{\Exg})$such that $(\gamma, \frac{\d \bJ}{\d \bC}(\gamma))\in \ExgCE(0,T)$.  
\end{theorem}
\begin{proof}[Proof of Theorem~\ref{thm:super_in_dExg}]
 Recalling the definition  \ref{eq:def:CEinfty}(0,T), it is clear that the assumptions of Lemma \ref{lem:superpos} are satisfied. 
\end{proof}


 \begin{definition}[Energy dissipation function for the Liouville equation]\label{def:EDF4Lie}
 For $\rho \ge 0$, the energy dissipation functional for the Liouville equation is  
 \begin{equation*} \begin{split}
 \cL^{\rho\wedge \rho_c}(\bC,\bJ)&:=\begin{cases} \cE^{\rho\wedge \rho_c}(\bC_T) - \cE^{\rho\wedge \rho_c}(\bC_0)+ \int_0^T \cR_t (\bC , \bJ) + \cD_t(\bC)  \d t &\text{ if } (\bC, \bJ) \in \ref{eq:def:CEinfty}(0,T),\\
 +\infty &\text{ otherwise.}
 \end{cases}
\end{split}\end{equation*}
The energy and dual dissipation potentials are 
\begin{equation}\label{eq:def:Li:EnergyDualDissipation}
	 \cE^{\rho\wedge \rho_c}(\bC):=\int\sfE^{\rho\wedge \rho_c}(c) \, \bC(\d c)  +\frac12(\lambda_c-\lambda(\rho))_+ \left(\rho- \int \Mom(c) \bC(\d c)\right)\text{ and }
	\cD(\bC):=   \int\sfD(c) \, \bC(\d c) \,.
\end{equation}
\begin{remark}
In the definition of $\cE^{\rho\wedge \rho_c}$, with an abuse of notation, the dependence of $\rho$ in the second term is not made explicit becasue for those $\bC$ from Theroem \ref{thm:super_in_dExg}, we have  $\cE^{\rho\wedge \rho_c}(\bC_T)- \cE^{\rho\wedge \rho_c}(\bC_0) =\int\sfE^{\rho\wedge \rho_c}(c) \, \bC_T(\d c)   - \int\sfE^{\rho\wedge \rho_c}(c) \, \bC_0(\d c).$
\end{remark}
The primal dissipation potential is
\begin{equation}\label{eq:def:Li:PrimalDissipation}
	\cR(\bC, \bJ):= \sum_{k,l\in\bN}\int    \phi\biggl(\frac{\d \bJ(\cdot,k,l-1)}{\d \Sigma_{k,l-1}}(c) \bigg| \frac{\d \mathbbm{\Theta}[\bC](\cdot,k,l-1)}{\d \Sigma_{k,l-1}}(c) \biggr)\Sigma_{k,l-1}(\d c),\quad \mathbbm{\Theta}[\bC]= \sqrt{ \nu[\bC] \nu^{\dagger}[\bC]},\end{equation}
where $\Sigma_{k,l-1}$ is any common dominating measure for $\bC$ and $\bJ(\cdot,k,l-1)$, the expected forward flux $\nu[\bC] (\d c,k,l)= \bC(\d c) \kappa[c](k,l)$ and expected backward flux $\nu^{\dagger}[\bC] (\d c,k,l-1)= \bC(\d c) \kappa^{\dagger}[c](l,k-1)$. \end{definition}

\begin{lemma}[Consequences of superposition principle]\label{cor:connection-to-mean-field-equation} Under assumption~\eqref{e:ass:K1}, let   $\rho \ge 0 $, suppose
\begin{enumerate}[(a)]
\item $(\bC,\bJ)\in \ref{eq:def:CEinfty}(0,T)$,
\item $\bC_0$ is concentrated on $\cP_{\le\tilde \rho}$ for some $\tilde \rho\in [0,+\infty)$,    
\item $\int\sfE^{\rho\wedge \rho_c}(c) \, \bC_0(\d c) < +\infty$ ,
\item  $\cL^{\rho\wedge \rho_c}(\bC,\bJ)<+\infty$.
\end{enumerate}
 Then, 
 \begin{enumerate}[(i)]
 \item  there exists a Borel probability measure $\lambda$ on $C([0,T],\bR^{\bN_0})$ satisfying $(e_t)_{\#}\lambda =\bC_t$ for all $t \in [0,T]$ and $\lambda$ is concentrated on curves $\gamma \in \bigcup_{\rho_0 \le \rho}\AC([0,T], \cP_{\!\rho_0}, \d_{\Exg})$ such that $(\gamma, \frac{\d \bJ}{\d \bC}(\gamma))\in \ExgCE(0,T)$;
 \item the energy dissipation functional $\cL^\rho$ has the superposition 
\[
	\cL^{\rho\wedge \rho_c}(\bC ,\bJ) = \int  \sfL^{\rho\wedge \rho_c}\biggl(\gamma, \frac{\d \bJ}{\d \bC}(\gamma) \biggr) \lambda (\d \gamma)  \, ;
\]
\item the solution characterization holds
\begin{equation*}\begin{split}
\cL^{\rho\wedge \rho_c}(\bC ,\bJ)  =0 \iff  
\lambda \text{ concentrates on solutions of~\eqref{eq:MF} (Proposition~\ref{prop:zero-set-mean-field-EDP-functional}).}
\end{split}\end{equation*}
\end{enumerate}
\end{lemma}

\begin{proof}  
 By the definition of $\mathcal{L}^{\rho\wedge \rho_c}$, we have $\mathcal{L}^{\rho\wedge \rho_c}(\bC,\bJ)<+\infty$ implies $\bJ_{k,l-1}(t) \ll \bC(t)$ for all $k,l\in\bN$, $t\in [0,T]$.  Therefore, the integrability and continuity equation assumptions of $(\bC,\bJ) \in \mathsf{CE}^\infty(0,T)$ can be rewritten using $\frac{\d \bJ_{k,l-1}}{\d \bC}$ which implies the assumptions of the superposition principle on $\bJ$.  By  Theorem \ref{thm:super_in_dExg}, the statement (i) follows.
 
It remains to justify the two other claims. We note that with $\bJ(\cdot,k,l-1)\ll \bC$, 
\[\mathcal{R}(\bC,\bJ)= \int \mathsf{Ent}\bigg(\frac{\d \bJ}{\d \bC}(c), \theta[c]\bigg) \d \bC(c)  = \int \mathsf{R}\bigg(c,\frac{\d \bJ}{\d \bC}(c)\bigg)\d \bC(c).\] As a consequence of the superposition principle, it holds $\d t \lambda(\d \gamma)= \bC_t(\d c)\d t$. Moreover, due to Theroem \ref{thm:super_in_dExg} $\int \Mom(c) \bC_T(\d c)=\int \Mom(c) \bC_0(\d c)$ so that
\[\cE^{\rho\wedge \rho_c}(\bC_T) - \cE^{\rho\wedge \rho_c}(\bC_0) =\int\sfE^{\rho\wedge \rho_c}(c) \, \bC_T(\d c) - \int\sfE^{\rho\wedge \rho_c}(c) \, \bC_0(\d c). \]
 We have 
\[
	\cL^{\rho\wedge \rho_c}(\bC ,\bJ) = \int  \sfL^{\rho\wedge \rho_c}\biggl(\gamma, \frac{\d \bJ}{\d \bC}(\gamma) \biggr) \lambda (\d \gamma)  \,.
\]

Finally, with this disintegration of $\mathcal{L}^{\rho\wedge \rho_c}$, we use the characterization of $\mathsf{L}^{\rho\wedge \rho_c}$ in Proposition \ref{prop:zero-set-mean-field-EDP-functional} to get the one of $\mathcal{L}^{\rho\wedge \rho_c}$.
\end{proof} 
\begin{proposition}[EDP solution for the Liouville equation]\label{prop:EDP-solution-Li-solution} Under assumption of  \eqref{e:ass:K1}, let $\rho\ge 0$,  $\cE^{\rho\wedge \rho_c}(\bC_0)< +\infty$ and assume $\bC_0$ is concentrated on $\cP_{\le \rho_0}$ for some $\rho_0 \in [0,+\infty)$,   then
\begin{equation*}\begin{split}
\cL^{\rho\wedge \rho_c} (\bC, \bJ) =0 \iff \begin{cases}
(\bC, \bJ) \in \mathsf{CE}(0,T), \\
\bJ_{k,l-1}\ll\bC \text{ for all }k,l\in\bN \text{ and } \frac{\d \bJ}{\d \bC}(\gamma)= \overline\jmath[\gamma] \text{ for  $\lambda$-a.e. } \gamma, \\
\cE^{\rho\wedge \rho_c}(\bC_T) +  \int^T_0 \mathcal{F}_t(\bC) \d t \le \cE^{\rho\wedge \rho_c}(\bC_0),
\end{cases}
\end{split}\end{equation*}
with $\lambda$ from Lemma~\ref{cor:connection-to-mean-field-equation}, $ \overline\jmath_t[\gamma] = \kappa_t [\gamma]$, and 
\[\mathcal{F}_t(\bC):= \int  \mathsf{F}_t (c) \, \bC(\d c) . \] 
\end{proposition}
\begin{proof}
The idea is to use Proposition \ref{prop:zero-set-mean-field-EDP-functional} and the superposition principle of Lemma \ref{lem:superpos} which makes the representation~\eqref{eq:SuperPosEDF} rigorous and gives using the definitions~\eqref{eq:def:Li:EnergyDualDissipation} and~\eqref{eq:def:Li:PrimalDissipation} the identity
\begin{equation*}
\cL^{\rho}(\bC,\bJ)
=\int \biggr[\sfE^\rho(\gamma(T))-\sfE^{\rho}(\gamma(0))+ \int^T_0 \biggl[ \sfR_t\biggl(\gamma, \frac{\d \bJ}{\d \bC}(\gamma)\biggr)+ \sfD_t(\gamma)\biggr] \dx t \biggl] \lambda(\d \gamma)
\end{equation*} 
provided $\rho\in[0,\rho_c]$ with $\rho<\infty$.
Given $\cL^{\rho}(\bC,\bJ)=0$, $\cE^\rho(\bC_0)< +\infty$ and $\bC_0$ is concentrated on $\cP_{\le \rho_0}$ for some $\rho_0 \in [0,+\infty)$, Corollary \ref{cor:connection-to-mean-field-equation} implies the assumption of Proposition \ref{prop:chain_rule} is satisfied pointwise $\lambda$ almost everywhere. Therefore we have  for $\lambda$-a.e. $\gamma$ the identities
\[
	\sfE^\rho(\gamma(T)) -\sfE^\rho(\gamma(0))= \int_0^T \sum_{k,l \in \bN}\frac12 \rmB\biggl(\kappa_t[\gamma](k,l-1),\kappa_t^\dagger[\gamma](l,k-1),\frac{\d \bJ_t}{\d \bC_t}(\gamma(t),k,l-1)\biggr) \dx t \,,
\]
and  
 \begin{equation*}\begin{split}
\MoveEqLeft\int^T_0 \biggl[ \sfR_t\biggl(\gamma, \frac{\d \bJ}{\d \bC}\bigl(\gamma\bigr)\biggr)+\sfD_t(\gamma) \biggr] \d t\\
&= - \int_0^T \sum_{k, l\in \bN }\frac12 \rmB_t\biggl(\kappa[\gamma](k,l-1),\kappa_t^\dagger [\gamma](l,k-1),\frac{\d \bJ_t}{\d \bC_t}(\gamma(t),k,l-1)\biggr) \dx t \,.
\end{split}\end{equation*}
By applying Lemma~\ref{lem:equ_in_D}  with $u =\kappa_t[\gamma](k,l-1) ,v=\kappa^\dagger_t[\gamma](l,k-1)$, we have \[\frac{\d \bJ_t}{\d \bC_t}(\gamma(t))= \overline\jmath_t[\gamma] \text{ for Lebesgue almost every } t \text{ and }\lambda \text{ almost every }\gamma.\]
By the disintegration of $\lambda$ corresponding to the map $e_t$ from $(e_t)_{\#}\lambda =\bC_t$, we conclude
\[0=\cL^{\rho}(\bC,\bJ)= \cE^\rho(\bC_T)-\cE^\rho(\bC_0) + \int_0^T \cF_t(\bC) \dx t.\] 
  Conversely, if the three conditions hold, then the assumptions of Theorem \ref{thm:super_in_dExg} are satisfied, so there exists a measure $\lambda$ with
\begin{equation*}\begin{split}
  \cL^{\rho}(\bC,\bJ)&= \cE^\rho(\bC_T) - \cE^\rho(\bC_0)+\int\int_0^T  \bigl[\sfR_t(\gamma,\overline\jmath[\gamma]  )  + \sfD_t(\gamma)\bigr] \dx t \, \lambda(\d \gamma ).
\end{split}\end{equation*} 
By the converse statement of Lemma \ref{lem:equ_in_D}, we get the identity
\[ 
	\int \int_0^T  \bigl( \sfR_t(\gamma,\overline\jmath[\gamma]  )  + \sfD_t(\gamma) \bigr) \d t  \lambda(\d \gamma )
	= \int \int_0^T \mathsf{F}_t(\gamma) \dx t \; \lambda(\d \gamma ) \,.  
\]
Using the disintegration of $\lambda$, we have
\begin{equation*}\begin{split}
  \cL^{\rho}(\bC,\bJ)&= \cE^\rho(\bC_T) - \cE^\rho(\bC_0)+\int \int_0^T  \mathsf{F}_t(\gamma)  \dx t\, \lambda(\d \gamma )\\
  &= \cE^\rho(\bC_T)  -\cE^\rho(\bC_0) + \int_0^T  \int  \mathsf{F}_t(c) \, \bC_t(\d c) \dx t\\
  & =\cE^\rho(\bC_T)  -\cE^\rho(\bC_0) + \int_0^T \cF_t(\bC) \dx t \le 0 \,.
  \end{split}\end{equation*} 
Since $\cL^{\rho}(\bC,\bJ) = \int \sfL^{\rho}(\gamma,\overline\jmath[\gamma] ) \lambda(\d \gamma )\le 0$, the assumptions of the converse of  Proposition~\ref{prop:zero-set-mean-field-EDP-functional} hold for $\lambda$-almost all $\gamma$. This implies $\sfL^{\rho}(\gamma,\overline\jmath[\gamma] )=0 $ for $\lambda$-almost every $\gamma$, so $ \cL^{\rho}(\bC,\bJ)=0$.
\end{proof}

\section{Convergence}
\subsection{Compactness of curves}

\begin{proposition}[Compactness of $\bC^{N,L}$]\label{prop:uniform_bbd_C}
Let $(\bC^{N,L},\bJ^{N,L})\in \ref{eq:cgce}(0,T)$  then  the family $\bigl(\bC^{N,L}_t \in \cP(\cP_{\frac{N}{L}})\bigr)_{N/L \to \rho, t \in [0,T]}$ is narrowly sequentially precompact.
\end{proposition}
\begin{proof}
Since $N/L \to \rho$ and $\bC^{N,L}\in \cP_{N/L}$,  the sequence $(\bC^{N,L})_{N/L\to \rho}$ is supported on the set of probability measure with first moment less than $\overline{\rho}$, which is a precompact set in narrow topology by Prokhorov's theorem. So again by Prokhorov's theorem, the sequence of measure $(\bC^{N,L})_{N/L\to \rho}$ is a precompact set in the narrow topology. Since $\cP(\bN_0)$ with narrow topology is a Polish space, we have the sequential precompactness.

\end{proof}
\begin{lemma}\label{lem:concen_on_rho_mean} If  $\bC^{N,L} \weakto  \bC$, then
\begin{enumerate}[(i)]
\item$\rho=\liminf_{N/L \to \rho} \bC^{N,L}(\Mom_1)\ge \bC(\Mom_1)$ and 
\item the probability measure $\bC$ is concentrated on $\cP_{\le \!\rho}$. 
\end{enumerate}
\end{lemma} 
\begin{proof}
By Fatou's Lemma, $\Mom_1$ is narrowly-lower-semicontinuous, so the Portmanteau theorem gives the first statement.
Note that for any $\rho_0 \ge 0$ $\cP_{\le \rho_0}$ is a narrowly closed subset of $\cP$. Thus, by the characterization of narrow convergence, 
\[\limsup_{N/L\to \rho} \bC^{N,L}(\cP_{\le \rho_0}) \le \bC(\cP_{\le \rho_0})\]
Hence for all $\epsilon >0$, 
\[1=\limsup_{N/L\to \rho} \bC^{N,L}(\cP_{\le \rho+\epsilon}) \le \bC(\cP_{\le \rho+\epsilon}).\]
 By the continuity from the above property of a measure
 \[
 	\bC(\cP_{\le \rho}) =\lim_{\epsilon\to 0}\bC(\cP_{\le \rho+\epsilon})=1 \,. \qedhere
 \]
\end{proof}

 \subsection{Compactness of fluxes}\label{subsec:compact-flux}
In this section, we  assume \begin{equation}\label{eq:finite_edf}
	\sup_{N/L \to \rho} \int_0^T  \cR^{N,L}_t (\bC^{N,L},\bJ^{N,L}) \dx t  =: C_{\cR} <+\infty \,.
\end{equation} 
\begin{proposition}[Narrow compactness of flux]\label{prop:comp_j}
	Under the assumption \eqref{e:ass:K2}, suppose there exists $C_{\cR}<\infty$ such that the sequence $((\bC^{N,L}_t,\bJ_t^{N,L})\in \ref{eq:cgce}(0,T))_{N/L\to \rho}$.  Then along a subsequence of the measures $  \bJ^{N,L}_t(\d c,k,l-1)$ on $[0,T]\times \cP_{\le \overline{\rho}}$, there exists a limit 
\[
\bJ^{N,L}_t(\d c,k,l-1) \dx t \to \bJ_t(\d c,k,l-1)\dx t  \qquad\text{ narrowly in $\cM^+ \bigl(\cP_{\le \overline{\rho}}\times \bN \times \bN_0\times [s,t]  \bigr)$}
\] and $(\bJ_t(\d c, k,l-1))_{t\in[0,T]}$ is a Borel measurable family.
\end{proposition} 
\begin{lemma}\label{lem: kappa-C-b}
Under $\eqref{e:ass:K2}$, the function $c\mapsto \sum_{k,l\in\bN} \kappa[c](k,l-1)$ is narrowly continuous and bounded on $\cP_{\le \rho}$ for $\rho\ge 0$.
\end{lemma}
\begin{proof}
The boundedness is clear because of $\eqref{e:ass:K2}$ and the first moment on $\cP_{\le \rho}$ is bounded by $\rho$. We have the uniform bound
\[\sum_{k,l\in\bN}\kappa[c](k,l-1)\le C e^{\|b\|_\infty} \rho(\rho+1).\]
For the continuity, the only interesting case is for $\rho>0$ since $\cP_{\le 0}=\{\delta_0\}$.
The idea is to apply \cite[Theorem 2.8.8]{bogachev2007measure}. Assume $c^n \to c$ in $\cP_{\le \rho}$. Note that for each $n\in\bN$, 
\[0\le  \kappa[c^n](k,l-1) \le e^{\|b\|_\infty} m(k)m(l) c^n_k c^n_{l-1}\] 
and by narrowly convergence $\lim_{n\to\infty} c^n_k =c_k$ \text{ for each } $k \in \bN_0$, we have \[\lim_{n\to\infty} \kappa[c^n](k,l-1) =  \kappa[c](k,l-1)\] and \[\lim_{n\to\infty}e^{\|b\|_\infty} m(k)m(l) c^n_k c^n_{l-1}=e^{\|b\|_\infty} m(k)m(l) c_k c_{l-1}\] for each $k,l\in\bN$. 
We claim $\lim_{n\to \infty} \sum_{k \in\bN} m(k) c_k^n = \sum_{k \in \bN} m(k) c_k.$ This follows from the sublinearity of $m$ and $c^n,c\in\cP_{\le \rho}$.
Indeed, for any $\epsilon > 0$, for $N_0$ such that for all $N\ge N_0$ $\frac{m(N)}{N}\le \frac{\epsilon}{4\rho}$, there exists $n_0$ such that for all $n\ge n_0$
\[
	\sum_{k=1}^{N_0} m(k) |c_k^n - c_k| < \epsilon /2 \,.
\] 
This implies 
\begin{align*}
	\abs[\bigg]{\sum_{k\in\bN} m(k) (c_k^n - c_k)}
		&\le\sum_{k=1}^{N_0} m(k)|c_k^n - c_k| + \frac{m(N_0)}{N_0}\sum_{k={N_0}+1}^\infty  k |c_k^n - c_k|\\
		&\le \epsilon/2 + \frac{m(N_0)}{N_0}\sum_{k=N_0+1}^\infty  k (c_k^n + c_k)\\
		&\le  \epsilon/2 + \frac{m(N_0)}{N_0}(2\rho)\le \epsilon.
\end{align*}
Since $\epsilon>0$ is arbitrary, the claim is justified. By \cite[Theorem 2.8.8]{bogachev2007measure}, we conclude  
\[
	\lim_{n\to\infty}\sum_{k,l\in \bN} \kappa[c^n](k,l-1)= \sum_{k,l\in \bN} \kappa[c](k,l-1) \,. \qedhere
\]
\end{proof}

\begin{proof}[Proof of Proposition \ref{prop:comp_j}]
\emph{Step 1: Narrow compactness of reference flux.}
From the narrow compactness of $(\bC^{N,L})_{N/L \to \rho}$ on $\cP_{\le \overline{\rho}}$, without relabelling the subsequence, we have $\bC^{N,L} \weakto \bC$. We show convergence of the reference fluxes 
\[\nu^L_t [\bC^{N,L}](\d c, k,l-1) = \kappa_t^L[c](k,l-1)\bC^{N,L}_t(\d c) \weakto \kappa_t[c](k,l-1) \bC_t(\d c) = \nu_t[\bC](\d c,k,l-1)  \] narrowly in $\cM^+(\cP_{\le \overline{\rho}} \times \bN \times \bN_0)$ for each $t\in [0,T]$. Since for $c \in \hat{V}^{N,L}$,
\[\begin{split} \sum_{k,l \in \bN} | \kappa^L[c](k,l-1) - \kappa[c](k,l-1) |&= \sum_{k,l \in \bN} \bigg| \frac{1}{L-1}K(k,l-1) c(k) \Big( c(l-1)-\frac1L \delta_{k,l-1} \Big)\bigg| \\
&\le   \frac{1}{L-1} \sum_{k,l \in \bN} K(k,l-1) c(k) c(l-1)\\
& \le C e^{\|b\|_\infty}\frac{1}{L-1}\frac{N}{L}\Big(\frac{N}{L}+1\Big)
\end{split}\]
and the map $c \mapsto \sum_{k,l\in \bN}\kappa_t[c](k,l-1) $ is continuous in $\cP_{\le \rho}$ and bounded by Lemma \ref{lem: kappa-C-b}, it can be seen readily that for $f \in C_b(\cP_{\le \overline{\rho}}\times \bN\times \bN_0)$
\[\int \sum_{k,l\in\bN} f(c,k,l-1)\kappa^L_t [c](k,l-1) \bC^{N,L}_t(\d c) \to \int \sum_{k,l\in\bN} f(c,k,l-1) \kappa_t [c](k,l-1) \bC_t(\d c)\] for each $t \in [0,T]$.  Similarly, one can show that $\flipL{\nu_t}[\bC^{N,L}] \weakto \nu^{\dagger}_t[\bC]$ in $\cM^+(\cP_{\le \overline{\rho}} \times \bN \times \bN_0)$.
 
\smallskip
\noindent
\emph{Step 2: Narrow compactness of fluxes.}
To show the narrow compactness of $\bJ^{N,L}_t$, we use that for $K$ a sequentially compact set with respect to narrow convergence, the set 
\[ \{ \eta: \mathcal{Ent}( \eta |\sqrt{\mu \nu}) \le L, \mu, \nu \in K\}\] is also narrowly sequential compact by \cite[Lemma 11.2 (4)]{Hoeksema-thesis}. 

By Step 1, the set of reference flux measures $\bigl(\nu^L_t[\bC^{N,L}] , \flipL{\nu_t}[\bC^{N,L}] \bigr)_{N/L \to \rho}$ is sequentially compact in $\cM(\cP_{\le \overline{\rho}} \times \bN \times \bN_0)$. We now justify that $\sqrt{\nu^L_t [\bC^{N,L}] \flipL{\nu_t}[\bC^{N,L}] }$ is narrowly sequentially compact in $\cM(\cP_{\le \overline{\rho}} \times \bN \times \bN_0\times [s,t])$. One characterization of narrow convergence is $\lim_{N/L\to \rho} \nu^L_t[\bC^{N,L}](A) = \nu_t[\bC](A)$ for all continuity set $A$. 
With this characterization and the continuity of square root, the narrow convergence of $(\nu^L_t[\bC^{N,L}] , \flipL{\nu_t}[\bC^{N,L}] )_{N/L \to \rho}$ implies narrow convergence of $\sqrt{\nu^L_t [\bC^{N,L}] \flipL{\nu_t}[\bC^{N,L}] }$ in  $\cM(\cP_{\le \overline{\rho}} \times \bN \times \bN_0)$ for each $t\in[0,T]$. 

Let $f$ be continuous bounded on $\cP_{\le \overline{\rho}} \times \bN \times \bN_0\times E$, where $E \subset [0,T]$ a measurable subset, then \[ \int_E \int f \sqrt{\nu^L_t [\bC^{N,L}] \flipL{\nu_t}[\bC^{N,L}] } \d t \to \int_E \int f \sqrt{\nu_t[\bC] \nu^{\dagger}_t[\bC]} \d t \] due to dominated convergence theorem applied on the time integral. This gives the narrow convergence of $\bigl(\sqrt{\nu^L_t [\bC^{N,L}] \flipL{\nu_t}[\bC^{N,L}] }\d t\bigr)_{N/L\to\rho}$. 
Hence $\bigl(\sqrt{\nu^L_t [\bC^{N,L}] \flipL{\nu_t}[\bC^{N,L}] }\dx t\bigr)_{N/L \to \rho}$ is sequentially compact in $\cM(\cP_{\le \overline{\rho}}\times \bN \times \bN_0 \times E)$.

From \eqref{eq:finite_edf}, we have 
\[ \int_0^T \mathcal{Ent}\bra*{ \bJ^{N,L}_t \,\middle|\, \sqrt{\nu^L_t [\bC^{N,L}] \flipL{\nu_t}[\bC^{N,L}]}} \d  t \le C_{\cR}\]
Again by \cite[Lemma 11.2 (4)]{Hoeksema-thesis}, this implies $\bJ^{N,L}_t \d t \weakto \bJ \in \cM(\cP_{\le \overline{\rho}}\times \bN \times \bN_0 \times E)$ for any measurable $E\subset [0,T]$ along a subsequence.

\smallskip
\noindent
\emph{Step 3: Absolute continuity in time of the limit flux.}
Since $\bJ^{N,L}_t \d t \weakto \bJ \in \cM(\cP_{\le \overline{\rho}}\times \bN \times \bN_0 \times E)$ for any $E\subset [0,T]$ measurable subset, we can take $1_{E} f$ as the test function with $f\in C_b(\cP_{\le \overline{\rho}}\times \bN \times \bN_0 )$ and $E$ with Lebesgue measure zero. Then we have $0 =\int_E \int_{\cP_{\le \overline{\rho}}\times \bN \times \bN_0 }   \bJ^{N,L}_t  \d t \to \int_{\cP_{\le \overline{\rho}}\times \bN \times \bN_0  \times E}  \bJ  $. This shows that $\bJ$ is absolutely continuous with respect to Lebesgue. Thus by Radon-Nikodym theorem, there exists a measurable family $(\bJ_t)_{t\in[0,T]} \subset  \cM(\cP_{\le \overline{\rho}}\times \bN \times \bN_0)$ such that  $\bJ = \bJ_t\dx t$.
\end{proof}

\subsection{Convergence towards the infinite particle continuity equation}\label{subsec:Limit-CE-compat}
This section aims to verify that the limit measure satisfies the infinite particle continuity equation~\eqref{eq:def:CEinfty} in the sense of Definition \ref{def:liftCE}. 
\begin{proposition}\label{prop:limit satisfies CEinf}
If $(\bC^{N,L}, \bJ^{N,L})\in \ref{eq:cgce}(0,T)$ converges to $(\bC,\bJ)$ , then $(\bC,\bJ) \in  \ref{eq:def:CEinfty}(0,T)$.
\end{proposition}
\begin{proof}
The conditions in Definition \ref{def:liftCE} for $(\bC,\bJ)\in\ref{eq:def:CEinfty}(0,T)$ are verified below, in detail:
	\begin{enumerate}[(1)]
		\item From Lemma \ref{lem:concen_on_rho_mean}, we have $(\bC_t)_{t \in [0,T]} \in \cP_{\le \rho}$. The continuity of $t\mapsto \bC_t$ is obtained in Lemma~\ref{lem:uni_cts_C} below. 
		\item  By Lemma~\ref{lem:mod_of_uni_int}~\eqref{eq:Jbd} below, we have 
		\[ 
			\int_0^T \sum_{k,l\geq 1}\bJ_t(k,l-1)(\cP_{\le \overline{\rho}}) \dx t  \le \sup_{N/L \to \rho}  \int_0^T \int_{\EhNL}  \bJ^{N,L}_t (\d e) \dx t <+\infty \,.
		\]
		The  Borel measurability follows from Proposition~\ref{prop:comp_j}.
		\item In Lemma~\ref{lem:limit_sat_CE} below, we prove that $(\bC,\bJ)$ satisfies the infinite particle continuity equation~\eqref{eq:def:CEinfty}.\qedhere
	\end{enumerate}
\end{proof}
\begin{lemma}[Modulus of uniform integrability]\label{lem:mod_of_uni_int} 
	Assume  $\bigl((\bC^{N,L}_t,\bJ_t^{N,L})\in \ref{eq:cgce}(0,T)\bigr)_{N/L\to \rho}$. Then the family of fluxes $\bigl\{\bbJ^{N,L}\bigr\}$ is equi-integrable in time, i.e. there exists a non-decreasing continuous function $w :[0,\infty) \to [0,\infty)$ and ${\lim_{r \to 0}w(r)=0}$
such that \begin{equation}\label{eq:supJ-bbd-w}
\sup_{N/L\to \rho} \bigl\|\bJ^{N,L}_t\d t \bigr\|_{\TV(I \times \hat{E}^{N,L})} \le w(\abs{I}) \qquad\text{for all measurable } I\subseteq [0,T] \,.
\end{equation} 
In particular, 
	\begin{equation}\label{eq:Jbd} 
		\sup_{N/L \to \rho}  \int_0^T \int_{\EhNL}  \bJ^{N,L}_t (\d e) \dx t \le C_\cR +C_{\overline{\rho},K} T  \phi^*(1) <+\infty
	\end{equation}
	holds, where $ C_{\overline{\rho}, K} =2 C_{\overline{K}} e^{\|b\|_\infty}\overline{\rho}(\overline{\rho}+1)$ and $C_{\overline{K}}$ as in~\eqref{e:ass:K1}.
\end{lemma}
\begin{proof}
For any $M>0$, we use the dual formulation of the total variation norm and the Legendre duality to get
\begin{align*}
 \int_I \int_{\hat{E}^{N,L}} M \xi(t,e)  \bJ^{N,L}_t (\d e) \d t &\le \int_I  \mathcal{R}^{N,L}_t (\bC^{N,L},\bJ^{N,L}) + {\mathcal{R}^{N,L}_t}^*(\bC^{N,L}, M \xi) \d t \\ &\le C_{\mathcal{R}} +  \int_{I} {\mathcal{R}^{N,L}_t}^* (\bC^{N,L}, M \xi) \d t. \end{align*}
The second term can also be bounded, for any $|\xi|\le 1$
 \begin{align*}
 \int_{I} {\mathcal{R}^{N,L}_t}^*(\bC^{N,L}, M \xi) \dx t
 &= 
  \int_{I}\int_{\hat{E}^{N,L}}   \mathbbm{\Theta}_t [\bC^{N,L}](e) \phi^*(M \xi(t,e)) \dx e \dx t\\
 &\le \sup_{t\in[0,T]} \|\mathbbm{\Theta}_t [\bC^{N,L}]\|_{\ell^1(\hat{E}^{N,L})}\int_I \phi^*(M ) \dx t \,,
\end{align*} where the monotonicity of $\phi^*$ was used. The supremum in time is finite due to the uniform bound 
\[ \begin{split}
\sup_{t\in [0,T]}\|\mathbbm{\Theta}_t [\bC^{N,L}]\|_{\ell^1(\hat{E}^{N,L})}&=\sup_{t\in[0,T]} \sum_{(c,k,l) \in \hat{E}^{N,L}} \sqrt{\nu^L_t [\bC^{N,L}](c,k,l-1) {\nu^{L}_t}^\dagger[\bC^{N,L}] (c,k,l-1)} \\
&\le\sup_{t\in[0,T]}  \sum_{(c,k,l)\in \hat{E}^{N,L}} \nu^L_t[\bC^{N,L}](c,k,l-1)\\
& \le  2 C_{\overline{K}} e^{\|b\|_{\infty}}  \sup_{N/L\to \rho}\frac{N}{L} \biggl(\frac{N}{L}+1\biggr). \end{split}\]
We obtain $\eqref{eq:Jbd}$ by setting $M=\xi =1$, $I=[0,T]$.

For $|I|=s$, we take $M_s$ such that $M_s =\max \{M: \phi^{*}(M) \le 1/s \}$, or explicitly $M_s=\log (1+1/s)$ then $\int_I \phi^*(M_s) \d t \le 1$. Then
\begin{align*}
	\int_I \xi(t,e)  \bJ^{N,L}_t (\d e) \dx t 
	\le \frac{1}{M_s} \bra[\bigg]{C_{\mathcal{R}} + \sup_{N/L \to \rho}\sup_{t\in[0,T]} \| \mathbbm{\Theta}_t [\bC^{N,L}]\|_{\ell^1}} \,.
\end{align*}
If we define $w(s):=\frac{1}{M_s} (C_{\mathcal{R}} +\sup_{N/L \to \rho} \sup_{t\in[0,T]} \| \mathbbm{\Theta}_t[\bC^{N,L}]\|_{\ell^1})$, then taking the supremum over $|\xi|\le 1$ and $N/L\to \rho$, we have
\[
	\sup_{N/L\to \rho} \| \bJ^{N,L}_t \d t\|_{\TV(I \times \hat{E}^{N,L})} \le w(|I|) \,. \qedhere
\]
\end{proof}

\begin{lemma}\label{lem:limit_sat_CE}
If $(\bC,\bJ)$ is a limit of $(\bC^{N,L},\bJ^{N,L})\in \ref{eq:cgce}(0,T)$, then the limit satisfies  the weak form of~\eqref{eq:def:CEinfty} in Definition~\ref{def:liftCE}.
\end{lemma}
\begin{proof}
Let $\Phi$ be a smooth cylindrical function and $s,t$ such that $0\le s \le t \le T$. By the definition of the convergence, we can pass to the limit on the left-hand side
\[   \lim_{\substack{N,L\to \infty\\N/L\to\rho }} \int_{\cP_{\le\overline{\rho}}}\Phi(c)  (\bC^{N,L}_t-\bC^{N,L}_s)(\d c) = \int_{\cP_{\le\overline{\rho}}}\Phi(c) (\bC_t-\bC_s)(\d c). \]
On the other hand,  we want to pass to the limit in the flux. On $\cP$, we can consider $\Phi$ as a smooth function on $[0,1]^n$ for some $n\in \bN$ so that $\nabla^\infty \Phi$ is uniformly continuous and we have 
\begin{equation*}\begin{split}
 \sup_{c\in \cP_{\le\overline{\rho}}} \sup_{k,l\geq 1}\bigl| \bigl( \Phi(c+L^{-1}\gamma^{k,l-1})-\Phi(c)\bigr)L - \gamma^{k,l-1} \cdot \nabla^\infty \Phi(c)\bigr| = o(1)_{L\to \infty} \,.
\end{split}\end{equation*}
Therefore, we can estimate
\begin{align*}
	\MoveEqLeft \biggl | \int_s^t  \int_{\cP_{\le\overline{\rho}}} \sum_{k,l\geq 1} \bigl(\Phi( c+L^{-1} \gamma^{k,l-1})-\Phi(c)\bigr) L \;  \bJ^{N,L}_r(\d c,k,l-1)  \dx r \\
	&\quad -\int_s^t \int_{\cP_{\le\overline{\rho}}} \sum_{k,l\geq 1} \ExNabla_{k,l-1} \nabla^\infty \Phi(c)\, \bJ_r(\d c,k,l-1) \dx r  \biggr|\\ 
&\le o(1) \sup_{N/L\to \rho}\bigl\lVert \bJ_t^{N,L} \d t\bigr\rVert_{\TV}+ \int_s^t \!\!\int_{\cP_{\le\overline{\rho}}} \sum_{k,l\geq 1}  \ExNabla_{k,l-1} \nabla^\infty \Phi(c)\bigl(\bJ^{N,L}_r- \bJ_r\bigr)(\d c,k,l-1)\dx r. 
\end{align*}
Since  $\overline{\nabla}\nabla^\infty \Phi \in C_b(\cP_{\le \overline{\rho}}\times \bN \times \bN_0)$, the integral in the second term converges to zero by convergence of $\bJ^{N,L}_r \d r \to \bJ_r \d r$ from Proposition  \ref{prop:comp_j}.  In summary, passing to the limit in $N/L \to \rho$, we get for  $\Phi$ satisfying the assumptions
\[
	\int_{\cP_{\le\overline{\rho}}} \Phi(c) (\bC_t-\bC_s)(\d c) 
	= \int_s^t\int_{\cP_{\le\overline{\rho}}} \sum_{k,l\geq 1} \ExNabla_{k,l-1} \nabla^\infty \Phi(c) \, \bJ_r(\d c,k,l-1) \dx r   \,. \qedhere
\]
\end{proof}
 
\begin{lemma}[Refinement of compactness and absolute continuity of limit curve]\label{lem:uni_cts_C}
Given $(\bC^{N,L},\bJ^{N,L})\in \ref{eq:cgce}(0,T)$ with $(\bJ^{N,L}_t)$ having a uniform modulus of continuity~\eqref{eq:supJ-bbd-w}.   
Then the measure-valued maps $( \bC^{N,L}:[0,T] \to \cP(\cP_{\le \overline{\rho}}))_{N/L \to \rho}$ have a subsequence converging to a limit $\bC:[0,T]\to \cP(\cP_{\le \rho})$ such that $\bC \in \AC([0,T];\cP(\cP_{\le \rho},\d_{\Exg}))$.
\end{lemma}
\begin{proof}
We apply the refined version of Arzela-Ascoli in \cite[Proposition 3.3.1]{ambrosio2008gradient}. First we justify that the narrow topology induced by the weak convergence on $\cP_{\le \rho}$ is a Hausdorff topology on~$\cP(\cP_{\le \rho})$ compatible with the narrow topology induced by $\d_\Exg$ on  $\cP_{\le \rho}$. Because $\cP_{\le \rho}$ with weak topology and  $(\cP_{ \le \rho }, \d_\Exg)$ are separable, they are metrizable. Note that the narrow topology induced by weak convergence on $\cP_{\le \rho}$ is weaker than the narrow topology induced by $\d_\Exg$.  Now assume $(\bC^n , \mathbb{B}^n)\weakto (\bC, \mathbb{B}) $ in the weak sense, we want to show $\liminf_{n\to\infty}\d_{\mathup{BL}}(\bC^n,\mathbb{B}^n) \ge \d_{\mathup{BL}}(\bC,\mathbb{B})$ for the bounded Lipschitz distance, which thanks to the Kantorovich-Rubinstein norm metrizes the narrow topology~\cite[Theorem 8.3.2]{bogachev2007measure}.
It is defined for $\mathcal{F}_{\mathup{BL}}:=\bigl\{ \Psi  \text{ is $1$-$\d_{\Exg}$-Lipschitz and }\|\Psi\|_{\infty}\le 1\bigr\}$ by 
\begin{align*}
\ d_{\mathup{BL}}\bra*{\bC^n , \mathbb{B}^n} &:= \sup_{\Psi\in\mathcal{F}_{\mathup{BL}}} \biggl\{ \int_{\cP_{\le\overline{\rho}}} \Psi(c) \, (\bC^n -\mathbb{B}^n) (\dx c) \biggr\} \\
&\ge \int_{\cP_{\le\overline{\rho}}} \Psi(c) \, (\bC^n -\bC + \mathbb{B}- \mathbb{B}^n ) (\dx c) +  \int_{\cP_{\le\overline{\rho}}} \Psi(c) \ (\bC-\mathbb{B})(\d c)
  \text{ for each } \Psi\in \mathcal{F}_{\mathup{BL}} .
 \end{align*} 
Now, let $\mathcal{F}^1_{\mathup{BL}} :=\{ \Psi \in \mathcal{F}_{\mathup{BL}}: \Psi \text{ depends on only finitely many coordinates.} \}$. The weak continuity and $\d_\Exg$ continuity are equivalent on functions that depend on only finitely many coordinates, so  
 \[\liminf_{n\to\infty} \d_{\mathup{BL}}(\bC^n,\mathbb{B}^n) \ge  \int_{\cP_{\le\overline{\rho}}} \Psi(c) \ (\bC-\mathbb{B})(\d c) \quad \forall \Psi \in \cF_1.\]
 We claim that for $\Psi \in \mathcal{F}_{\mathup{BL}}$, there exist $\Psi_k \in \mathcal{F}^1_{\mathup{BL}}$  such that $\lim_{k\to \infty}\int_{\cP_{\le\overline{\rho}}} \Psi_k (c) \ (\bC-\mathbb{B})(\d c) = \int_{\cP_{\le\overline{\rho}}} \Psi (c) \ (\bC-\mathbb{B})(\d c)$. From this, we have $\liminf_{n\to\infty}\d_{\mathup{BL}}(\bC^n,\mathbb{B}^n) \ge \d_{\mathup{BL}}(\bC,\mathbb{B})$.
 
Indeed, for $\Psi\in \cF_{\mathup{BL}}$, define for $c\in \cP$, $\cF^1_{\mathup{BL}} \ni \Psi_k(c) := \Psi(c^k)$ with the truncated state $c^k:= (1-\sum_{i=1}^k c_i ,c_1,..., c_k)$.   Then $|\Psi(c) - \Psi_k(c)| = |\Psi(c)-\Psi(c^k)| \le \d_\Exg(c,c^k)$. We claim $\lim_{k\to\infty} \d_{\Exg} (c,c^k) = 0$. 
 
Since $\d_\Exg(c,c^k)= \sum_{n=1}^\infty |T_n(c^k)- T_n(c) |= \sum_{n=1}^{\infty}| H_n(c^k) - H_n(c) |$, where 
$H_n(c)= 1- T_n(c) = \sum_{i=0}^{n-1} c_i$. Hence $H_n(c^k) \to H_n(c)$ and $T_n(c^k) \to T_n(c)$  as $k\to \infty$ for each $n$. Here we will invoke~\cite[Theorem 2.8.8]{bogachev2007measure} again for the convergence of integrals. Observe that 
\[0\le |T_n(c^k)- T_n(c) | \le T_n(c^k)+ T_n(c) \]
and by summation by parts and monotone convergence 
\[
	\lim_{k\to \infty} \sum_{n=1}^\infty T_n(c^k) =  \lim_{k\to \infty} \Mom_1(c^k) = \Mom_1(c) = \sum_{n=1}^\infty T_n(c) \,.
\]
We have
\[
	\lim_{k\to\infty}\sum_{n=1 }^\infty |T_n(c^k)- T_n(c) |  =  0 \,. 
\]
This shows $\Psi(c^k) \to \Psi(c)$ for each $c\in\cP$. By dominated convergence, we have
\[
	\lim_{k\to \infty}\int_{\cP_{\le\overline{\rho}}} \Psi_k (c) \ (\bC-\mathbb{B})(\d c) = \int_{\cP_{\le\overline{\rho}}} \Psi (c) \ (\bC-\mathbb{B})(\d c)\,.
\]
We conclude $\liminf_{n\to\infty}\d_{\mathup{BL}}(\bC^n,\mathbb{B}^n) \ge \d_{\mathup{BL}}(\bC,\mathbb{B})$.
 
To show the uniform equicontinuity $\cP(\cP_{\le \rho},\d_{\Exg})$, we take $\Psi$ to be  1-$\d_{\Exg}$-Lipschitz on $\cP_{\le\overline{\rho}}$ so that we have $|\hat{\nabla}^{L}_{k,l-1} \Psi(c) |=|\Psi(c^{k,l-1})-\Psi(c)|\le \d_{\Exg}(c^{k,l-1},c)= \frac2L$ by Lemma \ref{lem:varform_d_Exg}. We have for any $0\le s \le t \le T$ by using Lemma~\ref{lem:mod_of_uni_int} the estimate
\begin{equation}\begin{split}\label{eq:ACestimate}
 \biggl\lvert\int_{\cP_{\le\overline{\rho}}} \Psi(c) (\bC^{N,L}_t -\bC^{N,L}_s) (\d c)\biggr\rvert &=\biggl\lvert\int_{s}^t\d r \int_{\cP_{\le\overline{\rho}}}\sum_{k,l \in \bN} L\hat{\nabla}^{L}_{k,l-1} \Psi(c) \bJ^{N,L}_r (\d c,k,l-1)\biggr\rvert\\ 
&\le 2\biggl\lvert\int_{s}^t\d r \int_{\cP_{\le\overline{\rho}}}\sum_{k,l \in \bN}  \bJ^{N,L}_r (\d c,k,l-1)\biggr\rvert\\
&\le 2\sup_{N/L\to \rho} \bigl\lVert  \bJ^{N,L}_r\d r \bigr\rVert_{\TV([s,t]\times \EhNL)}\le  2w(|t-s|).
\end{split}\end{equation} 
From $\lim_{t\to 0} w(|t|)=0$ in Lemma \ref{lem:mod_of_uni_int}, we conclude the equicontinuity of $( t \mapsto \bC^{N,L}_t )_{N\in\bN_0}$ in $( \cP(\cP_{\le \overline{\rho}},\d_{\Exg}),\d_{\mathup{BL}})$ from
\[
	d_{\mathup{BL}}\bra[\big]{\bC^{N,L}_s , \bC^{N,L}_t} 
 	=\sup_{\Psi\in\cF_{\mathup{BL}}} \biggl\{ \int_{\cP_{\le\overline{\rho}}} \Psi(c) \, (\bC^{N,L}_s -\bC^{N,L}_t )(\dx c)  \biggr\} \le 2w(|t-s|)
\]  
Then the refined Arzelà-Ascoli theorem \cite[Proposition 3.3.1]{ambrosio2008gradient}  implies $\bC^{N,L}$ converges   to $\bC$ on $[0,T]$ along a subsequence $\cP(\cP_{\le\overline{\rho}})$ and $t\mapsto \bC_t$ is continuous on $[0,T]$ in   $ \cP(\cP_{\le \overline{\rho}},\d_{\Exg})$.  
\end{proof}

\subsection{$\Gamma$-convergence of energies}\label{subsec:lsc-energy}

In this section, we will prove the $\Gamma$- convergence and large deviation result of Theorem \ref{thm:gameconv-energy}, via $\Gamma$-convergence of unnormalized entropy functionals over $c$ instead of $\cE$.

We will work with slightly relaxed conditions on $w$ compared to the standard assumptions in Section~\ref{subsec:ass} (see also Definition~\ref{def:eq-meas}). Moreover, it will be convenient to work with $\lambda_c=\log \fugcrit$ instead of $\fug$. Therefore, with slight abuse of notation compared to Section~\ref{subsec:ass}, we denote 
\begin{definition}
For any non-negative sequence $(w(k))_{k\in\bN_0}$
\begin{equation}
    Z(\lambda):= \sum_{k\in \bN_0} e^{\lambda k} w(k), \quad  \lambda_c := \mathop{\mathrm{arg\,sup}}_{\lambda} \left\{ Z(\lambda)<\infty \right\}, 
\end{equation}
and, for any $\lambda$ such that $Z(\lambda)<\infty$, 
\begin{equation}\label{def:general-rho-c}
   \equi^{\lambda}(k):= \frac{1}{Z(\lambda)} e^{\lambda k} w(k), \qquad \rho_c:= \sup_{\lambda<\lambda_c} \left\{ \Mom\left(\equi^{\lambda}
   \right)<\infty \right\}.
\end{equation}

Moreover, as in Definition~\ref{def:eq-meas}, let $\lambda(\rho)$ for any $\rho\leq \rho_c$ with $\rho<+\infty$ be the unique $\lambda$ such that $\Mom_1(\equi^{\lambda})=\rho$, and set $\equi^{\rho}:=\equi^{\lambda(\rho)}$.
\end{definition}
Note 
that $Z(\lambda)$ and $\Mom(\equi^{\lambda})$ are continuous functions of $\lambda$ on their domains, with $\rho_c=\infty$ if $Z(\lambda_c)=\infty$.

\begin{remark}
For $w'(k)=\gamma e^{\lambda' k} w(k)$ with $\gamma >0$ and $\lambda \in \bR$, we have $\lambda_c(w') = \lambda_c(w)- \lambda'$ and $\omega^{\lambda}= \omega'^{\lambda-\lambda'} $.  Moreover since $\omega^{\rho}=\omega^{\lambda(\rho)}=\omega'^{\lambda(\rho)-\lambda'}= \omega'^{\rho} $, the quantities  $\omega^\rho$, $\lambda_c -\lambda(\rho)$ and $\rho_c $ are invariant under this change. As a result, without loss of generality, we set $w(0)=w(1)=1$ in Definition \ref{def:eq-meas} and consider $\omega^\rho$ parametrized by the first moment.
\end{remark}
Throughout this section, we require the following assumption.
\begin{assumption}\label{assu:w}
There exists at least one $\lambda\in \mathbb{R}$ such that $Z(\lambda)<\infty$, i.e.\ $\lambda_c>-\infty$, and $w$ is not everywhere zero. 
\end{assumption}

In particular, $0<Z(\lambda)<\infty$ for any $\lambda<\lambda_c$.

\begin{remark}
Note that Assumptions \eqref{positive_weight} and \eqref{e:ass:Kc} imply Assumption \ref{assu:w}, namely $w(k)>0$ for all $k\in \bN_0$ and that $\lim_{k\to\infty} w(k)/w(k-1) \in (0,\infty)$.
\end{remark}

The central object will be the rate functional, \eqref{eq:thmgammalim}, which we recall for convenience
 
  \begin{equation}\label{eq:gamma-lim-F} F(c,\rho)= \left\{\begin{aligned}
    &\mathsf{Ent}(c|\equi^{\rho \wedge \rho_c}) +(\lambda_c-\lambda(\rho))_+ \left(\rho-\Mom(c) \right) \qquad&&  \mbox{if $\Mom(c)\leq\rho$}\\
    &+\infty \quad&&  \mbox{otherwise}.
\end{aligned}\right.\end{equation}

with the notation of $\lambda(\rho)=\log \Phi(\rho)$, $\lambda_c=\log \Phi(\rho_c) $ (compare with Definition~\ref{def:eq-meas}).

It will appear as the relaxation of a similar entropy functional under the (possibly noncontinuous) constraint $\Mom(c)=\rho$.

\subsubsection{$\Gamma$-convergence of rate functions}

At the risk of overwhelming the reader with various definitions, let us introduce 
\[ J(c,\rho):=\left\{\begin{aligned}
    &\sum_{k\in \bN_0} c_k \log \frac{c_k}{w_k} \qquad&&  \mbox{if $\Mom(c)=\rho$, $c \ll w$}\\
    &+\infty \quad&&  \mbox{otherwise}
\end{aligned}\right.\]
and
\[\bar{J}(c,\rho):=\left\{\begin{aligned}
    &\sum_{k\in \bN_0} c_k \log \frac{c_k}{w_k}+\lambda_c \left(\rho-\Mom(c) \right)\qquad&&  \mbox{if $\Mom(c) \leq \rho$, $c\ll w$}\\
    &+\infty \quad&&  \mbox{otherwise}
\end{aligned}\right.\]
for $\lambda_c <+\infty$ and in the case that $\lambda_c=+\infty$, 
\[\bar{J}(c,\rho):=\left\{\begin{aligned}
    &\sum_{k\in \bN_0} c_k \log \frac{c_k}{w_k}\qquad&&  \mbox{if $\Mom(c) =\rho$, $c\ll w$}\\
    &+\infty\quad &&  \mbox{otherwise.}
\end{aligned}\right.\]
Here $c\ll w$ indicates that $c$, seen as a measure over $\bN_0$, is absolutely continuous with respect to $w$, i.e. $w_k=0$ for $k\in \bN_0$ implies that $c_k=0$. Due to the elementary inequality
\[-(a+b)\leq \phi(a|b)+a-b =a \log \frac{a}{b} \leq \phi(a|b)+b+a, \] 
for $a,b>0$, we have for any $\lambda<\lambda_c$ with either $\lambda=0$ or $\Mom(c)<\infty$ the equivalence
\begin{equation}\label{eq:ent_tilt}
    \mathsf{Ent}(c|\equi^\lambda)+\lambda \Mom(c)-\log Z(\lambda) = \left\{\begin{aligned}
    &\sum_{k\in \bN_0} c_k \log \frac{c_k}{w_k}\qquad&&  \mbox{if $c\ll w$}\\
    &+\infty.\quad &&
    \end{aligned}\right.,
\end{equation}
since $c\ll w \iff c \ll \equi^{\lambda}$. Note that $J(c,\rho)=\bar{J}(c,\rho)$ when either the moment constraint $\Mom(c)=\rho$ is satisfied, or if $\lambda_c=+\infty$. 

\begin{theorem}[Lower semicontinuous relaxation]\label{thm:glimr}
With respect to the narrow topology, 
     \begin{equation}
         \Gamma-\lim \, J(c,\rho) = \bar{J}(c,\rho).
     \end{equation} 
In particular, $\bar{J}(c,\rho)$ is jointly lower semicontinuous and convex. 
\end{theorem}

First, let us prepare some technical estimates.

\begin{lemma}\label{lm:gmest}
There exists a subsequence $a:\mathbb{N}\to \mathbb{N}$ of $\mathbb{N}$ such that 
\begin{equation}\label{eq:Jsub}
    \lim_{k\to \infty} \frac{1}{a(k)} \log w(a(k)) \geq -\lambda_c.
\end{equation}
Moreover, 
\begin{equation}\label{eq:Jdual}
    \bar{J}(c,\rho)=\sup_{\lambda <\lambda_c,\, f\in \mathrm{B}_b} \left( \int f c(\d k) + \lambda \rho - \log \int e^{f(k)+\lambda k} w(\d k) \right).
\end{equation}
In particular, $\bar{J}(c,\rho)$ is lower semicontinuous and convex. 
\end{lemma}
\begin{proof}
Through contradiction, we will establish the desired first statement. Namely, suppose without loss of generalization that $\lambda_c<\infty$, and no admissible subsequence exists, then $\limsup_{k\to \infty} \frac{1}{k} \log w(k) <  -\lambda_c$. Therefore, there exist constants $\varepsilon>0$, $C>0$ such that for sufficiently large $k$, 
    \[w(k) \leq C e^{-(\lambda_c+\varepsilon) k}\]
    but this implies that 
    \[ \int e^{(\lambda_c +\varepsilon')k} w(\d k) < \infty\]
    for any $0<\varepsilon'<\varepsilon$, which contradicts the maximality of $\lambda_c$. 

For the second statement, denote the right-hand side of \eqref{eq:Jdual} by $\tilde{J}(c,\rho)$:
\begin{equation}\label{eq:Jdual2}
    \tilde{J}(c,\rho):=\sup_{\lambda <\lambda_c,\, f\in \mathrm{B}_b} \left( \int f c(\d k) + \lambda \rho - \log \int e^{f(k)+\lambda k} w(\d k) \right).
\end{equation}

For any fixed $\lambda,\lambda'$, with $\lambda<\lambda'<\lambda_c$ we consider the truncated sequence $(f_n)_{n\in \mathbb{N}}$ with $f_n(k):=(\lambda'-\lambda) (k\wedge n)$. Taking the limit $n\to \infty$ in \eqref{eq:Jdual2} leads to the inequality
\begin{equation}\label{eq:Jdual3}
    \tilde{J}(c,\rho)\geq (\lambda'-\lambda)\Mom(c)+\lambda \rho- \log \int e^{\lambda' k} \omega(\d k)
\end{equation}
with $\Mom(c)$ possibly infinite, leading to the relation $\Mom(c)=+\infty \implies \tilde{J}(c,\rho)=+\infty$. 

Since $\bar{J}(c,\rho)=+\infty$ if $\Mom(c)=+\infty$ we can now restrict ourselves to the case $\Mom(c)<\infty$. Recall the dual representation of the relative entropy, where for each $\lambda<\lambda_c$, 
\[\mathsf{Ent}(c|\equi^\lambda)=\sup_{f\in \mathrm{B}_b} \left( \int f c(\d k) - \log \int e^{f(k)} \equi^{\lambda}(\d k) \right).\]
Combining this with \eqref{eq:ent_tilt} we find 
\begin{align*}
    \tilde{J}(c,\rho)&=\sup_{\lambda <\lambda_c} \left( \lambda \rho+ \sup_{f\in \mathrm{B}_b} \left[\int f c(\d k)  - \log \int e^{f(k)+\lambda k} \equi^{\lambda}(\d k) \right] -\log Z(\lambda) \right)\\
    &= \sup_{\lambda <\lambda_c} \left( \lambda(\rho-\Mom(c)) +\mathsf{Ent}(c|\equi^{\lambda})+\lambda \Mom(c)-\log Z(\lambda) \right) \\
    &=\sup_{\lambda <\lambda_c} \lambda (\rho-\Mom(c))+\left\{\begin{aligned}
    &\sum_{k\in \bN_0} c_k \log \frac{c_k}{w_k}\qquad&&  \mbox{if $c\ll w$}\\
    &+\infty\quad &&
    \end{aligned}\right.\\
    &= \bar{J}(c,\rho).
\end{align*}
Convexity and lower semicontinuity follow directly from the dual formulation, since for any narrowly converging sequence $c_n$ and converging sequence $\rho_n$ we have 
\begin{align*}
\liminf_{n\to\infty} \bar {J}(c_n,\rho_n) &\geq \liminf_{n\to \infty} \left( \int f c_n(\d k) + \lambda \rho_n - \log \int e^{f(k)+\lambda k} \omega(\d k) \right)\\
    &= \int f c(\d k) + \lambda \rho - \log \int e^{f(k)+\lambda k} \omega(\d k),
\end{align*}
and the desired liminf-inequality is obtained after taking the supremum over $\lambda<\lambda_c$ and $f\in \mathrm{B}_b$.

\end{proof}

We are now in a position to prove the relaxation of $J$.


\begin{proof}[Proof of Theorem \ref {thm:glimr}]
It falls upon us to prove the following statements, namely, \emph{(i)} that for any non-negative sequence $(\rho_n)_{n\in \mathbb{N}}$ converging to some $\rho < \infty$, and any pointwise converging sequence $(c_n)_{n\in \mathbb{N}}$ to some $c$, 
\[\liminf_{n \to \infty} J(c_n,\rho_n) \geq \bar{J}(c,\rho),\]
and, \emph{(ii)}, that for any $(c,\rho)$ such that $\bar{J}(c,\rho)$ is finite there exist corresponding converging sequences $(\rho_n)_{n\in \mathbb{N}}$ and $(c_n)_{n\in \mathbb{N}}$ such that 
\[\limsup_{n \to \infty} J(c_n,\rho_n) \leq \bar{J}(c,\rho).\]
Here, \emph{(i)} follows from the lower semicontinuity of $J$ as established in Lemma \ref{lm:gmest} and the equality $J=\bar{J}$ when $J$ is finite. 

Now, for \emph{(ii)}, fix any $(c,\rho)$ such that $\bar{J}(c,\rho)<\infty$. Since in the case of $\lambda_c=\infty$ or $\Mom(c)=\rho$ we can simply take the constant sequence $c_n:=c$ and $\rho_n:=\rho$, we will assume that $\lambda_c<\infty$ and $\Mom(c)<\rho$. For large enough $n$, we construct the sequences $\rho_n:=\rho$ and 
\[c_n:=(1-\eps_n) c + \eps_n \delta_{a(n)}, \qquad \eps_n:=\frac{\rho-\Mom(c)}{a(n)-\Mom(c)},\]
where $a(n)$ is the sequence obtained in Lemma \ref{lm:gmest}. Note that $\eps_n\sim a(n)^{-1} \to 0$ as $n\to \infty$, the constraint $\Mom(c_n)=\rho$ is satisfied, and $\rho_n=\rho=(1-\eps_n)\Mom(c)+\eps_n a(n)$. Therefore, by convexity
\begin{align*}
    \limsup_{n \to \infty} J(c_n,\rho_n) &\leq \limsup_{n \to \infty} \,(1-\eps_n) J(c,\Mom(c)) + \limsup_{n \to \infty} \eps_n \, J(\delta_{a(n)},a(n)) \\
    &= J(c,\Mom(c)) + \limsup_{n \to \infty} \frac{\rho-\Mom(c)}{a(n)-\Mom(c)} \left(-\log \equi(a(n))\right)\\
    &\leq J(c,\Mom(c)) + (\rho-\Mom(c)) \lambda_c = \bar{J}(c,\rho),
\end{align*}
where the last inequality follows from \eqref{eq:Jsub}.    
\end{proof}
Throughout we will also use the version of $\bar{J}$ with reference measure $\omega^\rho$,  
\begin{equation}
    \bar{F}(c,\rho):=\bar{J}(c,\rho)-\inf_{c} \bar{J}(c,\rho),
\end{equation}
which will look slightly different depending on $\lambda_c$, $Z(\lambda_c)$ and $\rho_c$. 
\begin{lemma}\label{lm:gammanom}
The normalization constants and normalized rate functions are as follows:
\begin{enumerate}[(i)]
    \item  If $\lambda_c=+\infty$, then 
    \[\inf_{c} \bar{J}(c,\rho)=\lambda(\rho) \rho-\log Z(\lambda(\rho))=\bar{J}(\equi^{\rho},\rho),\] 
    and
    \[ \bar{F}(c,\rho)= \left\{\begin{aligned}
    &\mathsf{Ent}(c|\equi^{\rho}) \qquad&&  \mbox{if $\Mom(c)=\rho$}\\
    &+\infty\quad &&  \mbox{otherwise}.
\end{aligned}\right.\]
\item If $\lambda_c<\infty$ and $\rho_c=+\infty$ (and therefore $Z(\lambda_c)=+\infty$), then $\inf_{c} \bar{J}(c,\rho)=\bar{J}(\equi^\rho,\rho)$,  
    \[ \bar{F}(c,\rho)= \left\{\begin{aligned}
    &\mathsf{Ent}(c|\equi^{\rho}) +(\lambda_c-\lambda(\rho)) \left(\rho-\Mom(c) \right)\qquad&&  \mbox{if $\Mom(c)\leq\rho$}\\
    &+\infty \quad&&  \mbox{otherwise}.
\end{aligned}\right.\]
\item If $\lambda_c<\infty$ and $\rho_c<\infty$, then 
\[\inf_{c} \bar{J}(c,\rho) = \bar{J}(\equi^{\rho\wedge \rho_c},\rho\wedge \rho_c)\]
and
\[ \bar{F}(c,\rho)= \left\{\begin{aligned}
    &\mathsf{Ent}(c|\equi^{\rho}) +(\lambda_c-\lambda(\rho)) \left(\rho-\Mom(c) \right)\qquad&&  \mbox{if $\Mom(c)\leq\rho \leq \rho_c$}\\
     &\mathsf{Ent}(c|\equi^{\rho_c})\qquad&&  \mbox{if $\Mom(c)\leq\rho$, $\rho\geq \rho_c$}\\
    &+\infty \quad&&  \mbox{otherwise}.
\end{aligned}\right.\]
\end{enumerate}   
\end{lemma}

Clearly, $\bar{F}$ is equal to the proposed rate function of (4.7).
\begin{proof}
\begin{enumerate}[(i)]
\item If $\lambda_c=+\infty$, then $\bar{J}=J$,  and for any $c$ with $\Mom(c)=\rho$ and $J(c,\rho)<\infty$, we find 
\begin{align*}
    \bar{J}(c,\rho)&=\int \log \frac{c(k)}{\equi(k)} \, c(\d k) \\
    &=\int \log \frac{c(k)}{\equi^{\rho}(k)} \, c(\d k) + \int \log \frac{\equi^{\rho}(k)}{\equi(k)} \, c(\d k)\\
    &=\mathsf{Ent}(c|\equi^{\rho}) + \lambda(\rho) \Mom(c)-\log Z(\lambda(\rho)),
\end{align*}
where the last equality can be shown to hold even if $J(c,\rho)=+\infty$ via a reverse argument, since $\lambda(\rho)<\infty$, $\Mom(c)=\rho<\infty$, and $0<Z(\lambda(\rho))<\infty$. In particular, 
\[J(\equi^{\rho},\rho)=\lambda(\rho)\rho-\log Z(\lambda(\rho)),\]
and the desired inequality now follows from the decomposition
\[J(c,\rho)=\mathsf{Ent}(c|\equi^{\rho})+J(\equi^{\rho},\rho)\]

\item We assume $\lambda_c<\infty$ and $\rho_c=+\infty$, and consider any $c$ with $\Mom(c)\leq\rho$. A similar calculation to the one above leads to 
\begin{align*}
    J(c,\Mom(c))+\lambda_c \left(\rho-\Mom(c) \right)&=\mathsf{Ent}(c|\equi^{\rho})+\lambda(\rho)(\Mom(c)-\rho)+J(\equi^{\rho},\rho)+\lambda_c \left(\rho-\Mom(c) \right)\\
    &=\mathsf{Ent}(c|\equi^{\rho})+J\equi^{\rho}|\rho)+(\lambda_c-\lambda(\rho)) \left(\rho-\Mom(c) \right),
\end{align*}
where clearly all resulting terms are non-negative, and the final expression is minimized by $c:=\equi^{\rho}$, since in that case $\mathsf{Ent}(c|\equi^{\rho})=0$ and $(\rho-\Mom(c))=0$. 

\item  Finally, we assume $\lambda_c<\infty$ and $\rho_c<\infty$, which in particular implies $Z(\lambda_c)<\infty$. When $\rho<\rho_c$, the previous calculations already provide the desired statement, so in addition we will assume that $\rho\geq \rho_c$. Going through the same steps as before but now with $\lambda_c=\lambda(\rho_c)$ instead of $\lambda(\rho)$, we conclude the proof by verifying
\begin{align*}
    J(c,\Mom(c))+\lambda_c \left(\rho-\Mom(c) \right)&=\mathsf{Ent}(c|\equi^{\rho_c})+\lambda_c (\Mom(c)-\rho_c)+J(\equi^{\rho_c},\rho_c)+\lambda_c \left(\rho-\Mom(c) \right)\\
    &=\mathsf{Ent}(c|\equi^{\rho_c})+J(\equi^{\rho_c},\rho_c).
\end{align*}
\end{enumerate}
\end{proof}




\subsubsection{Full $\Gamma$-convergence and LDP}

In this section, we prove Theorem \ref{thm:gameconv-energy}. We consider the measure of the canonical ensemble induced by an arbitrary probability measure $\omega\in \cP$, which generalizes the Definition \ref{def:lifted-equilibrium}.
\begin{definition}
For $\omega \in \cP$, define the canonical measure by
	\begin{equation}\label{eq:def:canonical}
		\mathbbm{\Pi}^{N,L}_\omega = C^L_{\#}\pi^{N,L}_\omega,
	\end{equation}
	where 
 	$\pi^{N,L}_\omega (\eta)=\frac{1}{\cZ^{N,L}}\prod_{x=1}^L \equi(\eta_x) \in \cP(V^{N,L})$ with normalization $\cZ^{N,L}$ and $\eta \in V^{N,L}$.
\end{definition}
\begin{remark}On $V^{N,L}$, $\sum_{x=1}^L \eta_x = N$. Hence the measure $\pi^{N,L}_\omega$ is invariant under the change $\omega(k)\mapsto e^{\lambda k} \omega(k) $.
\end{remark}
\begin{lemma}\label{lm:entest} The functional in Theorem \ref{thm:gameconv-energy} has the decomposition
\begin{equation}\label{eq:cdecom}
 \frac{1}{L} \mathcal{Ent}(\bC^{N,L}|\mathbbm{\Pi}^{N,L}_\omega)= \int J(c,\Mom(c)) \bC^{N,L}(\d c)+A_{N,L}+o(1)_{N/L \to \rho, N, L\to \infty},
\end{equation} where $A_{N,L}=\frac{1}{L}\log \cZ^{N,L}$.
\end{lemma}
\subsubsection{Two counting estimates}
\begin{lemma}\label{lem:upperbdd1}For any $\bC^{N,L} \in \cP(\hat{V}^{N,L})$, we have the asymptotics
\[-\int_{\hat{V}^{N,L}} \log(\bC^{N,L})  \dx \bC^{N,L} =O(\sqrt{N}).\]
\end{lemma} 
\begin{proof}
The entropy attains its maximum when $\bC^{N,L}$ is uniformly distributed so
\begin{equation*}\begin{split}
-\int_{\hat{V}^{N,L}} \log(\bC^{N,L})  \dx \bC^{N,L}  \le \log( |\hat{V}^{N,L}|).
\end{split}\end{equation*}
The size of the set $\hat{V}^{N,L}$ can be considered as the number of distinct ways of representing the natural number $N$ as a sum of up to $L$ positive integers.  Therefore, 
 $|\hat{V}^{N,L}|$ is bounded from above by $P(N)$, the partition function in number theory, which is the number of distinct ways of representing the natural number $N$ as a sum of positive integers.  This has a well-known asymptotics \cite{HardyRamanujan18} : 
 $|\hat{V}^{N,L}|\le P(N) = O(\frac{1}{N} \exp(\pi \sqrt{\frac{2N}{3}}))$ , which provides the bound $\log|\hat{V}^{N,L}| = O(\sqrt{N})$.
\end{proof}
\begin{lemma}\label{lem:upperbdd2}
	Any $c \in \hat{V}^{N,L}$ satisfies the asymptotics
	\[
		\biggl|\log(L!) - \sum_{k=0}^N \log(c_k L)! + L\sum_{k=0}^N c_k \log(c_k)\biggr|= (\sqrt{2N}+1)(1+ \log L) \,.
	\]
\end{lemma}
\begin{proof}
With $s(c):=|\{i\in \{0,1,.., N\}: c_i>0\}|$ for $c\in \hat{V}^{N,L}$, we can estimate
\begin{equation*}
\frac{N}{L}= \sum_{i=0}^N c_i i \ge \frac1L \sum_{i=0}^N i  \ge \frac1L \sum_{i=0}^{s(c)-1} i= \frac{1}{2L} (s(c)-1)(s(c)).
\end{equation*}
Moreover, the function $s(c)$ satisfies the inequality
\begin{equation}\label{eq:estimateP} 
1\le s(c) \le \sqrt{2N}+1.
 \end{equation}
For $m\in \{1/L,2/L,\dots,1\}$, define $\Lambda^L(m):= m \sum_{l=1}^L \log (m l) - \sum_{l=1}^{m L} \log(l)$. Then, for $c\in \hat{V}^{N,L}$, we arrive at the bound
\begin{align}\label{eq:est-factorial} 
\biggl|\log L! -\sum_{i=0}^N \log (c_i L)! + L \sum_{i=0}^N c_i \log(c_i)\biggr|&= \biggl|\sum_{i=0}^N \Lambda^L(c_i) \1_{c_i >0}\bigg|
&\le  s(c) \sup_{m\in\{1/L,2/L,\dots,1\}}  |\Lambda^L(m)| \,.
\end{align}
We can write
\begin{equation*}
\Lambda^L(m)=\int_0^L m \log (m \ceil{x} ) \d x- \int_0^{m L} \log(\ceil{x}) \d x= m \int_0^L \log\bigl(\frac{m \ceil{x}}{\ceil{mx}}\bigr) \d x.
\end{equation*}
By using that $ |m \ceil{x}- \ceil{m x}|\le 1$ for any $m\in(0,1]$ and $x >0$ as well as the bound elementary bound $\frac{x-1}{x}\le \log x \le x-1$ for $x> 0$, we have 
\[\frac{-1}{m \ceil{x}} \le \frac{m \ceil{x} - \ceil{mx}}{m \ceil{x}}\le \log \frac{m\ceil{x}}{\ceil{mx}} \le \frac{m\ceil{x}-\ceil{mx}}{\ceil{mx}}\le \frac{1}{\ceil{mx}}.\]
Now, we compute the integral of the upper and the lower bound and obtain 
\begin{equation*}\begin{split}
\int_0^L \frac{1}{\ceil{mx}} \dx x &\le \frac1m+ \int_{1/m}^L \frac{1}{mc} \d x =\frac1m(1 +  \log L + \log m) \,, \\
\int_0^L \frac{1}{m\ceil{x}} \dx x &\le  \frac1m  + \int_1^L \frac{1}{mx} \d x = \frac1m (1+ \log L) \,.
\end{split}\end{equation*}
Since $m\in (0,1]$, we can estimate
\[\bigg|\int_0^L \log \frac{m\ceil{x}}{\ceil{mx}} \dx x \bigg| \le \frac1m (1+\log L) \,. \]
Then, we obtain
\begin{equation}\label{eq:estim_Lambda}
\sup_{m\in\{1/L,2/L,\dots,1\}} | \Lambda^L(m)| \le  \sup_{m\in (0,1]}  |\Lambda^L(m)| \le 1+\log L.
\end{equation} 
The conclusion follows from applying the estimates \eqref{eq:estimateP} and \eqref{eq:estim_Lambda}  to the equation \eqref{eq:est-factorial}.
\end{proof}
\begin{proof}[Proof of Lemma \ref{lm:entest}]
If $\bC^{N,L}$ is not supported on $\{c|c\ll w\}$, i.e. $c$ such that $c_i=0$ whenever $w_i=0$, then both sides of \eqref{eq:cdecom} are equal to $+\infty$. Thus, let us now assume that $\bC^{N,L}$-a.e. we have $c\ll w$. We consider
\begin{align*}
 \frac{1}{L}& \mathcal{Ent}(\bC^{N,L}|\mathbbm{\Pi}^{N,L}_\omega)- \int J(c,\Mom(c)) \bC^{N,L}(\d c) \\
 &= \int \pra[\bigg]{\frac1L \log\frac{\bC^{N,L}(c)}{\#(C^L)^{-1}(c)} -\frac1L \log  \frac{1}{\cZ^{N,L}}\prod_{i=0,w_i\neq 0}^N w_i^{c_i L}  -\sum_{k=0, w_k\neq 0}^N c_k \log \frac{c_k}{w_k}}\bC^{N,L}( \d  c)\\
 &=\int \pra[\bigg]{\frac1L \log\frac{\bC^{N,L}(c)}{\#(C^L)^{-1}(c)}-\sum_{k=0}^N c_k \log c_k}\bC^{N,L}( \d  c) +\frac1L \log  \cZ^{N,L}\\
 &=\int \pra[\bigg]{\frac1L \log \bC^{N,L}(c) - \frac1L\bra[\bigg]{ \log {L! } -\sum_{k=0}^N \log\pra*{(c_k L)!} +L\sum_{k=0}^N c_k \log c_k}} \bC^{N,L}( \d  c)+\frac1L \log  \cZ^{N,L}.
\end{align*}
 
By Lemma \ref{lem:upperbdd1} and Lemma \ref{lem:upperbdd2}, we have the upper bound
\begin{align*}
\MoveEqLeft \biggl|\frac{1}{L} \mathcal{Ent}(\bC^{N,L}|\mathbbm{\Pi}^{N,L}_\omega)  - \int J(c,\Mom(c)) \; \bC^{N,L}(\dx c)- \frac1L \log  \cZ^{N,L}\biggr| \\
&\le \frac{1}{L} \biggl| \int  \log \bC^{N,L}(c) \; \bC^{N,L}(\dx c)\biggr|
		 +\frac{1}{L} \int \biggl|\log L! - \sum_{k=0}^N \log(c_k L)! + L\sum_{k=0}^N c_k \log(c_k)\biggr| \bC^{N,L} (\d  c)\\
		 & \le O(L^{-1/2} \log L)_{N/L \to \rho, N, L\to \infty}. 
\end{align*}
\end{proof}

\subsubsection{Proof of $\Gamma$-convergence}\label{subsubsec:proof-thm1.8}
\begin{proof}[Proof of Theorem \ref{thm:gameconv-energy} ]
First, note that equicoercivity in this setting already directly follows from the moment bounds, since $c\mapsto \Mom(c)$ is non-negative, narrowly lower semicontinuous, and coercive, and
\[ \limsup_{L\to \infty, N/L\to \rho} \int \Mom(c) \,\bC^{N,L}(\d c) = \limsup_{L\to \infty, N/L\to \rho} \frac{N}{L} = \rho < \infty.  \]
We will now show `unnormalized' versions of the necessary limsup and liminf estimates: \emph{(i)} that for any converging sequence $(\bC^{N,L},\rho_L=N/L)$ to $(\bC,\rho)$, we have the liminf estimate 
\begin{equation}\label{eq:tgliminf}
    \liminf_{L\to \infty} \frac{1}{L} \mathcal{Ent}(\bC^{N,L}|\mathbbm{\Pi}^{N,L}) \geq \int \bar{J}(c,\rho) \bC(\d c) + \liminf_{L\to \infty} A_{N,L}
\end{equation}
and, \emph{(ii)}, that for any $(\bC,\rho)$, there exist corresponding converging sequences $(\bC^{N,L}, \rho_L)$ such that 
\begin{equation}\label{eq:tglimsup}
    \limsup_{L\to \infty} \frac{1}{L} \mathcal{Ent}(\bC^{N,L}|\mathbbm{\Pi}^{N,L}) \leq  \int \bar{J}(c,\rho) \bC(\d c) + \limsup_{L\to \infty} A_{N,L}.
\end{equation}
This is enough to a posteriori conclude that 
\[-\lim_{L\to \infty} A_{N,L}=\inf_{c} \bar{J}(c,\rho)\]
and hence \eqref{eq:thmgammalim} follows from Lemma \ref{lm:gammanom}. Here, the limit of $A_{N,L}$ stems from the fact that by equicoercivity and \eqref{eq:tgliminf} there exist a $\bC^*$ (which in fact will be the limit point of $\mathbbm{\Pi}^{N,L}$)
\begin{align*}
   0 \geq \int \bar{J}(c,\rho) \bC^*(\d c) + \liminf_{L\to \infty} A_{N,L},
\end{align*}
and, by non-negativity of the entropy and \eqref{eq:tglimsup}, for any $\bC$, 
\begin{align*}
   0 \leq \int \bar{J}(c,\rho) \bC(\d c) + \limsup_{L\to \infty} A_{N,L}.
\end{align*}
Together, this implies that 
\begin{equation}\label{eq:limit ANL}-\lim_{L\to \infty} A_{N,L} = \inf_{\bC} \int \bar{J}(c,\rho) \bC(\d c) = \inf_{c} \bar{J}(c,\rho). \end{equation}

Therefore, let us now consider the unnormalized liminf estimate, \eqref{eq:tgliminf}, with $L\to \infty$, $\rho_L:=N(L)/L\to \rho$, and $\bC^{N,L}\to \bC$. Recall the liminf inequality of Theorem \ref{thm:glimr}, in particular that $\bar{J}(c,\rho)$ is jointly lower semicontinuous and equal to $J(c,\rho)$ whenever $\Mom(c)=\rho$. Therefore, by Lemma \ref{lm:entest}, it follows that
\begin{align*}
     \liminf_{L\to \infty} \frac{1}{L} \mathcal{Ent}(\bC^{N,L}|\mathbbm{\Pi}^{N,L}) &\geq  \liminf_{L\to \infty} \int J(c,\Mom(c)) \bC^{N,L}(\d c)+\liminf_{L\to \infty} A_{N,L} \\
      &= \liminf_{L\to \infty} \int \bar{J}(c,\rho_L) \bC^{N,L}(\d c)+\liminf_{L\to \infty} A_{N,L}\\
    &\geq \int \bar{J}(c,\rho) \bC(\d c)+\liminf_{L\to \infty} A_{N,L}.
\end{align*}

The limsup will be obtained via a diagonal argument. By an construction similar to \cite[Theorem 3.4]{Mariani2018}, it is sufficient to only consider $\bC=\delta_{c^*}$ for some $c^*$. Therefore, let us first start with any $c$ with  $\Mom(c)=\rho$ and $c$ compactly supported, so that $\bC=\delta_{c^*}$. From Proposition \ref{prop:recovery_seq} there exist corresponding converging sequences $\bC^{N,L}=\delta_{c^{N,L}}$ and $\rho_L=\frac{N}{L}$ such that 
\[\lim_{L\to \infty} \int \bar{J}(c,\rho_L) \bC^{N,L}(\d c)= \lim_{L\to \infty} \mathsf{Ent}(c^{N,L}|\equi)  = \mathsf{Ent}(c^*|\equi)= \int \bar{J}(c,\rho) \bC(\d c).\]
For $c$ that does not have compact support with $\Mom(c)= \rho$, consider $c^n$ to be the normalized truncation of $c$ on $\{0,1,...,n\}$ so that $c^n$ has compact support. Choose $K>\rho :\omega(K)>0$, for such $K$ define $\tilde c^n = (1-\alpha^n) c^n + \alpha^n \delta_K $, where $\alpha^n =\frac{\rho - \Mom(c^n)}{K-\Mom(c^n)}>0$ and converges to $0$ as $n\to \infty$. The sequence $\tilde c^n$ converges narrowly to $c$.  
Then by convexity, 
\[\mathsf{Ent}(\tilde c^n | \omega) \le (1-\alpha_n) \mathsf{Ent}(c^n | \omega)+ \alpha_n (-\log \omega(K))\] so that 
 \[\lim_{n\to\infty}\mathsf{Ent}(\tilde c^n | \omega)  \le \lim_{n\to \infty} \mathsf{Ent}(c^n | \omega) = \mathsf{Ent}(c|\omega).\] 
 By a diagonal sequence argument, we have along a subsequence 
 \[\lim_{j\to\infty} \mathsf{Ent}(\cD^{N_j,L_j}(\tilde c^{n_j})| \omega) \le \mathsf{Ent}(c|\omega).\]
 
Finally, suppose $\bC=\delta_{c^*}$ but $\Mom(c^*)<\rho$. By Theorem \ref{thm:glimr} there exists a sequence of $(c^*_n)_{n\in \mathbb{N}}$ converging to $c^*$ with $\Mom(c^*_n)=\rho$ for all $n\in \mathbb{N}$, such that $\bar{J}(c^*_n,\rho)\to \bar{J}(c^*,\rho)$, and therefore we can construct a suitable $\bC^{N,L}$ via a diagonal argument. 
\end{proof}

\subsection{Lower semicontinuity of dissipation potentials}


In this section, we write shortly $(\bC^{N,L},\bJ^{N,L}) \to (\bC,\bJ)$ to denote the convergence in the sense of Definition~\ref{def:conv_of_pair_measure}. 
Assumption~\eqref{eq:finite_edf} and the compactness statements from Propositions~\ref{prop:uniform_bbd_C} and~\ref{prop:comp_j} justify Definition~\ref{def:conv_of_pair_measure}, since we can extract a converging subsequence in this sense. 
\begin{proposition}[Lower-semicontinuity of dissipation potentials] \label{prop:lsc-diss-potentials} 
Assume \eqref{eq:finite_edf} and assume $(\bC^{N,L},\bJ^{N,L}) \to (\bC,\bJ)$, then for almost all $t\in [0,T]$
   \[\liminf_{N/L\to \rho}  \mathcal{R}^{N,L}_t(\bC^{N,L},\bJ^{N,L})\ge \mathcal{R}_t(\bC,\bJ) \]
and \[ \liminf_{N/L\to \rho}\cD^{N,L}_t(\bC^{N,L},\bJ^{N,L}) \ge \cD_t(\bC,\bJ). \]
In particular, $\bJ_t \ll \mathbbm{\Theta}_t [\bC]$ for almost all $t\in [0,T]$.
\end{proposition}
The proof of Proposition \ref{prop:lsc-diss-potentials} relies on a Reshetnyak-type lower-semicontinuity result from~\cite[Lemma 2.3 (8)]{PRST22} by a suitable rewriting of the functionals. In the following three statements, we show appropriate convergence statements that will be used in the proof of Proposition \ref{prop:lsc-diss-potentials}.
\begin{lemma}\label{lem:vector_measure_convergence}
If $(\bC^{N,L},\bJ^{N,L}) \to (\bC,\bJ)$ then for all $k,l \in \bN$ and  all $t\in [0,T]$
\[
	(\mathbbm{\Gamma}^{N,L,k,l-1,1}_t, \mathbbm{\Gamma}^{N,L,k,l-1,2}_t)\to ( \mathbbm{\Gamma}^{k,l-1,1}_t, \mathbbm{\Gamma}^{k,l-1,2}_t) \qquad\text{in } \sigma\bigl(\cM^+(\cP_{\le \overline{\rho}}^2;\bR^{2}), C_b(\cP_{\le \overline{\rho}}^2;\bR^{2})\bigr)
\] 
and 
\begin{equation*}
(\mathbb{V}^{N,L,k,l-1}_t , \mathbbm{\Gamma}^{N,L,k,l-1,1}_t,\mathbbm{\Gamma}^{N,L, k,l-1,2}_t ) \to  (\mathbb{V}^{k,l-1}_t , \mathbbm{\Gamma}^{k,l-1,1}_t  , \mathbbm{\Gamma}^{k,l-1,2}_t)
\end{equation*}
in  $\sigma(\cM^+(\cP_{\le \overline{\rho}}^2;\bR^{3}); C_b(\cP_{\le \overline{\rho}}^2;\bR^{3}))$,
where, for the finite particle system, the measures are defined by
\begin{align*}
 \mathbbm{\Gamma}^{N,L,k,l-1,1}_t(\dx c,\dx c')&:=\delta_{c^{k,l-1}}(\dx c')\nu_t^L [\bC^{N,L}](\d c,k,l-1) \,, \\
\mathbbm{\Gamma}^{N,L,k,l-1,2}_t(\dx c,\dx c')&:= \delta_{{c'}^{l,k-1}}(\dx c){\nu^{L}_t}^\dagger[\bC^{N,L}](\d c',l,k-1) \,, \\
\mathbb{V}^{N,L,k,l-1}_t(\dx c,\dx c')&:=\delta_{c^{k,l-1}}(\dx c') \bJ^{N,L}_t(\dx c,k,l-1) \,, 
\intertext{and for the limit}
 \mathbbm{\Gamma}^{k,l-1,1}_t(\dx c,\dx c')&:=\delta_c(\dx c')\nu_t[\bC](\d c,k,l-1)\,, \\	
  \mathbbm{\Gamma}^{k,l-1,2}_t(\dx c,\dx c')&:=\delta_{c'}(\dx c)\nu^{\dagger}_t [\bC] (\d c',l,k-1) \,,\\
  \mathbb{V}^{k,l-1}_t(\dx c,\dx c')&:=\delta_{c}(\dx c') \bJ_t(\dx c,k,l-1) \,,
\end{align*}
with $\nu^L [\bC^{N,L}]$ and $\nu[\bC]$ are defined in~\eqref{eq:expected-flux-finite} and in~\eqref{eq:expected-flux-in-limit} respectively.
\end{lemma}
\begin{proof}
The convergence to $(\bC,\bJ)$   implies for all $t\in [0,T]$ the convergence of the measure 
\[
  \bJ^{N,L}_t(\d c,k,l-1) \to \bJ_t(\d c,k,l-1)  \text{ narrowly in } \cM^+ \bigl(\cP_{\le \overline{\rho}}\times \bN \times \bN_0 \bigr).
\] 
Generally, to establish the weak convergence of product measures $\mu^n=(\mu^n_i)_{i=1}^d \to (\mu_i)_{i=1}^d $, it is equivalent to show the weak convergence of $\mu^n_i \to \mu_i$ for every $i\in \{1,\dots,d\}$, that is for each $f$ bounded and \emph{Lipschitz continuous} on $(\cP_{\le \overline{\rho}}^2, \bR)$ we have the convergence 
 \[\int f(c,c') \mu_i^n (\d c,\d c') \to \int f(c,c') \mu_i (\d c,\d c').\]  Therefore, it suffices to consider the case when $\mu^{N,L}=(\mathbb{V}^{N,L,k,l-1} , \mathbbm{\Gamma}^{N,L,k,l-1,1}  , \mathbbm{\Gamma}^{N,L,k,l-1,2})(t)$ for fixed  $k ,l \in \bbN$, $t\in [0,T]$.

\smallskip\noindent
\emph{Step 1: Convergence of $\mathbb{V}^{N,L,k,l-1}\to \mathbb{V}^{k,l-1}$.}  
\begin{align}
\MoveEqLeft\sum_{(c,c') \in \hat{V}^{N,L}\times\hat{V}^{N,L}}f(c,c') \delta_{c^{k,l-1}}(c') \bJ_t^{N,L}(c,k,l-1)-\int_{\cP_{\le \overline{\rho}}\times \cP_{\le \overline{\rho}}} f(c,c') \delta_c(\dx c') \, \bJ_t(\d c,k,l-1) \notag \\
&=\sum_{c \in \hat{V}^{N,L}} f(c,c^{k,l-1})  \bJ^{N,L}_t(c,k,l-1) - \int_{\cP_{\le \overline{\rho}}} f(c,c) \, \bJ_t(\d c,k,l-1) \notag \\
&\le \sup_{c\in \hat{V}^{N,L}} \bigl| f(c,c^{k,l-1})-f(c,c)\bigr| \sum_{c \in \hat{V}^{N,L}} \big|\bJ^{N,L}_t\big| (c,k,l-1) \label{eq:cov_Vkl-1} \\
&\qquad +\biggl|\int_{\cP_{\le \overline{\rho}}} f(c,c)\bigl( \bJ^{N,L}_t - \bJ_t\bigr)(\d c,k,l-1)\biggr|. \notag 
\end{align}
 Since $\bJ_t^{N,L}\to \bJ_t $ in duality with bounded continuous function and is non-negative, we have the convergence of $\int_{\cP_{\le\overline{\rho}}} \sum_{k,l\in \bN} \bJ^{N,L}_t(\d c,k,l-1) \to \int_{\cP_{\le\overline{\rho}}}  \sum_{k,l\in \bN} \bJ_t(d c,k,l-1) $. In particular, we have
\[\sup_{N/L\to\rho} \|\bJ_t^{N,L}(\cdot,k,l-1)\|_{\TV(\hat{V}^{N,L})}< C_t .\]
 Hence the first term in \eqref{eq:cov_Vkl-1} vanishes as $L \to \infty$ due to the upper bounded 
\[ \!\sup_{c\in \cP_{\le \overline{\rho}}} \!\!\bigl| f(c,c^{k,l-1})-f(c,c)\bigr| \!\! \sum_{c \in \hat{V}^{N,L}}\!\!\!  |\bJ^{N,L}_t| (c,k,l-1)  \le  \frac{2\operatorname{Lip}(f)}{L} \!\!\sup_{N/L\to\rho} \!\!\!  \bigl\|\bJ_t^{N,L}(\cdot,k,l-1)\bigr\|_{\TV( \hat{V}^{N,L})}.\]  
Since  $c$ and $c^{k,l-1}$ have the same zeroth and first moments, the Lipschitz estimate is obtained by  Lemma~\ref{lem:varform_d_Exg}. The second term in \eqref{eq:cov_Vkl-1} vanishes by weak convergence as $c\mapsto f(c,c)$ is a bounded continuous function.

\smallskip\noindent
\emph{Step 2: Convergence of $(\mathbbm{\Gamma}^{N,L,k,l-1,1}  , \mathbbm{\Gamma}^{N,L,k,l-1,2}) \to (\mathbbm{\Gamma}^{k,l-1,1}$,$\mathbbm{\Gamma}^{k,l-1,2})$.}\\
The two measures $(\mathbbm{\Gamma}^{N,L,k,l-1,1}  , \mathbbm{\Gamma}^{N,L,k,l-1,2})$  are symmetric under the exchange of $c \leftrightarrow c^{k,l-1}$ and $\nu \leftrightarrow \nu^{\dagger}$ so we only show the convergence  of $\mathbbm{\Gamma}^{N,L,k,l-1,1}$ in details. By definition, we consider
\begin{equation*}
\begin{split}
\MoveEqLeft\sum_{c,c' \in \hat{V}^{N,L}\times \hat{V}^{N,L}} f(c,c') \mathbbm{\Gamma}^{N,L,k,l-1,1}_t(c,c') -  \int f(c,c')\, \mathbbm{\Gamma}^{k,l-1,1}_t(\d c,\d c')\\
&= \sum_{c\in \hat{V}^{N,L}} f(c,c^{k,l-1})\nu_t^L [\bC^{N,L}](\d c,k,l-1)  -\int f(c,c) \nu_t^L [\bC^{N,L}](\d c,k,l-1) \,.
\end{split}\end{equation*}
By adding zeros and the triangle inequality, we estimate
\begin{align*}|f(c,c^{k,l-1})-f(c,c) |\kappa^L_t[c](k,l-1)&\le \frac{2\operatorname{Lip}(f)}{L} c_k c_{l-1}e^{\|b\|_\infty} K(k,l-1)\\
&\le  \frac{2\operatorname{Lip}(f)}{L}e^{\|b\|_\infty}K(k,l-1) 
\end{align*}
and similarly
\begin{align*}
\bigl|f(c,c)(\kappa^L_t[c](k,l-1)- \kappa_t[c](k,l-1))\bigr| \!&\le\! \frac{|f(c,c)|}{L-1}e^{\|b\|_\infty} K(k,l-1)c_k |c_{l-1} -\delta_{k,l-1}| \\
&\le\! \frac{\|f\|_{\infty}}{L-1} e^{\|b\|_\infty} K(k,l-1)\,.
\end{align*}
We observe for fixed $(k,l-1)$, the map $\bR^{\bN_0} \ni c\mapsto \kappa_t[c](k,l-1)$ is continuous in pointwise topology on $\bR^{\bN_0}$.  We have $c\mapsto f(c,c)\kappa_t[c](k,l-1) \in C_b(\cP_{\le \overline{\rho}})$ and
\[\int f(c,c) \kappa_t[c](k,l-1) (\bC^{N,L}_t(\d c)- \bC_t(\d c))\to 0 \qquad\text{ as } {N/L} \to \infty.\]
Hence, we conclude
\[
\biggl|\int f(c,c^{k,l-1}) \kappa^L_t[c](k,l-1)\bC^{N,L}_t(\d c) -\int f(c,c) \kappa_t[c](k,l-1) \bC_t(\d c) \biggr| \le \frac{C(f,k,l-1) e^{\|b\|_\infty}}{L} + o(1) \,. 
\]
The estimate for $\mathbbm{\Gamma}^{N,L,k,l-1,2}$ has the same structure to the same estimate for $\mathbbm{\Gamma}^{N,L,k,l-1,1}$ up to exchanging $k \leftrightarrow l$ and $\nu \leftrightarrow \nu^{\dagger}$.
\end{proof}	
\begin{corollary}\label{prop:wc-fb-flux}
If $(\bC^{N,L},\bJ^{N,L}) \to (\bC,\bJ)$, then for each $k,l\in \bN$, \[\nu_t^L [\bC^{N,L}](\d c,k,l-1) \to \nu_t[\bC](\d c,k,l-1)\]  and $S_{\#}{\nu^L_t}^{\dagger} [\bC^{N,L}](\d c,k,l-1) \to \nu^{\dagger}_t [\bC](\d c,l,k-1)$ in $\sigma(\cM^+(\cP_{\le \overline{\rho}};\bR); C_b(\cP_{\le \overline{\rho}}))$ for  all $t\in[0,T]$.
\end{corollary}
\begin{proof}
Observe \[\nu_t^L [\bC^{N,L}](\d c, k, l-1) = \sum_{c' \in \hat{V}^{N,L}} \mathbbm{\Gamma}^{N,L,k,l-1,1}_t(\d c ,\d c' ).  \] 
From Proposition \ref{lem:vector_measure_convergence}, we know the convergence of $ \mathbbm{\Gamma}^{N,L,k,l-1,1}_t$ with respect to $C_b(\cP_{\le \overline{\rho}}^2)$ so if we choose the test function in $C_b(\cP_{\le \overline{\rho}};\bR^2)$ depends only on the first variable, we have $\nu_t^L [\bC^{N,L}](\d c, k, l-1)\to \nu_t[\bC](\d c,k,l-1)$ in $C_b(\cP_{\le \overline{\rho}})$. Similarly, 
\[S_{\#}{\nu^L_t}^{\dagger}[\bC^{N,L}](\d c,k,l-1) = \sum_{c' \in \hat{V}^{N,L}} \mathbbm{\Gamma}^{N,L,k,l-1,2}_t (\d c ,\d c') \] so 
$\flipL{\nu_t}[\bC^{N,L}](\d c,k,l-1) \to \nu^{\dagger}_t [\bC](\d c,l,k-1)$.
\end{proof}

\begin{corollary}\label{prop:cts-fisher-info}
 If $(\bC^{N,L},\bJ^{N,L}) \to (\bC,\bJ)$, we have for all $t\in [0,T]$, the convergence of 
\[\int \sum_{k,l\ge 1} \nu_t^L[\bC^{N,L}](\d c,k,l-1) \to \int \sum_{k,l\ge1} \nu_t [\bC](\d c, k, l-1)\] and
\[
\int \sum_{k,l \ge 1} \flipL{\nu_t}[\bC^{N,L}](\d c,k,l-1) \to \int \sum_{k,l \ge 1} \nu_t^{\dagger}[\bC](\d c, l, k-1)
\]
as $N/L \to \rho$.
\end{corollary}
\begin{proof}
The following holds for almost all $t\in [0,T]$, and we omit the time index as it is a pointwise argument.
\begin{equation*}\begin{split}
\int \sum_{k,l=1}^N \nu^L [\bC^{N,L}](\d c, k,l-1) = \int \sum_{k,l=1}^N \kappa^L[c](k,l-1) \bC^{N,L}(\d c).
\end{split}
\end{equation*}
and for the flipped measure,
\[
\int \sum_{k,l=1}^N \flipL{\nu}(\d c,k,l-1) = \int \sum_{k,l=1}^N {\kappa^L}^{\dagger}[c^{k,l-1}](l,k-1) \bC^{N,L}(c^{k,l-1}) .
 \]
The limits have  similar forms
 \begin{equation*}
\int \sum_{k,l\in \bN} \nu[\bC](\d c, k,l-1)=\int \sum_{k,l\in \bN} \kappa[c](k,l-1)\bC (\d c)
\end{equation*} 
and \[
\int \sum_{k,l\in \bN} \nu^{\dagger}[\bC](\d c, l, k-1) = \int \sum_{k,l\in \bN} \kappa^\dagger[c](l,k-1) \bC (\d c).\]
Since each term in the sum converges by Proposition \ref{prop:wc-fb-flux}, we only need to verify that the sum is bounded. In the following, we only show the convergence of the first integral, but the flipped version can be shown by the same argument, up to notational changes.
We now want to estimate
\begin{equation*}
\biggl|\int \sum_{k,l\in \bN} (\kappa^L [c](k,l-1) - \kappa [c](k,l-1) ) \bC^{N,L} (\d c) + \int \sum_{k,l\in \bN} \kappa [c](k,l-1) (\bC^{N,L} (\d c) - \bC (\d c))\biggr|.
\end{equation*}
We can estimate the first integrand by
\[\begin{split}
    \sum_{k,l\in \bN}|\kappa^L[c](k,l-1)-\kappa [c](k,l-1)|&\le \frac{1}{L-1}\sum_{k,l\in \bN} K(k,l-1) c_k | c_{l-1}-\delta_{k,l-1}|\\
    &\le \frac{1}{L-1}\sum_{k,l\in \bN} K(k,l-1) c_k c_{l-1}\\
    &\le C_K \frac{1}{L-1} \frac{N}{L} \biggl(\frac{N}{L}+1 \biggr),
\end{split} \] where $C_K := C_{\overline{K}} e^{\|b\|_{\infty}}$.

Hence, the first integral converges to zero as $L\to +\infty$.
For the second integral, by Lemma \ref{lem: kappa-C-b} $\sum_{k,l\in \bN}\kappa[c](k,l-1) \in C_b(\cP_{\le \overline{\rho}})$, then the weak convergence of $\bC^{N,L}_t \to \bC_t$ implies this term vanishes in the limit.
\end{proof}
We are now ready to prove the main statement of this section.
\begin{proof}[Proof of Proposition \ref{prop:lsc-diss-potentials}]
 The proof uses the lower-semicontinuity of the functional defined on measures $\mathcal{F}\bigl[\hat{\phi}\bigr](\bV,\mathbbm{\Gamma}^1, \mathbbm{\Gamma}^2):=\int_{E} \hat{\phi}\biggl(\frac{\d \bV}{\d \Sigma}, \frac{\d \mathbbm{\Gamma}^1 }{\d \Sigma}, \frac{\d \mathbbm{\Gamma}^2}{\d \Sigma}\biggr) \d \Sigma  $, where $\hat{\phi}(y,x_1,x_2):=\phi(y| \sqrt{x_1 \, x_2})$, and $\cF[q](\mathbbm{\Gamma}^1,\mathbbm{\Gamma}^2) :=\int_{E} q\biggl(\frac{\d \mathbbm{\Gamma}^1}{\d \Sigma}, \frac{\d \mathbbm{\Gamma}^2}{\d \Sigma} \biggr) \dx \Sigma$ for a  dominating measure $\Sigma \in \cM^+(E)$. We omit the time index $t$ in the following as it holds for almost all $t\in[0,T]$. Hence, we express the dissipation potential as 
\[\begin{split}\mathcal{R}^{N,L}(\bC^{N,L},\bJ^{N,L})
 & =  \mathcal{Ent}(\bJ^{N,L} | \mathbbm{\Theta}[\bC^{N,L}])\\
&=\sum_{k,l\in\bN} \sum_{(c,c')\in \hat{V}^{N,L}\times \hat{V}^{N,L}} \mkern-32mu \phi\Bigl(\mathbb{V}^{N,L,k,l-1}(c,c') \,\Big|\, \sqrt{\mathbbm{\Gamma}^{N,L,k,l-1,1}(c,c')} \sqrt{\mathbbm{\Gamma}^{N,L,k,l-1,2}(c,c')}\Bigr)\\
&=\sum_{k,l\in\bN} \mathcal{F}\bigl[{\hat{\phi}}\bigr](\mathbb{V}^{N,L,k,l-1},\mathbbm{\Gamma}^{N,L,k,l-1,1},\mathbbm{\Gamma}^{N,L,k,l-1,2}).
\end{split}
\]  Since $\hat{\phi}$ is jointly convex, lower-semicontinuous and positive 1-homogeneous, the induced functional are lower-semicontinuous with respect to narrow convergence \cite[Lemma 2.3 (8)]{PRST22}.  By the lower-semicontinuity of convex functionals and Fatou's lemma, we have
\[\begin{split}
\liminf_{N/L \to \rho} \sum_{k,l\in\bN} \mathcal{F}\bigl[\hat{\phi}\bigr](\mathbb{V}^{N,L,k,l-1},\mathbbm{\Gamma}^{N,L,k,l-1,1},\mathbbm{\Gamma}^{N,L,k,l-1,2})& \ge \sum_{k,l\in\bN} \mathcal{F}\bigl[\hat{\phi}\bigr]( \mathbb{V}^{k,l-1},\mathbbm{\Gamma}^{k,l-1,1}, \mathbbm{\Gamma}^{k,l-1,2})\\
&=  \mathcal{Ent}(\bJ | \mathbbm{\Theta}[\bC])= \mathcal{R}(\bC,\bJ) .
\end{split}\] By the Assumption \eqref{eq:finite_edf} and the estimate above, we have $\bJ_t \ll \mathbbm{\Theta}_t[\bC]$ for almost all $t\in[0,T]$.  
Recall 
\begin{align*}
\mathcal{D}^{N,L}(\bC^{N,L}) 
&= \mathcal{H}^2(\nu^L[\bC^{N,L}], \flipL{\nu}[\bC^{N,L}])+\frac12 \int \sum_{k,l\in\bN} (\nu^L[\bC^{N,L}]-\flipL{\nu}[\bC^{N,L}])(\d c, k, l-1).
\end{align*} Now we can use Proposition \ref{prop:cts-fisher-info} for the continuity of the second term in $\mathcal{D}^{N,L}(\bC^{N,L}) $. It remains to show the lower-semicontinuity of the Hellinger distance, which can be expressed as 
\[\begin{split} 
 \mathcal{H}^2(\nu[\bC^{N,L}], \flipL{\nu}[\bC^{N,L}])&=   \sum_{k,l\in\bN}\frac12 \sum_{c,c'\in \hat{V}^{N,L}} (\sqrt{\mathbbm{\Gamma}^{N,L,k,l-1,1}(c,c') } - \sqrt{\mathbbm{\Gamma}^{N,L,k,l-1,2}(c,c')})^2\\
&=\sum_{k,l\in \bN} \frac12 \mathcal{F}[q] (\mathbbm{\Gamma}^{N,L,k,l-1,1},\mathbbm{\Gamma}^{N,L,k,l-1,2}).
\end{split}
\]
Again, by the lower-semicontinuity of the functional and Fatou's lemma, we obtain
\[\liminf_{N/L \to \rho} \sum_{k,l\in \bN} \frac12 \mathcal{F}[q] (\mathbbm{\Gamma}^{N,L,k,l-1,1},\mathbbm{\Gamma}^{N,L,k,l-1,2}) \ge \sum_{k,l \in \bN} \frac12 \mathcal{F}[q](\mathbbm{\Gamma}^{k,l-1,1},\mathbbm{\Gamma}^{k,l-1,2})=  \mathcal{H}^2(\nu[\bC],\flip{\nu}[\bC]).\]
\end{proof}
 
\subsection{Proof of main results}
\begin{proof}[Proof of Theroem \ref{thm:EDP_con}]
\begin{enumerate}[(a)]
\item The statement follows from Theorem \ref{thm:gameconv-energy} (proven in Section~\ref{subsubsec:proof-thm1.8}) by choosing the reference measure as the equilibrium measure of the unperturbed kernel.  
\item This follows from Fatou's lemma and Proposition \ref{prop:lsc-diss-potentials}.
\item \begin{enumerate}[(i)] 
	\item Under the given assumptions, Lemma~\ref{lem:mod_of_uni_int} holds so that we can apply Lemma~\ref{lem:uni_cts_C} to get each $\bC^{N,L}_t \weakto \bC_t$ for $t\in [0,T]$ along a subsequence. Moreover Proposition \ref{prop:comp_j} holds so we can extract yet another subsequence  so that $(\bC^{N,L},\bJ^{N,L})_{N/L \to \rho}$ converges. 
	The limit $(\bC,\bJ)$ satisfying the limit continuity equation~\eqref{eq:def:CEinfty} is a consequence of Proposition~\ref{prop:limit satisfies CEinf}.  
\item Since  for each $t\in[0,T]$, the measure $\bC_t$ is a subsequence limit,  by Lemma~\ref{lem:concen_on_rho_mean} we have the claim. \qedhere
\end{enumerate}
\end{enumerate}
\end{proof}

\begin{proof}[Proof of Theorem~\ref{thm:EDP_con_EDP}]
By the $\Gamma-\liminf$ of the energy and lower semicontinuity of the dissipation potentials of \eqref{eq:Gamma-free} and  \eqref{eq:lscPotentials} in Theorem~\ref{thm:EDP_con}. 
\end{proof}

\begin{proof}[Proof of Corollary~\ref{cor:sol2sol}]
This is a consequence of the functional characterisation of solutions of the three equations established in Section~\ref{sec:var}.
\begin{enumerate}[(i)] 
\item As an EDP solution of~\eqref{eq:FKE}, we have $\cL^{N,L}(\bC^{N,L},\bJ^{N,L})=0$ and we obtain a subsequence limit $(\bC,\bJ)\in \ref{eq:def:CEinfty}(0,T)$ by Theorem \ref{thm:EDP_con}.  By the result (ii) of Corollary \ref{cor:connection-to-mean-field-equation}  and Theorem\ \ref{thm:EDP_con_EDP}, $\cL(\bC,\bJ)=0$   so the limit is an EDP solution of~\eqref{eq:Li}. 
\item By  Theorem \ref{thm:EDP_con}(c\,ii), $\bC_0\in \cP_{\!\le\rho}$. Then the existence of the path measure $\lambda$ is given by Corollary~\ref{cor:connection-to-mean-field-equation}(i). It also follows from  Corollary \ref{cor:connection-to-mean-field-equation}(iii) that $\lambda$ concentrates on EDP solutions of~\eqref{eq:MF}.
\item By the well-preparedness of the initial data, lower semicontinuity of the functionals and $\cL^{\rho\wedge\rho_c}(\bC,\bJ)=0=\cL^{N,L}(\bC^{N,L},\bJ^{N,L})$ , we have 
\begin{equation*}\begin{split}
\limsup_{N/L \to \rho} \, \cE^{N,L}(\bC^{N,L}_t) &= 
\lim_{N/L\to\rho} \cE^{N,L}(\bC^{N,L}_0) - \liminf_{N/L\to\rho} \int_0^T  \Bigl(  \calR^{N,L}_t(\bC^{N,L},\bJ^{N,L})+\calD^{N,L}_t(\bC^{N,L})  \Bigr)\d t \\
&\le  \cE^{\rho\wedge\rho_c}(\bC_0) -    \int_0^T\bigl( \cR_t(\bC , \bJ) + \cD_t(\bC) \bigr) \d t \\
 &= \cE^{\rho\wedge\rho_c}(\bC_t) \qedhere
\end{split}\end{equation*}
\end{enumerate}
\end{proof}

By Proposition \ref{prop:EDG-time:Existence} for the time-inhomogeneous kernel (or \cite[Theorem 2.12]{Schlichting2020} for the time-homogeneous kernel), if the kernel satisfies assumption \eqref{e:ass:Ku}, then the solution of~\eqref{eq:MF} is unique given an initial data $c \in \cP_{\!\rho}$ and defines a semigroup on $\cP_{\!\rho}$, that is
\begin{equation*}
	S^{\mathsf{MFE}}_t c_0 := c_t \quad\text{ with $c_t$ the unique solution to~\eqref{eq:MF} for $t\in [0,T]$.}
\end{equation*} 
This uniqueness is now transferred to the EDP solutions of~\eqref{eq:Li} in the next Lemma.
\begin{lemma}[Uniqueness of EDP solution of \eqref{eq:Li}] \label{lem:unique_Li}
Under assumptions \eqref{e:ass:K1} and \eqref{e:ass:Ku}, suppose $\bC_0 \in \cP (\cP_{\!\le\rho})$ for some $\rho\in[0,\infty)$, $\cE^{\rho\wedge \rho_c}(\bC_0)< +\infty$ and $(\bC,\bJ)\in \ref{eq:def:CEinfty}(0,T)$ is an EDP solution of~\eqref{eq:Li} then there exists a measure $\lambda$ concentrated on EDP solution of~\eqref{eq:MF}.
Then for any $B \in \bigcup_{\rho_0 \le \rho}\cB(C([0,T];\cP_{\!\rho_0}))$ the set 
\[
  B|_0 := \{c_0 \in \cP_{\!\le \rho}: (S^{\mathsf{MFE}}_t c_0)_{t\in[0,T]}\in B \}
\]
is measurable and gives rise to the representation
\begin{equation}\label{eq:RepLi}
	\lambda(B)= \int  \1_{B|_0}(c)  \ \bC_0(\d c). 
\end{equation}
In particular, given initial data $\bC_0$, the EDP solution $(\bC,\bJ)$ to~\eqref{eq:Li} is unique.
\end{lemma}
\begin{proof}
The existence of $\lambda$ is given by Corollary \ref{cor:connection-to-mean-field-equation}. We only have to argue for the representation~\eqref{eq:RepLi}.
By the disintegration of measure, \cite[Theorem 10.4.12]{bogachev2007measure}, given a probability measure~$\lambda$ on $\bigcup_{\rho_0 \le \rho} C([0,T]; \cP_{\!\rho_0})$ there exists a probability kernel $(\lambda_c)_{c\in\cP_{\!\le \rho}}$ generated by the evaluation map $e_0$ at $t=0$ such that for all $B \in \bigcup_{\rho_0\le\rho}\cB(C([0,T];\cP_{\!\rho_0}))$, $G\in \cB(\cP_{\!\le \rho})$ it holds 
\begin{equation}\label{eq:disintegration}
\lambda(B \cap e_0^{-1}(G))= \int_{G} \lambda_{\tilde{c}}(B)\, (\lambda \circ e_0^{-1})(\d \tilde{c})=\int_{G} \lambda_{\tilde{c}}(B) \, \bC_0(\d \tilde{c}) \,.
\end{equation} 
Since $\lambda$ concentrates on the EDP solution of~\eqref{eq:MF} of the form $(c,\overline\jmath[c])\in \ExgCE(0,T)$, which has also the semigroup representation $c_t = S^{\mathsf{MFE}}_t c_0$ and we have by uniqueness the identity
\[
	\lambda_{\tilde{c}} = \delta_{(S^{\mathsf{MFE}}_t \tilde c)_{t\in [0,T]}} \,.
\]
Hence, we can conclude 
\begin{equation*}\begin{split}
\lambda_{\tilde{c}}(B)&= \1_{B|_0}(\tilde{c})
=
\begin{cases} 1 \quad \text{ if } (c_t)_{t\in [0,T]}\in B \text{ and } c_t = S^{\mathsf{MFE}}_t\tilde{c} \text{ for } t\in [0,T] ; \\
0 \quad \text{otherwise.}\end{cases}
\end{split}\end{equation*} 
Since $\tilde{c}\mapsto \lambda_{\tilde c}(B)$ is measurable so is $\tilde c \mapsto  \1_{B|_0}(\tilde{c})$.
By taking $G=\cP_{\!\le \rho}$ in \eqref{eq:disintegration}, we get the representation of $\lambda$. Since $(e_t)_\#\lambda =\bC_t$, this representation implies the uniqueness of the EDP solution $(\bC,\bJ)$ for~\eqref{eq:Li} given $\bC_0$.
\end{proof}

\begin{proof}[Proof of Theorem~\ref{thm:PropaChaos}]
The assumptions of Corollary \ref{cor:sol2sol} and Lemma~\ref{lem:unique_Li} are satisfied. Therefore, by the uniqueness of the EDP solution of~\eqref{eq:Li} in Lemma \ref{lem:unique_Li}, every subsequence obtained in Corollary \ref{cor:sol2sol} must converge to the same limit $(\bC,\bJ)$. In particular, $\bC^{N,L}_t \to \bC_t$ along the entire sequence. Using the representation in Lemma \ref{lem:unique_Li} with $\bC_0=\delta_{c_0}$, we have $\bC_t =(e_t)_\# \lambda= \delta_{c_t}$.  The proof is complete with the energy limit shown in Corollary \ref{cor:sol2sol}(iii).
\end{proof}

\begin{lemma}[Example of well-prepared initial data]\label{lem:example-well-prepare}
Suppose $\tilde \omega, \omega \in \cP^1$ such that $\sup_i \bigl|\log \frac{\tilde \omega_i}{\omega_i}\bigr|<+\infty$.  Then the canonical measure defined in~\eqref{eq:def:canonical} is a recovery sequence
\[\lim_{N/L\to \rho} \frac{1}{L} \mathcal{Ent}\bigl(\mathbbm{\Pi}^{N,L}_{\tilde \omega}|\mathbbm{\Pi}^{N,L}_\omega\bigr)  =  \mathsf{Ent}(\tilde \omega^{\rho\wedge \rho_c}| \omega^{\rho\wedge \rho_c} ).\]
In particular, if $\omega$ is the equilibrium measure with respect to $\overline{K}$, for any $\tilde \omega \in \cP^1$ such that $\sup_i \bigl|\log \frac{\tilde \omega_i}{\omega_i}\bigr|<+\infty$, it holds
\[
	\lim_{N/L\to \rho} \cE^{N,L}(\mathbbm{\Pi}^{N,L}_{\tilde \omega})= \cE^{\rho\wedge \rho_c}(\delta_{\tilde \omega^{\rho\wedge \rho_c} }) \,.
\]
\end{lemma}
\begin{proof}[Proof of  Lemma \ref{lem:example-well-prepare} ]
We consider
\begin{align*}
   \frac{1}{L}& \mathcal{Ent}(\mathbbm{\Pi}^{N,L}_{\tilde \omega}|\mathbbm{\Pi}^{N,L}_\omega)=   \int \pra[\bigg]{\frac1L \log\frac{\mathbbm{\Pi}^{N,L}_{\tilde \omega}(c)}{\#(C^L)^{-1}(c)} -\frac1L \log  \frac{1}{\cZ^{N,L}}\prod_{i=0}^N \omega_i^{c_i L} }\mathbbm{\Pi}^{N,L}_{\tilde \omega}( \d  c)\\
  =&\frac{1}{L} \mathcal{Ent}(\mathbbm{\Pi}^{N,L}_{\tilde \omega}|\mathbbm{\Pi}^{N,L}_{\tilde \omega}) + \int \pra[\bigg]{\frac1L \log  \frac{1}{{\widetilde \cZ}^{N,L}}\prod_{i=0}^N \tilde{\omega}_i^{c_i L}  -\frac1L \log  \frac{1}{\cZ^{N,L}}\prod_{i=0}^N \omega_i^{c_i L}}\mathbbm{\Pi}^{N,L}_{\tilde \omega}( \d  c)\\
  =&\frac1L \log  \frac{\cZ^{N,L}}{{\widetilde \cZ}^{N,L}}   + \int \pra[\bigg]{   \sum_{i=0}^N c_i \log \frac{\tilde{\omega}_i}{ \omega_i}}\mathbbm{\Pi}^{N,L}_{\tilde \omega}( \d  c).
\end{align*}
Let  $\cP^1 \ni c \mapsto 	F(c):= \sum_{i=0}^\infty c_i \log\frac{\tilde{\omega}_i}{ \omega_i} $. For $\mathbbm{\Pi}^{N,L}_{\tilde \omega} \weakto \delta_{ \tilde \omega^{\rho \wedge \rho_c}}$, we  have
\begin{equation}\label{eq: conv of integral}\lim_{N/L\to \rho}\int F(c)   \mathbbm{\Pi}^{N,L}_{\tilde \omega}( \d  c)=F( \tilde \omega^{\rho \wedge \rho_c}).
\end{equation}
Indeed, let $F^M(c) =\sum_{i=0}^M c_i \log \frac{\tilde \omega_i}{\omega_i}$, 
then
\begin{align*}
\int (F-F^M)(c)  \mathbbm{\Pi}^{N,L}_{\tilde \omega}( \d  c) \le \frac{1}{M+1}  \Big\|\log \frac{\tilde \omega}{\omega}\Big\|_\infty  \overline{\rho} 
\end{align*}
and by dominated convergence 
$\lim_{M\to \infty}F^M(\tilde\omega^{\rho \wedge \rho_c})= F(\tilde\omega^{\rho \wedge \rho_c}) $.
Hence, by a triangle inequality argument, we have \eqref{eq: conv of integral}.
 
 By Lemma \ref{lm:gammanom} and \eqref{eq:limit ANL},  we have 
\[\lim_{N/L \to \rho}\frac{1}{L} \log \cZ^{N,L} = -\inf_c \overline{J}(c,\rho)= - \mathsf{Ent}(\omega^{\rho\wedge \rho_c}| \omega). \] In summary,
\begin{align*}
\lim_{N/L\to \rho}  \frac{1}{L} \mathcal{Ent}(\mathbbm{\Pi}^{N,L}_{\tilde \omega}|\mathbbm{\Pi}^{N,L}_\omega)
&=\mathsf{Ent}(\tilde\omega^{\rho\wedge \rho_c} |\tilde \omega) - \mathsf{Ent}(\omega^{\rho\wedge \rho_c} | \omega)+ \sum_{i=0}^\infty \tilde \omega^{\rho \wedge \rho_c}_i \log \frac{\tilde \omega_i}{\omega_i}\\
&= \mathsf{Ent}(\tilde\omega^{\rho\wedge \rho_c} |  \omega)- \mathsf{Ent}(\omega^{\rho\wedge \rho_c} | \omega)\\
&= \mathsf{Ent}(\tilde \omega^{\rho\wedge\rho_c}|\omega^{\rho\wedge\rho_c})+ \sum_{i=0}^\infty  (\tilde\omega^{\rho\wedge \rho_c}_i  - \omega^{\rho\wedge \rho_c}_i ) \log \frac{\omega^{\rho\wedge \rho_c}_i}{\omega_i} .
\end{align*}
We now justify that the second term is zero. Note that $ \frac{\omega^{\rho\wedge \rho_c}_i}{\omega_i} = \frac{Z(0)}{Z(\lambda(\rho\wedge \rho_c))}\exp{i  \big(\lambda(\rho\wedge \rho_c)\big)  }$. Then 
\begin{align*}\sum_{i=0}^\infty  &(\tilde\omega^{\rho\wedge \rho_c}  - \omega^{\rho\wedge \rho_c} ) \log \frac{\omega^{\rho\wedge \rho_c}}{\omega}
\\
&= \log \frac{Z(0)}{Z(\lambda(\rho\wedge \rho_c))} \sum_{i=0}^{\infty} (\tilde\omega^{\rho\wedge \rho_c}_i  - \omega^{\rho\wedge \rho_c}_i )+ \big(\lambda(\rho\wedge \rho_c)\big)   \sum_{i=0}^{\infty} i (\tilde\omega^{\rho\wedge \rho_c}_i  - \omega^{\rho\wedge \rho_c}_i )\\
&=0,
\end{align*}
because $\omega^{\rho\wedge \rho_c}, \tilde\omega^{\rho\wedge \rho_c} $ are probability measures with the same first moment. 
\end{proof}

\begin{proof}[Proof of Corollary \ref{cor:loss-1-mom}]
By Theorem \ref{thm:gameconv-energy}, for any arbitrary probability measure $\omega\in \cP(\bN_0)$, the empirical measure $\mathbbm{\Pi}^{N,L}_\omega \in \cP(\hat{V}^{N,L})$ converges narrowly to $\delta_{\omega^{\rho\wedge \rho_c}}$ as $N/L \to \rho$, with $\rho_c$ defined in~\eqref{def:general-rho-c}. 

Now let $\omega = \omega^{\rho_0} \in \cP_{\rho_0}(\bN_0)$ with $\rho_0 \le \rho_c$ and $\rho_0<+\infty$, the equilibrium measure of the reversible kernel  $\overline{K}$ given in Definition~\ref{def:eq-meas},
and  $g\in \bN_0 \to \bR$ be bounded. Define $\omega_g = e^g \omega Z_g^{-1}\in \cP(\bN_0)$.  If the sequence $N/L \to \rho$ for some $\rho>\rho_c$, then
 $\mathbbm{\Pi}^{N,L}_{\omega_g} \to \delta_{\omega_g^{\rho_c}}$ and the initial condition for the stochastic process $\mathbbm{\Pi}^{N,L}_{\omega_g}$   is well-prepared by Lemma \ref{lem:example-well-prepare}. Therefore, loss of mass occurs at the initial time in the limit from the process to the deterministic system. By Theorem \ref{thm:PropaChaos}, as $\bC^{N,L}_t \to \delta_{c_t}$ for each $t>0$ where $(c_t)_{t\ge 0}$ is the solution to \eqref{eq:MF} with initial data $\omega_g^{\rho_c}$, the loss of the first moment propagates for $t>0$ as the first moment is conserved along the evolution.    
\end{proof}
\appendix
 
\section{Contraction principle}

The formulation of EDF in net flux can be considered as the contraction of the formulation with unidirectional flux in the sense of the contraction principle in large deviations theory. As discussed in \cite[Chapter 10.5.2]{Hoeksema-thesis}, \cite[Corollary 1.15]{PeletierSchlichting2022}, and \cite[Appendix A]{HoeksemaTse2023}, under the detailed balance conditions ($b\equiv 0$), we have
\[\begin{split}
\cR^{N,L}_{\mathrm{net}}(\bC^{N,L},\bJ^{N,L}_{\mathrm{net}})&:=
\inf_{\bJ\in \cM^+ (\hat{E}^{N,L})} \{   \cR^{N,L}(\bC^{N,L},\bJ): \bJ^{N,L}_{\mathrm{net}}= \frac{\bJ-S_{\#}\bJ}{2}\}\\
&= \frac1{2} \sum_{(c,k,l)\in \hat{E}^{N,L}}
\mathsf{C}\Bigl( \bJ^{N,L}(c,k,l-1) \;\Big\vert\; \iota[\bC^{N,L}]{(c,k,l-1)} \Bigr)
\end{split}\]
with
\begin{gather*}
	\iota[\bC^{N,L}_t]{(c,k,l-1)}  :=\frac{1}{2}k^{N,L}_{(c,k,l-1)}\sqrt{u^{N,L}_t(c)u^{N,L}_t(c^{k,l-1})} \,,\\
	k^{N,L}_{c,k,l-1} := \mathbbm{\Pi}^{N,L}(c)\kappa^L[c](k,l-1), \qquad
	u^{N,L}_t(c):=\frac{\bC^{N,L}_t(c)}{\mathbbm{\Pi}^{N,L}(c)} \,.
\end{gather*} 
Therefore, the net flux EDF  
\begin{equation*}\begin{split}
&\cL^{N,L}_{\mathrm{net}}(\bC^{N,L}, \bJ^{N,L}_{\net}):=
\left.\cE^{N,L}(\bC^{N,L}_t)\right|_{t=0}^T
\!\\
&+\int_0^T  \Bigl[  \calR^{N,L}(\bC^{N,L}_t,{\bJ^{N,L}_{\net}}_t)+\sum_{(c,k,l)\in \hat{E}^{N,L}}\iota[\bC^{N,L}_t](c,k,l-1)\mathsf{C}^*(-
 \hat{\nabla}^L_{k,l-1}D\cE^{N,L}(\bC^{N,L}_t)(c))  \Bigr]\d t,   \end{split}
\end{equation*} 
is related to the unidirectional flux EDF via
\[
	\frac12 \cL^{N,L}_{\mathrm{net}}\bigl(\bC^{N,L},\bJ^{N,L}_{\net}\bigr)=\inf_{\bJ\in \cM^+ (\hat{E}^{N,L})} \set*{\cL^{N,L}(\bC^{N,L},\bJ): \bJ^{N,L}_{\net}= \frac{\bJ-S_{\#}\bJ}{2}}.
\]
\begin{proposition}
Given $j_{\net}$ is a family of net fluxes with discrete state space, it holds
\[
	\inf_{ j } \set*{\cR(c, j) : \frac{j- j^\dagger}{2} = j_{\net}}  = \frac12 \cR_{\net} (c, j_{\net}) \,,
\] 
where $j^{\dagger}(x,y) = j(y,x)$, the infimum is over one-way fluxes, 
\[
	\cR(c,j)= \int \phi\Bigl(\frac{j}{\theta_c}\Bigr)  \dx \theta_c  , \quad
	 \cR_{\net}(c,j_{\net})=    \int \psi\biggl(\frac{2 j_{\net}}{\theta_c}\biggr) \dx \theta_c= \int \sfC\biggl( j_\net \,\bigg|\, \frac{\theta_c}{2}\biggr), \] with $\theta_c$ a measure on edges, $\phi(s)=s\log s -s +1$, $\phi^*(s)= e^{s}-1$, $\psi^*(s) = e^{s}+ e^{-s} - 2 = \phi^*(s)+\phi^*(-s)$, $\sfC^*(s)=4(\cosh(s/2)-1)= 2 \psi^*(s/2)$.
\end{proposition}
\begin{proof}
\emph{Lower bound}\\
We show $\cR(c,j) \ge \cR_{\net}(c, j_{\net}) $ via the convex duality  of $\cR$,
\[\begin{split}
	\cR(c, j) &= \int \phi\Bigl(\frac{j}{\theta_c}\Bigr)  \dx \theta_c=
	\sup_{w: \text{ bounded}} \int w \dx j - \int \phi^*(w) \dx\theta_c
\\
&\ge \sup\bigg\{ \int w \dx j - \int \phi^*(w) \dx\theta_c  :w \text{ bounded, } w(x,y)=-w(y,x) \bigg\}.
\end{split}\]
By antisymmetry, $\int \phi^* (w) \dx\theta_c = \frac12 \int \psi^*(w) \dx\theta_c$ and $\int w \dx j =  \int w\  (\frac{\dx j - j^\dagger}{2})= \int w \dx j_{\net}$.
Hence, for $w$ bounded and antisymmetric, we have  \[
  \int w\dx j - \int \phi^*(w) \dx\theta_c =  
  \int w\dx j_{\net} - \int  \psi^*(w) \frac{\dx\theta_c}{2}.
 \]
 Moreover, we observe by convexity of $\psi$, for any $w$, we write $w= \frac{ w_{\text{sym}}+w_{\text{asym}}}{2}$, with $w_{\text{sym}}(x,y)=w(x,y) + w(y,x)$ and $w_{\text{asym}}(x,y)=w(x,y)-w(y,x)$. Since $\psi^*$ is even and convex,
 \[\int \psi^*(w_{\text{asym}}) \dx\theta_c \le \int \psi^*(w) \dx\theta_c.\]
 Thus
 \[\sup_{\substack{w \text{ bounded},  \\\text{antisymmetric}}}\bigg\{  \int w\dx j - \int \phi^*(w) \dx\theta_c \bigg\} \ge    \sup_{w \text{ bounded}} \bigg\{ \int w\dx j_{\net} - \int  \psi^*(w) \frac{\dx\theta_c}{2}\bigg\} =  \int \psi\Big(\frac{2 j_{\net}}{\theta_c}\Big) \frac{\theta_c}{2},\]  which implies the inequality is actually an equality. Again, we use the duality to retrieve $\psi$,
  \[   \sup_{w \text{ bounded}} \bigg\{ \int w\dx j_{\net} - \int  \psi^*(w) \frac{\dx\theta_c}{2}\bigg\} =  \int \psi\Big(\frac{2 j_{\net}}{\theta_c}\Big) \frac{\theta_c}{2}.\] 
\noindent
\emph{Step 2: For given $j_{\net}$, show there exists $j$ such that $\frac{j - j^\dagger}{2} = j_\net$ and $\cR(c,j) =\cR_{\net}(c,j_\net).$} By the Lagrange multipler method for the constrainted minimization of $j \mapsto \int \phi(\frac{j}{\theta_c}) \dx\frac{\theta_c}{2}$ with constraint $\frac{j - j^\dagger}{2} = j_\net$ where $j_\net$ is a fixed parameter, we find the condition $j= \exp(z) \theta_c$  for some  antisymmetric function $z: (x,y)\mapsto z(x,y)$ . Setting $g = j/\theta_c$, $h = j_\net/\theta_c$ and solving the quadratic equation  $g^2  -2 h g- 1 =0$, we find the admissible solution $g = j/\theta_c = \sqrt{h^2 +1}+h$ and $g^\dagger = j^\dagger / \theta_c =  \sqrt{h^2 +1}- h= 1/g.$

Then for this choice of $j$,  using the elementary equality $\phi(e^z)+\phi(e^{-z})=\psi(e^z - e^{-z}),$ we have,
\[\begin{split}
	\int \phi\Big(\frac{j}{\theta_c}\Big) \dx \theta_c&=\int \phi( e^z ) \dx\theta_c =  \int \Big(\phi(e^z) + \phi(e^{-z}\Big) \frac{\dx\theta_c}{2}\\
	&=  \int \psi(e^z  - e^{-z})  \frac{ \dx\theta_c}{2}=   \int \psi\Big(\frac{2 j_\net}{\theta_c}\Big) \frac{\dx\theta_c}{2}= \frac12 \cR_{\net}(c,j_\net) \,.\qedhere
\end{split}\]

 \end{proof}
\section{Existence and uniqueness of time-inhomogeneous EDG}\label{append:Exist}
\begin{proposition}\label{prop:EDG-time:Existence}
Assume \eqref{e:ass:K1} \eqref{e:ass:Ku} and $K:[0,\infty) \to \bR^{\bN \times \bN_0 }$ is measurable. For any $T>0$, there exists a unique solution of the (time-inhomogeneous) EDG. 
\end{proposition}
\begin{remark}
The solution of EDG is defined in \cite[Definition 2.2]{Schlichting2019}. The existence proof via truncation \cite[Theorem 2.4]{Schlichting2019} can be done for the time-dependent kernel $K$. Nevertheless, we provide a proof along the lines of the Picard-Lindelöf theorem for the existence and uniqueness.  
\end{remark}
\begin{proof}
\emph{Step 1: Verify the assumptions of the Banach fixed point theorem.} 

Given the initial data $\overline{c}\in \cP^1$ with $\Mom_1(\overline{c})=\rho$, 
for $M>0$, $T>0$, consider
\[
	X:=X_{M,T}(\overline{c}):=\set[\bigg]{c: [0,T]\to \bR^{\bN_0} \,\bigg\vert\, c(0) = \overline{c},   \sup_{t\in [0,T]} \sum_{k\ge 0} (k+1)|c_k(t)-\overline{c}|  \le M < +\infty. }\,,
\] 
equipped with metric $(c,d)\mapsto \d_\Exg^T(c,d):=\sup_{t \in [0,T]}\lVert\tail(c_t-d_t)\rVert_{\ell^1}$.

Observe that if $c^n \to c$ in $\d_\Exg$ for $(c^n)\subset X$, by Fatou's lemma, $\sum_{k\ge 0} (k+1)|c_k -\overline{c}|\le M$. Hence $X$ is a closed subset under the metric $\d^T_\Exg$. So it is a complete metric space under $\d^T_\Exg$.

Consider the map $F(c)(t):=\overline{c}+ \int_0^t Qc(s) \d s,$ where 
\[{Qc}_k = J_{k-1}[c]-J_k[c]=A_{k-1}[c] c_{k-1}- B_k[c] c_k - (A_k[c] c_k - B_{k+1}[c] c_{k+1}).\]

We need to show that $F$ is a self-map, i.e., $F(c): X\to X$, for small enough $T$. First, we argue that $Qc_k$ is well-defined by showing the absolute summability of terms in $Qc$. We take a sequence of non-negative numbers $g_k \ge 0$ and write
\[\begin{split}\sum_{k\ge 0} g_k {Qc}_k &=\sum_{k\ge 0} g_k \bigl(A_{k-1}[c] c_{k-1}- A_k[c]c_k\bigr) - \sum_{k\ge 0} g_k\bigl(B_k[c] c_k- B_{k+1}[c] c_{k+1}\bigr) \\
&=\sum_{k \ge 0} g_k \biggl(\sum_{l\ge 1} K(l,k-1 )c_l c_{k-1} - K(l,k)c_l c_k \biggr) \\
&\qquad - \sum_{k\ge 0} g_k \biggl(\sum_{l\ge 1}K(k,l-1)c_{l-1} c_k - K(k+1,l-1) c_{l-1} c_{k+1} \biggr) \,.
\end{split}
\]
Here we give the complete arguments for the first term; the second term will be similar. We have
\begin{align}\label{eq:1. term}
    \MoveEqLeft\sum_k g_k \biggl(\sum_{l\ge 1} K(l,k-1 )c_l c_{k-1} - K(l,k)c_l c_k \biggr) \\
    &= \sum_k g_k c_{k-1}\sum_{l\ge 1}c_l (K(l,k-1)-K(l,k)) + \sum_{k} g_k (c_{k-1}- c_k)\sum_{l\ge 1} K(l,k) c_l \notag
\end{align}
and
\begin{align}\label{eq:2.term} 
\MoveEqLeft \sum_{k\ge 0} g_k \biggl(\sum_{l\ge 1}K(k,l-1)c_{l-1} c_k - K(k+1,l-1) c_{l-1} c_{k+1} \biggr) \\
&= \sum_{k\ge 0} \mkern-1mu g_k c_{k} \mkern-1mu  \sum_{l\ge 1}\!c_{l-1} \bigl(K(k,l-1)-K(k+1,l-1)\bigr) + \!\sum_{k\ge 0} \mkern-1mu g_k (c_{k}- c_{k+1})\sum_{l\ge 1} \mkern-2mu  K(k+1,l-1) c_{l-1}. \notag
\end{align}
Using Assumption \eqref{e:ass:Ku} in \eqref{eq:1. term} to get
\[\bigg| \sum_{k\ge 0} g_k c_{k-1}\sum_{l\ge 1} c_l (K(l,k-1)-K(l,k))\bigg| \le C_K \sum_k g_k |c_{k-1}| \sum_l l |c_l|.
\]
For the second part of the first term, we perform a summation by parts. We introduce the truncation $g_k^N =0$ for $k> N$ of $g_k$ and get
\[\begin{split}
\MoveEqLeft \sum_{l\ge 1} c_l \sum_{k\ge 0}^{N}   K(l,k)  g^N_k (c_{k-1}-c_k) \\
&=\sum_l c_l \sum_k^{N-1} c_k \Bigl((g^N_{k+1}-g^N_k)K(l,k+1) + g^N_k \bigl(K(l,k+1)-K(l,k)\bigr)\Bigr)- \sum_l c_l g^N_{N}K(l,N)c_{N}.
\end{split}\]
Similarly, we get for the second term
\[\begin{split}
 \MoveEqLeft \sum_{l\ge 1} c_{l-1} \sum_{k\ge 0}^{N}   K(k+1,l-1)  g^N_k (c_{k}-c_{k+1}) \\
&=\sum_l c_{l-1} \sum_k^{N} c_k \Bigl((g^N_{k}-g^N_{k-1})K(k+1,l-1) + g^N_{k-1} \bigl(K(k+1,l-1)-K(k,l-1)\bigr)\Bigr)\\
&\qquad - \sum_l c_{l-1} g^N_{N}K(N+1,l-1)c_{N+1}.
\end{split}\]
The double sum can be bounded by the assumptions on the kernel
\[\begin{split}
\MoveEqLeft \sum_l c_l \sum_{k\ge 0}^{N-1} c_k \Bigl((g^N_{k+1}-g^N_k)K(l,k+1) + g^N_k \bigr(K(l,k+1)-K(l,k)\bigr)\Bigr)\\
& \le C_K \sum_l  |c_l| l \sum_{k=0}^{N-1} |c_k| k |g^N_{k+1}-g^N_k| + C_K\sum_{l} |c_l|l \sum_k^{N-1} |c_k| g^N_k
\end{split}\] 
and a similar bound for
\[\begin{split}
	\MoveEqLeft \sum_l c_{l-1} \sum_k^{N} c_k \Bigl((g^N_{k}-g^N_{k-1})K(k+1,l-1) + g^N_{k-1} \bigl(K(k+1,l-1)-K(k,l-1)\bigr)\Bigr)\\
&\le C_K \sum_{l} |c_{l-1}| l \sum_k^N |c_k| (k+1) |g^N_k-g^N_{k-1}| + C_K \sum_{l} |c_{l-1}| (l-1)  \sum_k^N  |c_k| g^N_{k-1}.
\end{split}\]
Now, we set $g_k^N := k $ for $k\le N/2$ and $g_{k+1}^N - g_k^N = -1 $ for $N/2 < k < N$ so that $g_N^N=0$. Thus, the boundary term vanishes. Moreover, 
\[\sum_l c_l \sum_{k}^{N-1} c_k \Bigl((g_{k+1}-g_k)K(l,k+1) + g_k \bigl(K(l,k+1)-K(l,k)\bigr)\Bigr) \le 2 C_K \sum_l |c_l| l \sum_k^{N-1} |c_k| k. \]
In summary, we have
\[\sum_k^N g_k^N |Qc_k| \le 4 C_K ( (\Mom_1(|c|)^2+ \Mom_1(|c|)) \le 8 C_K (M^2 +M +\rho^2+\rho). \]
for every $N$.
The upper bound is uniform in $N$, so we can take $N$ to $\infty$ to get a bound for $\sum_k k |Q c_k|$.
Moreover, we can take $g_k \equiv 1$ in \eqref{eq:1. term} and \eqref{eq:2.term}.
Then, we use the assumptions of the kernel directly to see that they are also absolutely summable
\[
	\sum_k |Qc_k| \le 6 C_K (\Mom_0(|c|)\Mom_1(|c|))\le 6 C_K(M+\rho)(M+1) \,.
\]
By summing the two upper bounds, we get the estimate
\[ \begin{split}\sup_{t\in[0,T]}\sum_k (k+1) |F_k(c)(t) - \overline{c}_k| 
	&=\sup_{t\in[0,T]}\sum_k (k+1) \biggl|\int_0^t Qc_k(s) \dx s\biggr|\\&
\le \sup_{t\in [ 0,T]} \int_0^t \sum_k (k+1)|Qc_k (s) \dx s| \\
& \le   14 T C_K \bigl((M+\rho)(M+1)+\rho(\rho+1)\bigr)  < +\infty \,.
\end{split}\] 
Hence, if we choose $T\le 14 C_K \frac{M}{(M+\rho)(M+1)+(\rho+1)\rho}$, then $F(c)\in X$.

It remains to show that $F$ is a contraction on $X_{M,T}(\overline{c})$ for $T$ small enough.
Let if $c,d\in X$, for any $t\in [0,T]$, the sums are absolutely summable so
\[\begin{split}\d_\Exg(F(c(t),F(d(t) )&= \lVert \tail(F(c(t)))-\tail(F(d(t)))\rVert_{\ell^1} \\
&=\norm*{\tail\bra*{\int_0^t \bigl( Qc(s) -Qd(s)\bigr) \dx s}}_{\ell^1}\\
&\le 4 C_K M \int_0^t \sum_{m\ge 1} |E_m(s)| \dx s\\
&\le  4 C_K M T  \sup_{s\in [0,T]} \d_{\Exg}(c(s),d(s)).
\end{split}\]
Hence for 
\begin{equation}
\label{eq:T bound}T < \min \bigg\{\frac1{4 C_K M}, 14C_K \frac{M}{(M+\rho)(M+1)+(\rho+1)\rho}\bigg\},
\end{equation}
we have a contraction and can apply Banach fixed point theorem to have a unique fixed point $c\in X_{M,T}(\overline{c}): F(c(t)) = c$.
Note that the upper bound of $T$ is uniform for all initial data $\overline{c}\in \calP_\rho$.

\smallskip
\noindent
\emph{Step 2: Properties of the fixed point.}  
We need to show $c\in \cP^1$ with $\Mom_1(c)=\Mom_1(\overline{c})=\rho$.
First, we consider the positivity of $c$. This argument is identical to \cite{Esenturk2018}. 
Since $c$ is a fixed point which means the equality  \[c_k(t) = \overline{c}_k + \int_0^t Qc_k(s) \d s\] holds. This implies continuity of $c$ and 
\begin{equation}
    \label{eq:balance-eq} \frac{\d}{\d t} c_k (t) + A_k[c] c_k + B_k[c] c_k = A_{k-1}[c] c_{k-1} + B_{k+1}[c]c_{k+1}.
\end{equation}  
We derive a contradiction by considering the first time $t_0$  some coordinate of $c$ hits zero and becomes negative afterward. We denote that coordinate $k\in\bN_0$. We have $c_k(t_0)=0$ and $\frac{\d}{\d t}  c_k(t) < 0$. But this contradicts~\eqref{eq:balance-eq} because at $t_0$ the right side of \eqref{eq:balance-eq} is non-negative and the left side of~\eqref{eq:balance-eq} is negative.

As a result, such a time $t_0$ does not exist, and  $c(0)\ge 0$ implies $c(t)\ge 0$ up to $T$. We may write compactly
\[Qc_m(s) = \sum_{k,l\ge 1}\kappa_s[c](k,l-1)\gamma_{m}^{k,l-1}\] with $\gamma_m^{k,l-1}=-\delta_{k,m}-\delta_{l-1,m}+\delta_{k-1,m}+\delta_{l,m}$.
By the absolute summable bound from Step 1, we could change the order of summation.  Hence 
\[\sum_{m\ge 0} Qc_m(s) = \sum_{k,l\ge 1} \kappa_s[c](k,l-1) \sum_{m\ge 0} \gamma_m^{k,l-1} =0\]
and 
\[\sum_{m\ge 1} m Qc_m(s) = \sum_{k,l\ge 1} \kappa_s[c](k,l-1) \sum_{m\ge 1} m \gamma_m^{k,l-1} = \sum_{k,l\ge 1} \kappa_s[c](k,l-1)(-k-(l-1)+k-1+l) = 0 \,,
\]
which shows the conservation of the zeroth and first moment.
Hence, it follows $c(t) \in \cP^1$ and $\Mom_1(c(t)) = \Mom_1(\overline{c})$ for $t\in[0,T]$.  

\smallskip
\noindent
\emph{Step 3: Iterations of local existence.}
Since the bound \eqref{eq:T bound} for existence time $T$  is uniform for initial data with the same first moment and a fixed $M$, we can iterate the existence and uniqueness proof on $X_{M,T}$ with the next initial data generated by the unique fixed point from the previous step to get the solution of EDG for any positive time.
\end{proof}

\section{Explicit recovery sequence for the energy functional}\label{subsec:recovery-seq} 
For the construction of the recovery sequence, we need a suitable discretisation in $\hat V^{N,L}$ for a given $c\in \cP_{\!\rho}$. We do this via a projection in $\ell^{1,1}$, which might be non-unique.  Hence, we need to ensure that we find a measurable selection among the closest points. This measurable map will be used to define a recovery sequence via pushforward.  With the continuity of relative entropy in $\ell^{1,1}$, we conclude the convergence of energies along the discrete $\ell^{1,1}$-approximation empirical measures.

\begin{lemma}[Existence of Borel selection]\label{lem:borelselection} Let $0<\rho\le \overline{\rho}<+\infty$. Consider the complete separable metric space $(\cP_{\le \overline{\rho}},\ell^{1,1})$ defined in~\eqref{def:ell1,1}, define the set-valued map $\cD^{N,L}: \cP_{ \rho} \to 2^{\hat{V}^{N,L}}$ as 
\[
 	\cD^{N,L}(c) := \operatorname{arg\,min}\{\| c^{N,L}-c \|_{\ell^{1,1}} \colon c^{N,L}\in \hat{V}^{N,L}  \} \,,
\] 
where $\hat{V}^{N,L}$ is considered as a finite subset of $(\cP_{\le \overline{\rho}},\ell^{1,1})$.

Then there exists a measurable map  $\scrD^{N,L}: \cP_{\!\rho}\to \hat{V}^{N,L}$ satisfying $\scrD^{N,L}(c)\in \cD^{N,L}(c)$ for all~$c \in \cP_{\!\rho}$.
\end{lemma}
\begin{proof}
Let us assume for the moment that for any subset $U\subset \cP_{\le\overline{\rho}}$, the set defined by
\[\hat{\cD}^{N,L}(U):=\{ c\in \cP_{\!\rho} : \cD^{N,L}(c) \cap U \neq \emptyset \} \] 
is closed. In particular, $\hat{\cD}^{N,L}(U)$ is Borel measurable for any open subset $U \subset \cP_{ \le\overline{\rho}}$. Then by Kuratowski and Ryll-Nardzewski measurable selection theorem \cite[Theorem 6.9.3]{bogachev2007measure}, there exists a measurable map $\scrD^{N,L}: \cP_{\!\rho}\to \hat{V}^{N,L}$ and it satisfies $\scrD^{N,L}(c)\in \cD^{N,L}(c)$ for all $c \in \cP_{\!\rho}$.

Hence, it remains to show that $\hat{\cD}^{N,L}(U)$ is $(\cP_{\!\rho} ,\ell^{1,1})$ closed.
For $\hat{c} \in \hat{V}^{N,L}$, define the associated Voronoi cell in the $(\cP_{\!\rho},\ell^{1,1})$ metric space by
\[
	\VorCel^{N,L}(\hat{c}):=\Bigl\{c \in \cP_{\!\rho}\colon \|c- \hat{c}\|_{\ell^{1,1}}=\min_{\tilde c\in \hat{V}^{N,L}}\| c-\tilde c\|_{\ell^{1,1}} \Bigr\}.
\]
We note that $c \mapsto \| c - \tilde c\|_{\ell^{1,1}}$ is $\ell^{1,1}$-continuous for any fixed $\tilde c$ therefore $c\mapsto \|c- \hat{c}\|_{\ell^{1,1}}-\min_{\tilde c\in \hat{V}^{N,L}}\| c-\tilde c\|_{\ell^{1,1}}$ is also $\ell^{1,1}$-continuous. 
Since $\VorCel^{N,L}(\hat{c})$ is the preimage of the closed subset $(-\infty,0]\subset \bbR$ under a continuous map, $\VorCel^{N,L}(\hat{c})$ is closed in $(\cP_{\!\rho},\ell^{1,1})$. Observe that $\cP_{\!\rho}= \bigcup_{\hat{c} \in \hat{V}^{N,L}}\VorCel^{N,L}(\hat{c})$ so that 
\begin{equation*}\begin{split}
\hat{\cD}^{N,L}(U ) 
	&= \bigcup_{\hat{c} \in \hat{V}^{N,L}}\bigl\{ c\in\VorCel^{N,L}(\hat{c}) : \cD^{N,L}(c) \cap U \neq \emptyset \bigr\} \\
	&= \bigcup_{\hat{c} \in \hat{V}^{N,L}}\bigl\{ c\in\VorCel^{N,L}(\hat{c}) : \{\hat{c}\} \cap U \neq \emptyset \bigr\} \\
	&= \bigcup_{\hat{c} \in \hat{V}^{N,L}\cap U} \VorCel^{N,L}(\hat{c}) \,.
\end{split}\end{equation*}
Since, each $\VorCel^{N,L}(\hat{c})$ for $\hat c \in \hat V^{N,L}$ is $\ell^{1,1}$ closed, we get that $\hat{\cD}^{N,L}(U)$ is $(\cP_{\!\rho} ,\ell^{1,1})$ closed because it is a finite union of closed set.
%
%
\end{proof}
 
\begin{lemma}\label{lem:VNLdensePrho}
	Given two sequences $(N_n)_{n\in\bN}$ $(L_n)_{n\in \bN}$ with $N_n,L_n \to \infty$ and $N_n/L_n \to \rho\in [0,\infty)$ as $n\to \infty$, then $ \cP_{\!\rho} \subset \overline{\bigcup_{n\in \bN}\hat{V}^{N_n,L_n}}^{\ell^{1,1}}$.
	
Moreover, if $c\in \cP_\rho$ with compact support, it is possible to construct the sequence $(c^n\in \hat{V}^{N_n,L_n})_n$ that is supported on a compact set. 
\end{lemma}
\begin{proof}
 For $\rho=0$, there is only one element $c\in \cP_0$ given by $c_0 =1$ and $c_k=0$ for $k\geq 1$. In this case $\frac{N_n}{L_n}\to 0$, so for $n$ large, $N_n < L_n$. We can use the approximation $c^n= \frac{L_n-N_n}{L_n} \delta_0 + \frac{N_n}{L_n} \delta_{1}$, which provides the approximation in $\ell^{1,1}$, since
 \[
   \norm{c^n - c}_{\ell^{1,1}} = 3\frac{N_n}{L_n} \to 3 \rho =0 \qquad\text{ as } n\to \infty.
 \]
 Hence, let now $\rho>0$. Fix $c \in \cP_{\!\rho}$, define $\gamma(c):=\sum_{l\geq 1} c_l=1-c_0 \in (0,1]$ and choose $\eps \in (0,\gamma\wedge \rho/3)$. For technical reasons, which will become clear later, we approximate the measure $c^\eps \in \cP_{\rho-\eps}$ given by 
 \begin{equation}\label{eq:def:ceps}
   c^\eps_0 = c_0 + \frac{\eps}{\gamma} \sum_{k\geq 1} \frac{c_k}{k} \qquad\text{and}\qquad c^\eps_k = c_k \bra*{ 1 - \frac{\eps}{\gamma k}} \text{ for } k\geq 1 \,.
\end{equation}
Indeed, it holds that $c^\eps_k \in [0,1)$ for each $k\in\bN_0$, so $c^\eps$ is a probability measure. Also, the first moment of $c^\eps$ is given by $\sum_{k\geq 1} k c_k^\eps = \sum_{k\geq 1} k c_k - \gamma^{-1} \eps \sum_{k\geq 1} c_k = \rho-\eps$. 

We can use $c^\eps$ instead of $c$ for the approximation, because it satisfies the closeness estimate
\[
	\norm{c^\eps - c}_{\ell^{1,1}} = \frac{\eps}{\gamma} \sum_{k=1}^\infty c_k (1+ \frac{2}{k}) \leq 3\eps\frac{1}{\gamma}\sum_{k=1}^\infty c_k  = 3 \eps. 
\]
Now, we choose $\bN \ni M_\eps> \rho$ s.t.\@ $\sum_{k> M_\eps} k c_k^\eps \leq \eps$ or equivalently
\begin{equation}\label{eq:choice:Meps}
 	\sum_{k=1}^{M_\eps} k c_k^\eps \geq  \rho - 2\eps . 
\end{equation}
In particular, we define the truncated measure supported on $\{0,\dots,M_\eps\}$
\begin{equation}\label{eq:def:cepsMeps}
 	c^{\eps,M_\eps}|_{1,\dots,M_\eps}:= c^\eps|_{1,\dots,M_\eps} \qquad\text{and}\qquad c^{\eps,M_\eps}_0 := c^\eps_0 + \sum_{k>M_\eps} c^\eps_k ,
\end{equation}
which satisfies the bound 
\[
	\norm[\big]{c^\eps - c^{\eps,M_\eps}}_{\ell^{1,1}} = \sum_{k>M_\eps} c^\eps_k + \sum_{k>M_{\eps}} (1+k) c_k^\eps \leq 3\eps.
\] 
Hence, it is sufficient to approximate $c^{\eps,M_\eps}\in \cP_{\tilde \rho^\eps}$ for some $\tilde\rho^\eps \leq \rho -\eps$.
We note that~\eqref{eq:def:ceps} and~\eqref{eq:def:cepsMeps} imply that $c_0^{\eps,M_\eps}>0$ possibly depending on $\eps$. 
This allows us to choose $n$ large enough such that
\begin{equation}\label{eq:Ln:choices}
	\frac{1}{L_n} \leq \frac{\eps}{\sum_{k=1}^{M_\eps} k}  = \frac{\eps}{\tfrac12 M_\eps(M_\eps+1)}
 	\qquad\text{and}\qquad
 	\frac{1}{L_n} \leq c_0^{\eps,M_\eps}
\end{equation}
as well as 
\begin{equation*}
 	\abs*{\frac{N_n}{L_n}-\rho}\leq \eps
 	\qquad\text{and}\qquad 
 	N_n\geq M_\eps \,.
\end{equation*} 
Next, we define the approximating measure supported on $\{0,\dots,M_\eps\}$ by
\[
    \hat c^{\eps,n}_k := \frac{[L_n c_k^{\eps,M_\eps}]}{L_n} \quad\text{ for } 1\leq k\leq M_\eps  \quad \text{and} \quad  \hat c^{\eps,n}_0 := 1-  \sum_{k=1}^{M_\eps} \frac{[L_n c_k^{\eps,M_\eps}]}{L_n}.
\] 
Note that $\hat{c}^{\eps,n}_0 \in \frac{\bN_0}{L_n}$  since $\hat{c}^{\eps,n}_k \in \frac{\bN_0}{L_n}$ for $k\ge 1$.
We have the approximation property by our choice of $L_n$ from the bound
\begin{align*}
	\norm[\big]{\hat c^{\eps,n}-c^{\eps,M_\eps}}_{\ell^{1,1}} &= \sum_{k=1}^{M_\eps} (k+2) \frac{L_n c_k - [L_n c_k]}{L_n} 
 	\leq \frac{1}{L_n} \sum_{k=1}^{M_\eps} (k+2)  \\
	&\leq \frac{\eps}{\tfrac12 M_\eps(M_\eps+1)} \tfrac12(M_\eps+5)M_\eps \le 3\eps. 
\end{align*}
By construction, we obtain $\hat c^{\eps,n}\in \hat V^{N^{\eps,n},L_n}$, where
\[
	N^{\eps,n} := \sum_{k=1}^{M_\eps} k [L_n c_k^{\eps,M_\eps}].
\] 
If $N^{\eps,n}=N_n$,
we constructed $\hat{c}^{\eps,n} \in \hat{V}^{N_n,L_n} \cap B_{9 \eps}(c)$.
If $N^{\eps,n}\neq N_n$, we need to correct the number of particles.
 Together with the properties $c^\eps \in \cP_{\rho -\eps}$, the choices of parameters~\eqref{eq:choice:Meps} and~\eqref{eq:Ln:choices}, we have the chain of inequalities
 \begin{equation}\label{eq:Nn:ineqchain}
    \rho - 3\eps \leq \frac{N^{\eps,n}}{L_n} =  \sum_{k=1}^{M_\eps} k c_k^\eps - \sum_{k=1}^{M_\eps} k \frac{ L_n c_k^\eps -[L_n c_k^\eps]}{L_n} \leq \rho -\eps \leq \frac{N_n}{L_n} \leq \rho + \eps,
	\end{equation}
	In particular, $N_n > N^{\eps,n}$ in this case. 
	Since, we have $\hat{c}_0^{\eps,n}\geq L_n^{-1}$ by our choice of $L_n$ in~\eqref{eq:Ln:choices}, we can define
 \[
    \overline c^{\eps,n} = \hat c^{\eps,n} - \delta_0 L_n^{-1} + \delta_{N_n-N^{\eps,n}} L_n^{-1} ,
 \]
 which by construction satisfies $\overline c^{\eps,n} \in \hat V^{N_n,L_n}$.The error in $\ell^{1,1}$ of this construction step is
 \[
    \norm{\overline c^{\eps,n} - \hat c^{\eps,n}}_{\ell^{1,1}} = L_n^{-1} + \frac{N_n-N^{\eps,n}+1}{L_n} \le 6\eps,
 \] 
 where we used $\frac{N_n - N^{\eps,n}}{L_n} \leq 4 \eps$ from~\eqref{eq:Nn:ineqchain} and hence we found $\overline{c}^{\eps,n} \in \hat{V}^{N_n,L_n} \cap B_{15 \eps}(c)$. 
 Since $\eps >0$ is arbitrary, this shows the first claim. 
 
Furthermore, in the case of $c$ has compact support, in the above construction, we can choose $M:=| \supp c |$ and at the end we choose $\epsilon< \frac{1}{8 L_n}$ so that $N_n - N^{\epsilon,n} \le 4 \epsilon L_n \le \frac12 $, then $\overline{c}^n := \overline{c}^{\epsilon = (8L_n)^{-1},L_n}$ is supported on $\supp c \cup \{0\}$ as $N_n-N^{\epsilon,n}\in\bN_0$.
\end{proof}
 
\begin{proposition}[Recovery sequence]\label{prop:recovery_seq}
 Let $0<\rho\le \overline{\rho}<+\infty$. Given $\bC^\rho$ supported on $\{ c\in \cP_{\!\rho}: c \text{ supported on } K\}$, where $K$ is a finite subset of $\bN_0$, there exists a sequence $\big(\bC^{N,L}\in \cP(\hat{V}^{N,L})\big)_{N/L \to \rho}$ such that 
$\bC^{N,L} \to \bC^\rho$ in duality with $C_b(\cP_{\le\overline{\rho}},\ell^{1,1})$ and 
\[
\lim_{\substack{N,L\to +\infty\\ N/L \to \rho }}\frac{1}{L} \mathcal{Ent}(\bC^{N,L}|\mathbbm{\Pi}^{N,L}_\omega) = \int \bar{F}(c,\rho)  \, \bC(\d c). \] 
In particular, $\scrD^{N,L}_{\#}\bC^\rho  \weakto  \bC^\rho$.
\end{proposition}
\begin{proof} 
	We have both $(\cP_{\!\rho},~\text{narrow})$ and $(\cP_{\!\rho}, \ell^{1,1})$ are Polish spaces such that  $\cB(\cP_{\!\le\overline{\rho}},\ell^{1,1})\subset\cB(\cP_{\!\le\overline{\rho}},~\text{narrow})  $. By a corollary of the Lusin-Souslin Theorem \cite[Theorem 15.1, Exercise 15.4]{kechris2012classical}, we have $\cB(\cP_{\!\le\overline{\rho}},~\text{narrow})= \cB(\cP_{\!\le\overline{\rho}},\ell^{1,1})$ so we do not have to distinguish the Borel measures in $\cP((\cP_{\!\le\overline{\rho}},\ell^{1,1}))$ and $\cP((\cP_{\!\le\overline{\rho}},~\text{narrow}))$.
 
We obtain a $\ell^{1,1}$-measurable map $\scrD^{N,L}:\cP_{\!\rho} \to \hat{V}^{N,L}$ from Lemma \ref{lem:borelselection}. 
This allows us to define the pushforward measure $\scrD^{N,L}_{\#}\bC^\rho \in \cP(\hat{V^{N,L}})$. 

By the density of $\bigcup_{N/L \to \rho}\hat{V}^{N,L}$ in $(\cP_{\!\rho},\ell^{1,1})$ proven in Lemma~\ref{lem:VNLdensePrho}, the sequence $(\scrD^{N,L}(c) \in \hat{V}^{N,L})_{N/L\to \rho}$ converges to $c \in \cP_{\!\rho}$ in $\ell^{1,1}$. In Lemma~\ref{lem:VNLdensePrho}, the appoxrimating sequence of $c$ with compact support has support of $\supp(c) \cup \{0\}$, by a contradiction argument, it can be seen that the minimizing sequence $\cD^{N,L}(c)$ in $\ell^{1,1}$ also has to have support on $\supp(c) \cup \{0\}$.

Let $f \in C_b(\cP_{\le \overline{\rho}}, \ell^{1,1})$. Since $\scrD^{N,L}(c) \to c$ in $\ell^{1,1}$ for each $c$, $f(\scrD^{N,L}(c)) \to f(c)$ as $N/L \to \rho$. By dominated convergence, we have 
\[\lim_{N/L \to \rho} \pra*{ \int f(c) \scrD^{N,L}_{\#}\bC^\rho(\d c) - \int f(c) \, \bC^\rho(\d c) }
= \lim_{N/L\to\rho} \int \bigl( f(\scrD^{N,L}(c)) - f(c)  \bigr) \, \bC^\rho(\d c)= 0. \] 
This means $\scrD^{N,L}_{\#}\bC^\rho \to \bC^\rho$ in $\sigma(\cP(\cP_{\le \overline{\rho}}), C_b(\cP_{\le \overline{\rho}},\ell^{1,1}))$.
 
Now the map $c\mapsto \mathsf{Ent}(c | \omega)$ on $\{ c\in \cP_{\!\rho}: c \text{ supported on } K.\}$ is clearly $\ell^{1,1}$ continuous. Hence 
 \[\lim_{N/L\to\rho}\int\mathsf{Ent}(c | \omega) \scrD^{N,L}_{\#}\bC^\rho(\d c)  =\int \mathsf{Ent}(c| \omega) \bC^\rho(\d c).\]
By the estimates Lemma \ref{lm:entest} and Equation \eqref{eq:limit ANL}, we have that \[
\lim_{\substack{N,L\to +\infty\\ N/L \to \rho }} \frac{1}{L} \mathcal{Ent}(\bC^{N,L}|\mathbbm{\Pi}^{N,L}_\omega) = \int \bar{F}(c,\rho)  \, \bC(\d c) . \] 
Since $C_b(\cP_{\le \overline{\rho}},~\text{narrow}) \subset C_b(\cP_{\le \overline{\rho}}, \ell^{1,1})$, this implies  $\scrD^{N,L}_{\#}\bC^\rho  \weakto  \bC^\rho$.
\end{proof}
\bibliographystyle{abbrv}
\bibliography{bib.bib}
	
\end{document}